\documentclass[12pt,reqno]{amsart}
\usepackage{amsmath}
\usepackage{amssymb}
\usepackage{amstext}
\usepackage{mathrsfs}
\usepackage{a4wide}
\usepackage{verbatim}
\usepackage{bm}
\usepackage{graphicx}
\usepackage{tikz}
\usepackage{pgfplots}
\usepackage{titletoc}
\usepackage[nocompress]{cite}

\allowdisplaybreaks 
\usepackage{color}
\usepackage{cases}
\usepackage{hyperref}
\hypersetup{
	colorlinks=true,
	linkcolor=black,
	anchorcolor=black,
	citecolor=black}

\pgfplotsset{compat=1.18}

\numberwithin{equation}{section}

\newtheorem{theorem}{Theorem}[section]
\newtheorem{proposition}[theorem]{Proposition}
\newtheorem{corollary}[theorem]{Corollary}
\newtheorem{lemma}[theorem]{Lemma}

\theoremstyle{definition}

\theoremstyle{remark}

\newcommand{\vertiii}[1]{{\left\vert\kern-0.25ex\left\vert\kern-0.25ex\left\vert #1
		\right\vert\kern-0.25ex\right\vert\kern-0.25ex\right\vert}}

\begin{document}
	
    \title[Water waves with compactly supported vorticity II]{Capillary--gravity water waves with compactly supported vorticity\\ II. Continuation and folds of the solution curve}
	
	\author{Lili Du}
	\address{Department of Mathematics, Sichuan University, Chengdu, Sichuan 610064, P.R. China}
		\email{dulili@scu.edu.cn}
	
	\author{Changjun Zou}
	\address{Department of Mathematics, Sichuan University, Chengdu, Sichuan 610064, P.R. China}
	\email{zouchangjun@amss.ac.cn}
	
    \begin{abstract}
    	This five-part series establishes the first systematic theory of capillary--gravity water waves carrying compactly supported vorticity with fixed positive circulation. Treating the free surface and the vortex as coupled unknowns, we study the well-posedness of traveling-wave equilibria, including the existence and continuation of solution branches with respect to bifurcation parameters, as well as their variational structure, geometric rigidity, canonical classification, and conditional dynamics. In this second paper of the series, we apply the Haziot--Wheeler analytic continuation theorem \cite{HaziotWheeler2023} to extend the small-speed point-vortex branch constructed in Part~I \cite{DuZouPartI} to a global curve of coupled equilibria. We establish two alternatives: either the geometric control quantity diverges, or it remains bounded along a sequence of solutions with no strongly convergent subsequence. If the second alternative occurs, we use concentration--compactness to obtain a finite profile decomposition without assuming monotonicity of the free surface: a nontrivial centered vortex wave and finitely many reflected pairs of irrotational waves, with a remainder converging strongly to zero. At any simple speed fold, a Schur-complement reduction in the physical vortex coordinate shows that the static energy--momentum Hessians of the two nearby equal-speed waves differ by one in negative index within the symmetric class. Finally, we use a joint two-parameter Lyapunov--Schmidt reduction to extend the regularization of Part~I to compact arcs of the global solution curve, carrying the canonical smooth-core branches through their specified simple speed folds. These results provide a global point-vortex reference for the variational construction in Part~III \cite{DuZouPartIII} and canonical smooth-core comparison arcs for the rigidity and uniqueness theory in Part~IV \cite{DuZouPartIV}.
    \end{abstract}

	\maketitle
	

	
	\clearpage
	\tableofcontents
	
	\section{Introduction}
	\label{sec1}
	
	This five-part series develops a systematic theory of capillary--gravity water waves carrying compactly supported vorticity with fixed positive circulation. Treating the free surface and the vortex as coupled unknowns, we address the existence and global continuation of traveling-wave equilibria, their variational structure, geometric rigidity, canonical classification, and conditional dynamics. Part~I \cite{DuZouPartI} constructs the small-speed point-vortex branch, characterizes vortex equilibria via the Kirchhoff--Routh function, and regularizes individual waves by smooth cores. The present Part~II develops the global continuation theory, analyzes possible solution curve folds and loss of compactness, and extends regularization to compact solution arcs passing through folds. These results provide a global point-vortex reference for the variational construction of smooth solutions developed in Part~III \cite{DuZouPartIII}. By contrast, the smooth-core solution arcs constructed in this paper arise from a singular perturbation of the point-vortex problem, while Part~IV \cite{DuZouPartIV} establishes the equivalence of the perturbative and variational constructions by proving rigidity and fixed-speed uniqueness under the stated hypotheses.
	
	We consider symmetric two-dimensional capillary--gravity solitary waves with compactly supported vorticity in an inviscid, incompressible fluid of infinite depth; see \cite{Lannes2013,HaziotEtAl2022} for general background. As in Part~I \cite{DuZouPartI}, we fix $g,\sigma,\Gamma>0$ and work in a frame moving with velocity $-We_1$, where $W>0$, so that the free surface $\partial\Omega$, with $\Omega=\{(\xi,y)\in\mathbb R^2:y<\eta(\xi)\}$ and even $\eta$, is stationary. Writing $v_{\rm rel}=v+We_1$ for the relative velocity and $\omega=\operatorname{curl}v_{\rm rel}$ for the vorticity, the steady Euler equations yield $(v_{\rm rel}\cdot\nabla)v_{\rm rel}+\nabla\mathsf p+ge_2=0$, $\nabla\cdot v_{\rm rel}=0$, and $(v_{\rm rel}\cdot\nabla)\omega=0$. On the free surface, the kinematic and capillary conditions are $v_{\rm rel}\cdot n=0$ and $\mathsf p=\sigma\kappa$, where $n=(-\eta',1)/\sqrt{1+(\eta')^2}$ and $\kappa=-\eta''/(1+(\eta')^2)^{3/2}$. The solitary-wave conditions $\eta\to0$ and $v\to0$ at infinity yield the Bernoulli relation
	\[
	\frac12|v_{\rm rel}(\xi,\eta(\xi))|^2+g\eta(\xi)+\sigma\kappa(\xi)=\frac{W^2}{2}.
	\]
	For a point vortex of circulation $\Gamma$ at $p_0\in\Omega$ on the symmetry axis, stationarity in the moving frame is equivalent to the vanishing of its regular drift \cite{Saffman1993},
	\[
	v_{\rm reg}(p_0):=\lim_{q\to p_0}\left[v_{\rm rel}(q)-\frac{\Gamma}{2\pi}\frac{(q-p_0)^\perp}{|q-p_0|^2}\right]=0,
	\qquad r^\perp:=(-r_2,r_1).
	\]
	The resulting coupling is bidirectional: the vortex-induced velocity determines the free-surface response through Bernoulli's condition, while the surface geometry feeds back into the vortex equilibrium through the conformal factor in the Kirchhoff--Routh function. 
	
	As the solution curve is continued globally from the local branch constructed in Part~I, two interesting mechanisms may arise: loss of compactness, where strong convergence fails, and \emph{simple folds}, where the curve remains regular but the wave speed undergoes a nondegenerate turning. Beyond establishing global continuation and regularization, we also provide a detailed mathematical characterization of both mechanisms.
	
	In the irrotational capillary--gravity setting, Buffoni, Groves, and Toland \cite{BuffoniGrovesToland1996} established the existence of a rich family of solitary waves. The existence and properties of such waves have also been investigated in \cite{VandenBroeckDias1992,IoossKirrmann1996,GrovesWahlen2011,Wheeler2018}. For rotational waves, Constantin and Strauss \cite{ConstantinStrauss2004} constructed a global continuum of periodic gravity waves, while Constantin, Strauss, and Varvaruca \cite{ConstantinStraussVarvaruca2016} made a major advance by developing a Riemann--Hilbert formulation of analytic global bifurcation encompassing large-amplitude waves and critical layers. Further developments include large-amplitude and stratified solitary waves \cite{Wheeler2013,ChenWalshWheeler2018} and the recent construction of overhanging solitary waves by D\'avila, del Pino, Musso, and Wheeler \cite{DavilaDelPinoMussoWheeler2026}. Particularly relevant to the present work, Gui, Wang, Yang, and Zhang \cite{GuiWangYangZhang2026} established a global bifurcation theory for deep-water Stokes waves with piecewise smooth vorticity and internal streamline interfaces, addressing the non-Fredholm difficulty through a singular bifurcation argument.
	
	Shatah, Walsh, and Zeng \cite{ShatahWalshZeng2013} established the existence of small-amplitude solitary capillary--gravity waves carrying a point vortex or compactly supported vorticity, using circulation as a small parameter. They also obtained global continuation for periodic point-vortex waves. Chen, Walsh, and Wheeler \cite{ChenWalshWheeler2019} developed a detailed theory of deep-water solitary waves with localized vorticity, including precise far-field asymptotics, nonexistence results, and new existence constructions with compactly supported vorticity. Related point-vortex constructions include finite-depth waves with vertically aligned vortices \cite{Varholm2016} and infinite-depth solitary waves with a finite dipole \cite{Le2019}. Ehrnstr\"om, Walsh, and Zeng \cite{EhrnstromWalshZeng2023} constructed stationary finite-depth waves with smooth, exponentially localized vorticity, whose surface deformation is small while the kinetic energy remains of order one. In the absence of gravity and surface tension, Crowdy \cite{Crowdy2023} obtained exact traveling waves with submerged point vortices.
	
	A major advance in the global theory of vortex-carrying water waves was made by Chen, Varholm, Walsh, and Wheeler \cite{ChenVarholmWalshWheeler2025}, who constructed large-amplitude solitary gravity waves with a submerged point vortex in finite depth at a fixed supercritical Froude number. Their framework allows overhanging free surfaces and also yields hollow-vortex regularization under a nondegeneracy condition. Their compactness argument exploits monotonicity in conformal coordinates to exclude an irrotational front. In contrast, we fix the circulation and vary the wave speed for infinite-depth capillary--gravity waves, restricting the free surface to the conformal graph class without assuming monotonicity.
	
	Classical local bifurcation theory \cite{CrandallRabinowitz1971} and global continuation methods \cite{Rabinowitz1971,Dancer1973,BuffoniToland2003} provide complementary frameworks for constructing and extending solution branches; related developments for elliptic front problems appear in \cite{ChenWalshWheeler2026Fronts}. Slease, Walsh, and Zhong \cite{SleaseWalshZhong2026} recently constructed small-amplitude point-vortex waves and hollow-vortex waves in finite depth using local implicit-function arguments. They note that surface tension generally prevents the monotonicity properties needed to control the limiting behavior of a global solution curve, leaving open the possibility of closed solution loops. Our branch presents a different challenge: it approaches a singular boundary point where the transformed equation is undefined and the vortex depth diverges as the wave speed tends to zero. We therefore apply the Haziot--Wheeler analytic continuation theorem \cite[Theorem~B.1]{HaziotWheeler2023}, which accommodates an initial analytic arc approaching such a boundary point under suitable Fredholm and local properness hypotheses. This differs from the Buffoni--Toland continuation used for rotating vortex patches \cite{HassainiaMasmoudiWheeler2020}, which starts from an interior bifurcation point. The continuation theorem provides a global solution curve, but its compactness requires further analysis. Under uniform surface and graph bounds and speeds in a compact subcritical interval, the vortex equilibrium equation controls the vortex depth and prevents collision with the free surface. Combining concentration--compactness \cite{Lions1984I,Lions1984II,Gerard1998,Struwe1984} with nonlinear profile subtraction, we characterize bounded loss of compactness through a finite profile decomposition consisting of a centered vortex wave and finitely many reflected pairs of irrotational waves, with a remainder converging strongly to zero in $H^{s+2}(\mathbb R)$ for $s>\frac12$.
	
	The global alternatives describe the behavior of the solution curve at large continuation parameters, while its wave speed may turn as a function of the conformal vortex depth $a$ at an interior point. At a simple fold, this turning is nondegenerate and gives two distinct nearby waves at the same speed on one side of the turning point. Variational and Hamiltonian approaches to water-wave multiplicity and stability have been developed in \cite{GrillakisShatahStrauss1987,GrillakisShatahStrauss1990,Plotnikov1992,Mielke2002,Buffoni2004,ConstantinStrauss2007,GrovesWahlen2015,BurtonToland2011,VarholmWahlenWalsh2020}. For the coupled surface--vortex problem, fixing the conformal vortex depth does not fix its physical position when the free surface varies. We therefore compute the energy--momentum Hessian in the physical vortex coordinate and use a Schur-complement reduction to separate the surface variations from vortex displacement. The full Hessian is thus decomposed into a surface quadratic form and a scalar contribution associated with the vortex equilibrium. Across a simple fold, the surface Hessian retains its negative index, while the scalar Schur complement associated with vortex displacement changes sign. Consequently, the energy--momentum Hessians of the two nearby equal-speed waves have negative indices differing by one in the symmetric class. Our fold results are conditional: we characterize the local geometry and variational structure of any simple fold on the continued branch and prove the persistence of prescribed folds under regularization, without asserting that such a fold actually occurs on the global curve.
	
	Classical constructions of concentrated steady vortices employed variational methods \cite{Turkington1983I,Turkington1983II} and singularly perturbed elliptic equations \cite{FlucherWei1998,SmetsVanSchaftingen2010,DeValeriolaVanSchaftingen2013,CaoLiuWei2014,CaoPengYan2015,LongWangZeng2019}. Related perturbative and global theories address vortex patches, pairs, and multipoles \cite{Wan1988,HmidiMateu2017,HassainiaMasmoudiWheeler2020,HassainiaWheeler2022,GarciaHaziot2023}. These regularization methods generally concern local constructions near prescribed equilibria, while existing global results do not address regularization across point-vortex folds. In the water-wave setting, related results include concentrated-vorticity constructions \cite{CordobaDiIorio2021} and hollow-vortex regularization at nondegenerate point-vortex configurations \cite{ChenVarholmWalshWheeler2025,ChenWalshWheeler2026Hollow}. To overcome the singular fixed-speed linearization at simple folds, we employ a joint two-parameter Lyapunov--Schmidt reduction to construct smooth-core solution arcs through these folds, with circulation and vorticity law fixed.
	
	On each prescribed compact regular point-vortex arc, we solve the surface and mass-constrained inner equations together in the reference lower half-plane. The remaining translation condition yields a normalized drift residual that converges to the point-vortex regular-drift residual in $C^2$ with an error of order $O(\varepsilon)$. This estimate gives uniform control of the smooth-core solution arcs and ensures the persistence of their specified simple folds.
	
	In this five-part series, Part~III \cite{DuZouPartIII} develops a joint dual variational construction of finite-core waves and analyzes their effective Hessian and refined asymptotics. Part~IV \cite{DuZouPartIV} establishes rigidity, exact reflection symmetry, and fixed-speed local uniqueness, identifying the perturbative and variational constructions under the stated hypotheses. Finally, Part~V \cite{DuZouPartV} studies fixed-impulse coercivity, energy exchange, and conditional logarithmic-time orbital tracking for the coupled free-surface Euler evolution.
	
	We now state the main results.
	
	\smallskip

	\subsection{Main results}
	
	We keep $g,\sigma,\Gamma>0$ fixed and use the conformal formulation of Part~I \cite{DuZouPartI}. Write $p=(p_1,p_2)$ for physical points and $z=x+\mathrm{i}y$ for the reference coordinate in the lower half-plane $\Pi_-:=\{z\in\mathbb C:\operatorname{Im}z<0\}$. We use $dp$ for physical area measure and $dx$ for measure on the reference boundary. Fix $s>\frac12$, and let $H_{\rm e}^{s+2}(\mathbb R)$ denote the real-valued even subspace of $H^{s+2}(\mathbb R)$. Our Hilbert transform convention is $\widehat{Hh}(\zeta)=\mathrm{i}\operatorname{sgn}(\zeta)\widehat h(\zeta)$, so $|D|=-H\partial_x$.
	
	The free surface is represented by a conformal map $Z_Y:\Pi_-\to\Omega_Y$, normalized by $Z_Y(z)-z\to0$ and $Z_Y'(z)\to1$ as $|z|\to\infty$ in $\overline{\Pi_-}$. Its boundary trace is
	\[
	Z_Y(x)=\Xi_Y(x)+\mathrm{i}Y(x),\qquad \Xi_Y(x):=x-HY(x).
	\]
	Here $Y$ is the surface height in conformal coordinates, while the physical surface is $\eta_Y=Y\circ\Xi_Y^{-1}$. We impose the graph condition $c_Y:=\inf_{\mathbb R}\Xi_Y'=\inf_{\mathbb R}(1+|D|Y)>0$, which ensures that $\Xi_Y$ is strictly increasing and $\Omega_Y=\{p\in\mathbb R^2:p_2<\eta_Y(p_1)\}$. No monotonicity of the surface profile is assumed.
	
	The restriction on wave speed arises from the linearization at the flat far field, where the influence of the localized vortex vanishes. For a Fourier mode $e^{ikx}$ with $k\ne0$, the linearized surface operator has symbol $g+\sigma k^2-W^2|k|$. This symbol vanishes when $W=c_{\rm ph}(|k|):=(g/|k|+\sigma|k|)^{1/2}$. Its minimum is
	\[
	c_{\min}:=(4g\sigma)^{1/4}, 
	\]
	attained at $|k|=\sqrt{g/\sigma}$. Thus $0<W<c_{\min}$ ensures that the far-field symbol is strictly positive for every Fourier mode. This guarantees invertibility of the flat far-field operator, but not necessarily of the Bernoulli linearization about a nonflat wave. Moreover, $c_{\min}$ is a threshold for linear resonance, not the speed of a possible fold of the nonlinear coupled solution curve. We define
	\begin{equation}
		\mathcal O^s:=\left\{Y\in H_{\rm e}^{s+2}(\mathbb R):c_Y>0\right\},\qquad \mathscr U^s:=\mathcal O^s\times(0,\infty)\times(0,c_{\min}).
		\label{1-1}
	\end{equation}
	The coordinates on $\mathscr U^s$ are $(Y,a,W)$, where $a>0$ is the conformal vortex depth. The point vortex is located at $-\mathrm{i}a$ in the reference half-plane and at $p_a:=Z_Y(-\mathrm{i}a)$ in the physical domain. Writing $E_ah:=(e^{-a|D|}h)(0)$ for harmonic evaluation at the reference vortex point, symmetry gives $p_a=(0,p_2(Y,a))$, where $p_2(Y,a)=-a+E_aY$. Thus the physical vortex position depends on both $a$ and the unknown surface.
	
	For $Y\in\mathcal O^s$, the boundary metric and signed curvature are
	\[
	J(Y):=|Z_Y'|^2=(1+|D|Y)^2+(Y')^2,\qquad \kappa(Y):=-\frac{(1+|D|Y)Y''-Y'(|D|Y)'}{J(Y)^{3/2}}.
	\]
	We also set
	\[
	q_{a,W}(x):=W-\frac{\Gamma a}{\pi(x^2+a^2)}.
	\]
	The squared relative velocity on the free surface is $q_{a,W}^2/J(Y)$. The Bernoulli and vortex equilibrium residuals are
	\begin{equation}
		\begin{aligned}B(Y,a,W)&:=gY+\sigma\kappa(Y)+\frac{q_{a,W}^2}{2J(Y)}-\frac{W^2}{2},\\ P(Y,a,W)&:=W-\frac{\Gamma}{4\pi a}-\frac{\Gamma}{4\pi}\frac{E_a(Y'')}{1+E_a(|D|Y)}.\end{aligned}
		\label{1-2}
	\end{equation}
	The equation $B=0$ imposes Bernoulli's condition, while $P=0$ expresses the point-vortex equilibrium. Indeed, the horizontal component of the regular drift is $P(Y,a,W)/Z_Y'(-\mathrm{i}a)$, where $Z_Y'(-\mathrm{i}a)=1+E_a(|D|Y)>0$. The first two terms in $P$ give the flat-boundary balance, and the last term accounts for the influence of the free surface.
	
	The coupled point-vortex traveling-wave problem is therefore equivalent to
	\begin{equation}
		\mathscr F(Y,a,W):=\begin{pmatrix}B(Y,a,W)\\ P(Y,a,W)\end{pmatrix}=0,\qquad (Y,a,W)\in\mathscr U^s.
		\label{1-3}
	\end{equation}
	This equivalence is established through the velocity reconstruction in Proposition~\ref{prop2-1}.

	\smallskip
	\noindent\textit{Global continuation and possible loss of compactness.}
	Part~I \cite{DuZouPartI} provides a small-speed solution branch $(Y_W,a_W,W)$ with fixed circulation $\Gamma$. As $W\to0^+$, the conformal vortex depth satisfies $a_W\sim\Gamma/(4\pi W)\to\infty$, while $Y_W\to0$ in $H^{s+2}(\mathbb R)$. Thus the limiting configuration lies outside $\mathscr U^s$. We seek to continue this branch beyond the small-speed regime, without requiring the wave speed to remain a local parameter or the surface linearization to remain invertible along the entire curve.
	
	A central difficulty is the possible loss of compactness. If $\|Y\|_{H^{s+2}}$ remains bounded and $c_Y$, $W$, and $c_{\min}-W$ stay bounded away from zero, the vortex equilibrium equation $P=0$ controls the vortex depth and prevents collision with the free surface; see Lemma~\ref{lem4-1}. Nevertheless, these bounds do not guarantee strong convergence, since surface profiles may separate and move toward spatial infinity. By symmetry, such profiles occur in reflected pairs, while the vortex remains near the center.
	
	Far from the vortex, $q_{a,W}(x)$ approaches $W$, so any nontrivial distant profile satisfies the irrotational Bernoulli equation. For a real profile $Q\in H^2(\mathbb R)$ with $\inf_{\mathbb R}(1+|D|Q)>0$, not necessarily even, we define
	\begin{equation}
		B_{\rm irr}(Q,W):=gQ+\sigma\kappa(Q)+\frac{W^2}{2J(Q)}-\frac{W^2}{2},
		\label{1-4}
	\end{equation}
	where the boundary metric $J$ and signed curvature $\kappa$ are given by the same formulas as above. The following theorem distinguishes divergence of the geometric control quantity from bounded loss of compactness, and characterizes the latter situation through a finite decomposition into a centered vortex wave and reflected pairs of irrotational waves.
	
	\begin{theorem}[Global continuation and two alternatives]
		\label{thm1-1}
		Fix $g,\sigma,\Gamma>0$ and $s>\frac12$. Then there exists a continuous solution curve with $0<W(t)<c_{\min}$ for every $t>0$,
		\[
		\gamma:(0,\infty)\longrightarrow\mathscr U^s,\qquad \gamma(t)=(Y(t),a(t),W(t)),\qquad \mathscr F(\gamma(t))=0,
		\]
		whose image $\mathscr C_0:=\gamma((0,\infty))$ contains, for some sufficiently small $W_{\rm loc}>0$, the initial arc $\mathscr C_{\rm loc}:=\{(Y_W,a_W,W):0<W<W_{\rm loc}\}$ from Part~I. Near every $t_0>0$, the curve admits a local real-analytic reparametrization. We may choose the parametrization to follow $\mathscr C_{\rm loc}$ near $t=0$ and to lie outside this initial arc for all sufficiently large $t$. At the initial end, 
		\begin{equation}
			\begin{aligned}W(t)&\longrightarrow0,\qquad a(t)\sim\frac{\Gamma}{4\pi W(t)},\\ \|Y(t)\|_{H^{s+2}}&=O\bigl(W(t)^{3/2}\bigr)\qquad\text{as }t\downarrow0.\end{aligned}
			\label{1-5}
		\end{equation}
		Set
		\begin{equation}
			\mathcal N_{\rm s}(t):=\|Y(t)\|_{H^{s+2}}+\frac1{c_{Y(t)}}+\frac1{W(t)}+\frac1{c_{\min}-W(t)}.
			\label{1-6}
		\end{equation}
		Then one of the following alternatives holds.
		\begin{enumerate}
			\item[(i)] We have
			\begin{equation}
				\mathcal N_{\rm s}(t)\longrightarrow\infty\qquad\text{as }t\to\infty.
				\label{1-7}
			\end{equation}
			
			\item[(ii)] There is a sequence $t_n\to\infty$ with $\sup_n\mathcal N_{\rm s}(t_n)<\infty$ such that the corresponding conformal surface profiles $\{Y(t_n)\}$ have no strongly convergent subsequence in $H^{s+2}(\mathbb R)$. Set $(Y_n,a_n,W_n):=\gamma(t_n)$. After passing to a subsequence,
			\[
			a_n\longrightarrow a_\infty>0,\qquad W_n\longrightarrow W_\infty\in(0,c_{\min}).
			\]
			There is a nonzero centered profile $Y^{(0)}\in\mathcal O^s\cap H^{s+3}(\mathbb R)$ with $\mathscr F(Y^{(0)},a_\infty,W_\infty)=0$. There are also an integer $N\ge1$ and nonzero real profiles $Q_1,\ldots,Q_N\in H^{s+3}(\mathbb R)$, where $N$ is the number of reflected profile pairs, such that
			\[
			B_{\rm irr}(Q_j,W_\infty)=0,\qquad \inf_{\mathbb R}(1+|D|Q_j)>0,\qquad 1\le j\le N.
			\]
			There are positive centers $x_n^{(j)}$ in the conformal coordinate $x$, for $1\le j\le N$, such that $Y_n(\cdot+x_n^{(j)})\to Q_j$ in $H_{\rm loc}^{s+2}(\mathbb R)$ and
			\[
			x_n^{(j)}\longrightarrow+\infty,\qquad |x_n^{(j)}-x_n^{(k)}|\longrightarrow\infty\quad\text{for }j\ne k.
			\]
			For these centers, we have
			\begin{equation}
				\begin{aligned}Y_n(x)&=Y^{(0)}(x)+\sum_{j=1}^N\bigl[Q_j(x-x_n^{(j)})+Q_j(-x-x_n^{(j)})\bigr]+R_n^{(N)}(x),\\ \|R_n^{(N)}\|_{H^{s+2}}&\longrightarrow0.\end{aligned}
				\label{1-8}
			\end{equation}
			In particular, $Y_n\to Y^{(0)}$ in $H_{\rm loc}^{s+2}(\mathbb R)$, while the nonzero profiles at infinity prevent global strong convergence in $H^{s+2}(\mathbb R)$.
		\end{enumerate}
	\end{theorem}
	
	Along the solution curve $\gamma(t)$, we may take $t=W$ on the initial small-speed arc, but a different parameter may be needed if the speed turns (at a fold) along the continued curve. Because $W_{\mathrm{loc}}$ lies strictly inside the fixed-speed uniqueness interval of Lemma~\ref{lem3-1}, the continued curve must coincide locally with the initial branch at every point of that branch with $0<W\le W_{\mathrm{loc}}$. The global curve eventually leaves the initial arc; consequently, it must contain solutions with $W>W_{\mathrm{loc}}$. This does not imply that $W(t)$ is monotone along the curve.
	
	In alternative~{\rm(i)}, divergence of $\mathcal N_{\rm s}$ means that, along some sequence, the surface norm becomes unbounded, $c_Y$ approaches zero, or the speed approaches $0$ or $c_{\min}$. In particular, $c_Y\to0$ indicates loss of the conformal graph bound, not necessarily a singular physical free surface. The profile decomposition in alternative~{\rm(ii)} identifies centers $x_n^{(j)}$ in conformal coordinates, which need not coincide with wave crests. Their separation describes the spatial structure of a sequence of steady waves, rather than the motion of profiles along the solution curve.
	
	The centered limit $(Y^{(0)},a_\infty,W_\infty)$ is a nontrivial coupled wave satisfying \eqref{1-3} and carrying the full circulation $\Gamma$, while each $Q_j$ is an irrotational wave at the same speed $W_\infty$. The Bernoulli equation gives both types of profiles one additional derivative, placing them in $H^{s+3}(\mathbb R)$. Although reflection symmetry pairs each distant profile with its reflection, we impose no sign, concavity, single-crest, or evenness condition on $Q_j$. Their separation and the strong convergence of the remainder to zero in $H^{s+2}(\mathbb R)$ yield
	\[
	\lim_{n\to\infty}\|Y_n-Y^{(0)}\|_{H^{s+2}}^2=2\sum_{j=1}^N\|Q_j\|_{H^{s+2}}^2>0.
	\]
	Thus the failure of global strong convergence is entirely due to profiles whose centers escape to horizontal infinity in conformal coordinates. Figure~\ref{fig:profile-splitting} illustrates one reflected pair.
	
	\begin{figure}[htbp]
		\centering
		\includegraphics[width=0.96\linewidth]{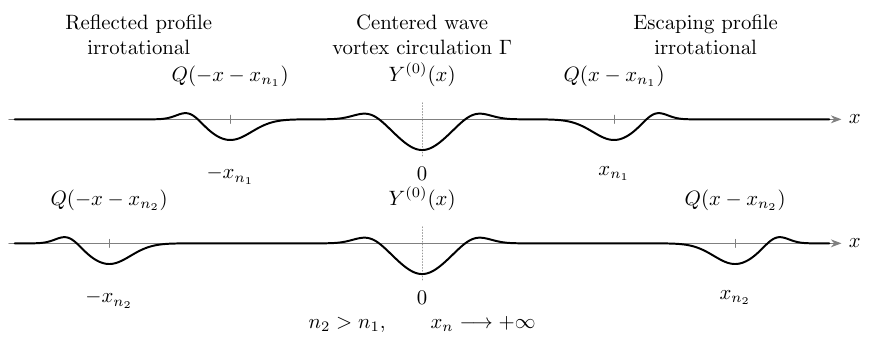}
		\caption{Profile separation in alternative~{\rm(ii)} of Theorem~\ref{thm1-1}, illustrated for one reflected pair. The centered wave retains circulation $\Gamma$, while the outer profiles are irrotational; the remainder in \eqref{1-8} is omitted. The rows depict different steady configurations with increasing separation for illustration, not time evolution or a claimed monotone movement along the curve. The profile shapes are schematic and have no prescribed sign or concavity.}
		\label{fig:profile-splitting}
	\end{figure}
	
	\smallskip
	\noindent\textit{Speed folds along the solution curve and the energy--momentum Hessian.}
	To study possible speed folds, consider a coupled solution $(Y_\star,a_\star,W_\star)$ at which the surface linearization
	\[
	A:=D_YB:H_{\rm e}^{s+2}(\mathbb R)\longrightarrow H_{\rm e}^s(\mathbb R)
	\]
	is invertible. The Bernoulli equation then determines a unique nearby surface $Y=Y(a,W)$, real analytic in the conformal vortex depth $a$ and wave speed $W$. On this local \emph{Bernoulli chart}, the remaining vortex equilibrium condition reduces to
	\begin{equation}
		\mathscr G_0(a,W):=P(Y(a,W),a,W),\qquad L:=\partial_a\mathscr G_0,\qquad M:=\partial_W\mathscr G_0.
		\label{1-9}
	\end{equation}
	Here $L$ and $M$ are the total derivatives of the vortex equilibrium residual with respect to $a$ and $W$, respectively, including the dependence of the free surface $Y(a,W)$ on these parameters. The nearby coupled solutions are exactly the zeros of $\mathscr G_0$, and both $L$ and $M$ incorporate the response of the free surface to changes in $a$ and $W$.
	
	For a flat surface, the vortex equilibrium gives $W=\Gamma/(4\pi a)$, which is strictly decreasing in the conformal vortex depth $a$. However, in the coupled problem, the free surface also contributes to the regular drift, so this monotonicity is no longer guaranteed along the solution curve. Hence it is important to study possible folds: at such points, the speed ceases to be a local parameter, and distinct nearby waves may have the same speed. On a Bernoulli chart with $M_\star\ne0$, the reduced equation $\mathscr G_0=0$ determines $W=w(a)$, with
	\[
	w'(a)=-\frac{L(a,w(a))}{M(a,w(a))}.
	\]
	If $L_\star=0$ and $(\partial_aL)_\star\ne0$, then $w'(a_\star)=0$ and $w''(a_\star)=-(\partial_aL)_\star/M_\star\ne0$, giving a simple speed fold. The solution curve remains regular, but the fixed-speed linearization becomes singular. 
	
	A simple speed fold also has a variational interpretation. Following the Hamiltonian formulation of free surfaces coupled to interior vortices \cite{RouhiWright1993}, we consider the energy $E$, after subtraction of the log-divergent point-vortex self-energy, and the horizontal momentum $\mathcal P$. With $\alpha(Y,a):=Z_Y'(-\mathrm{i}a)=1+E_a(|D|Y)>0$, their expressions for the reconstructed point-vortex flow are given below; see Section~\ref{sec6} for their derivation.
	\[
	\begin{aligned}
		E(Y,a,W)={}&\frac g2\int_{\mathbb R}Y^2(1+|D|Y)\,dx+\sigma\int_{\mathbb R}\bigl(\sqrt{J(Y)}-1-|D|Y\bigr)\,dx\\
		&+\frac{W^2}{2}\int_{\mathbb R}Y|D|Y\,dx+\frac{\Gamma^2}{4\pi}\log\bigl(2a\alpha(Y,a)\bigr),\\
		\mathcal P(Y,a,W)={}&-W\int_{\mathbb R}Y|D|Y\,dx-\Gamma a.
	\end{aligned}
	\]
	For waves traveling with velocity $-We_1$, the energy--momentum action is $E+W\mathcal P$, which gives
	\begin{equation}
		\begin{aligned}
			\mathcal J_W(Y,a):={}&\frac g2\int_{\mathbb R}Y^2(1+|D|Y)\,dx\\
			&+\sigma\int_{\mathbb R}\bigl(\sqrt{J(Y)}-1-|D|Y\bigr)\,dx-\frac{W^2}{2}\int_{\mathbb R}Y|D|Y\,dx\\
			&+\frac{\Gamma^2}{4\pi}\log\bigl(2a\alpha(Y,a)\bigr)-\Gamma Wa.
		\end{aligned}
		\label{1-10}
	\end{equation}
	As shown in Section~\ref{sec6}, its critical points at fixed $W$ are precisely the coupled equilibria. The last two terms give the vortex Kirchhoff--Routh contribution, including the dependence on the free surface through $\alpha$. We also hold $W$ fixed when computing the Hessian at each wave.
	
	The choice of vortex coordinate matters when computing the Hessian. Recall that the physical vertical position of the vortex is $p_2(Y,a)=-a+E_aY$, so fixing the conformal depth $a$ does not fix the physical vortex position when the surface varies. Since $\partial_a p_2(Y,a)=-\alpha(Y,a)<0$ at fixed $Y$, we may use $(Y,p_2)$ as local coordinates and define $\widetilde{\mathcal J}_W(Y,p_2):=\mathcal J_W(Y,\mathfrak a(Y,p_2))$, 
	where $a=\mathfrak a(Y,p_2)$ is obtained by inverting the relation $p_2=-a+E_aY$ for fixed $Y$. On the Bernoulli chart, set $p_2(a,W):=p_2(Y(a,W),a)$ and $\alpha(a,W):=\alpha(Y(a,W),a)$. Differentiating $p_2(a,W)=-a+E_aY(a,W)$ with respect to $a$ at fixed $W$ gives
	\[
	\partial_a p_2(a,W)=-1-E_a(|D|Y(a,W))+E_a(Y_a(a,W))=-\alpha(a,W)+E_a(Y_a(a,W)).
	\]
	We therefore define
	\begin{equation}
		\beta_{\rm g}(a,W):=1-\frac{E_a(Y_a(a,W))}{\alpha(a,W)},
		\label{1-11}
	\end{equation}
	so that $\partial_a p_2(a,W)=-\alpha(a,W)\beta_{\rm g}(a,W)$. Thus $\beta_{\rm g}$ measures how the surface response affects the physical displacement of the vortex. 
	
	At a coupled equilibrium and for fixed wave speed $W$, we consider the Hessian of $\widetilde{\mathcal J}_W$ with respect to the free surface and physical vortex position. For a surface variation $h\in H_{\rm e}^{s+2}(\mathbb R)$ and a physical vortex displacement $\nu\in\mathbb R$, define
	\[
	\mathfrak q[(h,\nu)]:=D^2\widetilde{\mathcal J}_W(Y,p_2(Y,a))[(h,\nu),(h,\nu)],\qquad \mathfrak q_{\rm s}[h]:=\mathfrak q[(h,0)].
	\]
	Here $\mathfrak q_{\rm s}$ describes surface variations at fixed physical vortex position, while $\mathfrak q$ also includes vortex displacement. We write $n_-(\mathfrak q)$ for the maximal dimension of a subspace on which $\mathfrak q$ is negative definite, and $\ker\mathfrak q$ for its Hessian kernel. A Schur-complement reduction separates the surface and vortex contributions without requiring positivity of the surface Hessian. The following theorem characterizes simple speed folds and the associated change in negative index.
	
	\begin{theorem}[Simple speed folds and the negative index of the energy--momentum Hessian]
		\label{thm1-2}
		Let $(Y_\star,a_\star,W_\star)\in\mathscr U^s$ be a coupled solution with invertible surface linearization $A_\star$. On the associated Bernoulli chart, suppose that
		\begin{equation}
			L_\star=0,\qquad M_\star\ne0,\qquad (\partial_aL)_\star\ne0,
			\label{1-12}
		\end{equation}
		and set $\chi_\star:=-(\partial_aL)_\star/M_\star$. Then the following conclusions hold.
		\begin{enumerate}
			\item[(i)] All sufficiently nearby coupled solutions form a regular real-analytic curve
			\[
			\gamma_\star(a):=(Y(a,w(a)),a,w(a)),\qquad w(a_\star)=W_\star,
			\]
			with
			\begin{equation}
				w(a)-W_\star=\frac{\chi_\star}{2}(a-a_\star)^2+O(|a-a_\star|^3),\qquad \chi_\star\ne0.
				\label{1-13}
			\end{equation}
			In a sufficiently small configuration neighborhood, each speed $W$ close to $W_\star$ admits exactly two solutions if $\chi_\star(W-W_\star)>0$, none if $\chi_\star(W-W_\star)<0$, and only the fold solution if $W=W_\star$.
			
			The full linearization $D_{(Y,a,W)}\mathscr F$ is surjective along this curve, while the fixed-speed linearization $D_{(Y,a)}\mathscr F$ is invertible away from the fold and satisfies
			\[
			\ker D_{(Y,a)}\mathscr F(Y_\star,a_\star,W_\star)=\operatorname{span}\{(-A_\star^{-1}B_{a,\star},1)\}.
			\]
			
			\item[(ii)] The geometric factor defined in \eqref{1-11} satisfies $\beta_{{\rm g},\star}:=\beta_{\rm g}(a_\star,W_\star)\ne0$. Let $\mathfrak q(a)$ and $\mathfrak q_{\rm s}(a)$ denote the Hessian forms defined above, evaluated along $\gamma_\star(a)$ with $W=w(a)$ fixed in each variation. On a sufficiently small segment around the fold, these forms extend to closed, lower-bounded forms on $L_{\rm e}^2(\mathbb R)\times\mathbb R$ and $L_{\rm e}^2(\mathbb R)$, with respective common form domains $H_{\rm e}^1(\mathbb R)\times\mathbb R$ and $H_{\rm e}^1(\mathbb R)$. Both have finite negative index.
			
			The surface form induces an isomorphism $H_{\rm e}^1(\mathbb R)\to H_{\rm e}^{-1}(\mathbb R)$, and its negative index remains constant along this segment. For every sufficiently small $\delta>0$,
			\begin{equation}
				n_-(\mathfrak q(a_\star+\delta))-n_-(\mathfrak q(a_\star-\delta))=\operatorname{sgn}\bigl(\beta_{{\rm g},\star}(\partial_aL)_\star\bigr).
				\label{1-14}
			\end{equation}
			The full form has trivial kernel away from the fold, whereas at $a=a_\star$,
			\[
			n_-(\mathfrak q(a_\star))=n_-(\mathfrak q_{\rm s}(a_\star)),\qquad \dim\ker\mathfrak q(a_\star)=1.
			\]
			The Hessian of $\mathcal J_{w(a)}$ in conformal coordinates also extends to $H_{\rm e}^1(\mathbb R)\times\mathbb R$ and has the same negative index and nullity as $\mathfrak q(a)$.
		\end{enumerate}
	\end{theorem}
	 
	Part~{\rm(i)} describes the local geometry of a simple speed fold. The two nearby waves at the same speed have different conformal vortex depths, and their physical vortex depths also differ. Indeed, setting $d(a):=-p_2(a,w(a))$, we have $d'(a_\star)=\alpha_\star\beta_{{\rm g},\star}\ne0$ by part~{\rm(ii)}. Thus the fold corresponds to distinct physical configurations, rather than merely different conformal coordinates. Figure~\ref{fig:simple-fold} illustrates the case $\chi_\star>0$.
	\begin{figure}[!htbp]
		\centering
		\includegraphics[width=0.76\linewidth]{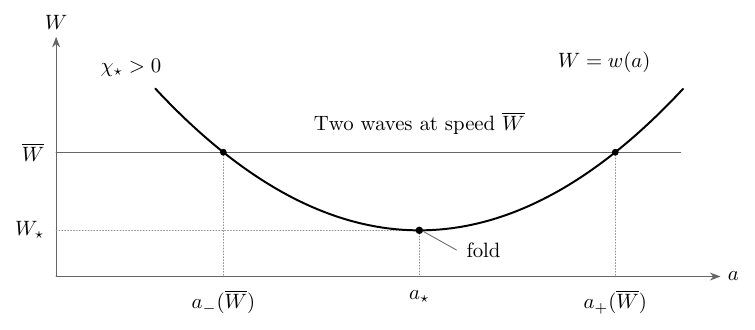}
		\caption{A simple speed fold in Theorem~\ref{thm1-2} for $\chi_\star>0$. Each nearby speed $\overline W>W_\star$ corresponds locally to two waves with conformal depths $a_-(\overline W)$ and $a_+(\overline W)$. The case $\chi_\star<0$ gives a local maximum. The schematic diagram depicts the solution curve, not the free surface.}
		\label{fig:simple-fold}
	\end{figure}
	Part~{\rm(i)} also identifies the kernel of the fixed-speed linearization at the fold, spanned by $(Y_a(a_\star,W_\star),1)=(-A_\star^{-1}B_{a,\star},1)$. This direction describes the coupled variation of the free surface and conformal vortex depth that preserves the equilibrium equations to first order while the wave speed remains fixed.
	
	To explain the negative-index change in Part~{\rm(ii)}, we eliminate the surface variable from the energy--momentum action. For fixed wave speed $W$ and physical vortex position $p_2$, the invertibility of the surface Hessian allows us to solve the surface stationarity condition for $Y=\widetilde Y(p_2,W)$. To identify the contribution of vortex displacement to the full Hessian, we substitute the stationary surface $Y=\widetilde Y(p_2,W)$ into the action and define the reduced action $\Phi(p_2,W):=\widetilde{\mathcal J}_W(\widetilde Y(p_2,W),p_2)$. At each coupled equilibrium and for fixed $W$, its second derivative with respect to $p_2$ satisfies
	\begin{equation}
		\Phi_{p_2p_2}(p_2(a,W),W)=-\frac{\Gamma L(a,W)}{\alpha(a,W)^2\beta_{\rm g}(a,W)},
		\label{1-15}
	\end{equation}
	which relates the effective curvature of the action to the derivative $L$ of the reduced vortex equilibrium residual. Since $\Gamma>0$, $\alpha>0$, and $\beta_{\rm g}$ remains nonzero near a simple fold, this curvature vanishes at the fold and changes sign as $L$ crosses zero.

	The reduced curvature $\Phi_{p_2p_2}$ is the Schur complement of the surface Hessian. Writing the full Hessian in physical coordinates $(Y,p_2)$ with blocks $H_{YY},H_{Yp},H_{pY},H_{pp}$, we have
	\[
	\Phi_{p_2p_2}=H_{pp}-H_{pY}H_{YY}^{-1}H_{Yp}.
	\]
	Near a simple fold, the surface Hessian has constant negative index, while $L$ changes sign and $\alpha>0$, $\beta_{\rm g}\ne0$. Consequently, the scalar Schur complement changes sign, and the two nearby equal-speed waves have negative indices differing by one. This provides a point-vortex explanation of the index change and complements the independent joint dual variational construction and effective finite-core Hessian analysis in Part~III \cite{DuZouPartIII}.

	\smallskip
	\noindent\textit{Regularization along the solution curve.}
	Part~I \cite{DuZouPartI} constructs smooth-core regularizations of individual small-speed point-vortex waves at fixed speed. We now seek to regularize compact arcs of the continued point-vortex solution curve, including any specified simple speed folds. Let $\mathscr K\subset\mathscr C_0$ be a compact regular embedded subarc satisfying
	\begin{equation}
		A\text{ is invertible},\qquad L^2+M^2>0\quad\text{on }\mathscr K.
		\label{1-16}
	\end{equation}
	Recall that $A=D_YB$ is the Bernoulli surface linearization, while $L=\partial_a\mathscr G_0$ and $M=\partial_W\mathscr G_0$ are defined in \eqref{1-9}. The invertibility of $A$ allows us to solve the Bernoulli equation locally for $Y=Y(a,W)$, while $L^2+M^2>0$ ensures that the reduced equation $\mathscr G_0(a,W)=0$ defines a regular curve, even at points where $L=0$ and $M\ne0$.
	
	As in Part~I, we fix a vorticity law $f\in C_{\rm loc}^{2,\gamma}(\mathbb R)$, with $\gamma\in(0,1)$, such that $f(\tau)=0$ for $\tau\le0$ and $f(\tau)>0$ for $\tau>0$. For the prescribed circulation $\Gamma$, we assume that there exists a radial profile $U\in C_{\rm loc}^{2,\gamma}(\mathbb R^2)$, positive in $B_1$, satisfying
	\begin{equation}
		\begin{aligned}
			-\Delta U&=f(U)\quad\text{in }\mathbb R^2,\qquad \int_{\mathbb R^2}f(U)\,dX=\Gamma,\\
			U(X)&=-\frac{\Gamma}{2\pi}\log|X|\qquad\text{for }|X|\ge1.
		\end{aligned}
		\label{1-17}
	\end{equation}
	The vorticity law $f$, reference core $U$, and circulation $\Gamma$ remain fixed throughout $\mathscr K$.
	
	The construction also requires the two inner nondegeneracy conditions of Part~I \cite[Section~5.1]{DuZouPartI}, verified there for the power-law class $f(\tau)=c(\tau_+)^p$, $p>2$. First, every bounded solution of $(-\Delta-f'(U))h=0$ in $\mathbb R^2$ belongs to $\operatorname{span}\{\partial_1U,\partial_2U\}$. Symmetry removes the horizontal translation mode, while the vertical mode gives a scalar solvability condition. Second, we assume
	\begin{equation}
		\int_{B_1}f'(U)\,dX\ne0.
		\label{1-18}
	\end{equation}
	This condition allows us to choose the scalar level in the vorticity law so that the total circulation is exactly $\Gamma$. Under these assumptions, we use a joint two-parameter Lyapunov--Schmidt reduction to solve the free-surface and mass-constrained inner equations together in conformal coordinates, treating $a$ and $W$ as independent parameters. The remaining vertical translation condition yields a scalar regular-drift equation $\mathscr G_\varepsilon(a,W)=0$, whose limit is the point-vortex regular-drift equation $\mathscr G_0(a,W)=0$.
	
	The main difficulty is to preserve simple folds along the regularized solution curve. At a simple fold, $L_\star=0$, so the fixed-speed argument cannot determine the vortex depth as a function of the wave speed. Moreover, the existence of a nearby smooth-core curve alone does not guarantee the persistence of a turning point. Hence on compact parameter sets $K$ in the local Bernoulli charts, we establish the two-parameter estimate
	\[
	\|\mathscr G_\varepsilon-\mathscr G_0\|_{C^2(K)}\le C_K\varepsilon.
	\]
    To preserve a simple fold, we must solve both $\mathscr G_\varepsilon=0$ and $\partial_a\mathscr G_\varepsilon=0$, rather than merely construct nearby solution curves. Since the Jacobian of this system involves second derivatives of $\mathscr G_\varepsilon$, $C^1$ convergence alone is insufficient. Our $C^2$ estimate controls this Jacobian and the quadratic turning coefficient. As the limiting Jacobian has determinant $-M_\star(\partial_aL)_\star\ne0$, each specified simple fold persists uniquely under regularization, with the same turning direction and local two-wave structure. The following theorem establishes the regularization of prescribed compact solution arcs and the persistence of their specified interior simple folds.

	\begin{theorem}[Regularization of solution curves and persistence through folds]
		\label{thm1-3}
		Let $f$ and $U$ satisfy the assumptions above, and let $\mathscr K\subset\mathscr C_0$ be a compact regular embedded subarc satisfying \eqref{1-16}. Choose a smooth regular embedding $\bm{\gamma}:I\to\mathscr C_0$ with image $\mathscr K$, where $I$ is a compact interval, and write $\bm{\gamma}(t)=(Y(t),a(t),W(t))$. Then there exist $\varepsilon_{\mathscr K},C_{\mathscr K}>0$ such that, for every $0<\varepsilon<\varepsilon_{\mathscr K}$, there is a family of symmetric smooth-core traveling waves
		\[
		\bigl(\Omega_\varepsilon(t),\Psi_\varepsilon(t),\mu_\varepsilon(t)\bigr),\qquad t\in I,
		\]
		traveling with velocities $-W_\varepsilon(t)e_1$. The following properties hold uniformly for $t\in I$.
		
		\begin{enumerate}
			\item[(i)] The relative stream function and vorticity satisfy
			\begin{equation}
				-\Delta_p\Psi_\varepsilon(t)=\omega_{\varepsilon,t},\qquad \omega_{\varepsilon,t}=\varepsilon^{-2}f(\Psi_\varepsilon(t)-\mu_\varepsilon(t))\ge0\quad\text{in }\Omega_\varepsilon(t).
				\label{1-19}
			\end{equation}
			On the free surface,
			\[
			\Psi_\varepsilon(t)=0,\qquad \frac12|\nabla_p\Psi_\varepsilon(t)|^2+gp_2+\sigma\kappa_\varepsilon(t)=\frac{W_\varepsilon(t)^2}{2}\quad\text{on }\partial\Omega_\varepsilon(t),
			\]
			where $\kappa_\varepsilon(t)$ is the signed curvature defined above. The fluid velocity $v_\varepsilon(t):=-\nabla_p^\perp\Psi_\varepsilon(t)-W_\varepsilon(t)e_1$ tends to zero at spatial infinity.

			The fluid domains satisfy $\Omega_\varepsilon(t)=Z_{\widehat Y_\varepsilon(t)}(\Pi_-)$ for some $\widehat Y_\varepsilon(t)$, and the macroscopic variables
			\[
			\bm{\gamma}_\varepsilon(t):=\bigl(\widehat Y_\varepsilon(t),a_\varepsilon(t),W_\varepsilon(t)\bigr)\in\mathscr U^s
			\]
			form a regular $C^2$ embedded arc satisfying
			\begin{equation}
				\|\bm{\gamma}_\varepsilon-\bm{\gamma}\|_{C^2(I;H_{\rm e}^{s+2}(\mathbb R)\times\mathbb R^2)}\le C_{\mathscr K}\varepsilon.
				\label{1-20}
			\end{equation}
			To establish a transverse correspondence between the point-vortex and smooth-core solution curves, we choose their parametrizations so that
			\begin{equation}
				M(t)\bigl(a_\varepsilon(t)-a(t)\bigr)-L(t)\bigl(W_\varepsilon(t)-W(t)\bigr)=0,
				\label{1-21}
			\end{equation}
			where $L(t)$ and $M(t)$ are evaluated at $\bm{\gamma}(t)$. Under the prescribed mass and translation normalizations, this correspondence is locally unique in the formulation of Section~\ref{sec7}.
			
			\item[(ii)] For every $t\in I$, the vorticity has exactly the prescribed circulation $\Gamma$. With
			\[
			p_\varepsilon(t):=Z_{\widehat Y_\varepsilon(t)}(-\mathrm{i}a_\varepsilon(t)),\qquad p(t):=Z_{Y(t)}(-\mathrm{i}a(t)),
			\]
			we have
			\begin{equation}
				\int_{\Omega_\varepsilon(t)}\omega_{\varepsilon,t}\,dp=\Gamma,\qquad \operatorname{supp}\omega_{\varepsilon,t}\subset B_{C_{\mathscr K}\varepsilon}(p_\varepsilon(t)).
				\label{1-22}
			\end{equation}
			The vortex centers satisfy
			\[
			\sup_{t\in I}|p_\varepsilon(t)-p(t)|\le C_{\mathscr K}\varepsilon.
			\]
			Extending $\omega_{\varepsilon,t}$ by zero outside $\Omega_\varepsilon(t)$, for every bounded Lipschitz function $\varphi:\mathbb R^2\to\mathbb R$,
			\begin{equation}
				\sup_{t\in I}\left|\int_{\mathbb R^2}\varphi(p)\omega_{\varepsilon,t}(p)\,dp-\Gamma\varphi(p(t))\right|\le C_{\mathscr K}\Gamma\operatorname{Lip}(\varphi)\varepsilon.
				\label{1-23}
			\end{equation}
			
			\item[(iii)] Let $(Y_\star,a_\star,W_\star)$ be a specified simple fold in the interior of $\mathscr K$, satisfying \eqref{1-12}. For sufficiently small $\varepsilon>0$, the smooth-core solution curve has a unique simple fold near $(Y_\star,a_\star,W_\star)$ within the prescribed normalized configuration neighborhood. Near this fold, the smooth-core solutions admit a $C^2$ parametrization by the conformal depth $a$, with speed $W=w_\varepsilon(a)$ satisfying
			\[
			w_\varepsilon(a_{\varepsilon,\star})=W_{\varepsilon,\star},\qquad w_\varepsilon'(a_{\varepsilon,\star})=0,\qquad \chi_{\varepsilon,\star}:=w_\varepsilon''(a_{\varepsilon,\star})\ne0.
			\]
			With $\chi_\star=-(\partial_aL)_\star/M_\star$ as in Theorem~\ref{thm1-2},
			\begin{equation}
				|a_{\varepsilon,\star}-a_\star|+|W_{\varepsilon,\star}-W_\star|+|\chi_{\varepsilon,\star}-\chi_\star|\le C_\star\varepsilon.
				\label{1-24}
			\end{equation}
			In particular, $\chi_{\varepsilon,\star}$ has the same sign as $\chi_\star$, and
			\[
			w_\varepsilon(a)-W_{\varepsilon,\star}=\frac{\chi_{\varepsilon,\star}}2(a-a_{\varepsilon,\star})^2+o(|a-a_{\varepsilon,\star}|^2).
			\]
			Within the same normalized configuration neighborhood, every sufficiently close speed $W$ admits exactly two smooth-core solutions if $\chi_{\varepsilon,\star}(W-W_{\varepsilon,\star})>0$, none if this product is negative, and only the fold solution if $W=W_{\varepsilon,\star}$.
		\end{enumerate}
	\end{theorem}

	Part~{\rm(i)} compares the point-vortex and smooth-core solution arcs through the transverse correspondence \eqref{1-21}. Since $L(t)a'(t)+M(t)W'(t)=0$, the displacement $(a_\varepsilon-a,W_\varepsilon-W)$ is normal to the reference curve in the $(a,W)$ plane. This determines a correspondence between the two curves without imposing an additional physical constraint. At a simple fold, $L=0$ and $M\ne0$, so paired waves have the same conformal depth but may have different speeds. In particular, the smooth-core wave paired with the reference fold need not be the fold of the regularized curve. Figure~\ref{fig:fold-persistence} illustrates the coupled free boundaries and fold persistence.
	
	\begin{figure}[!htbp]
		\centering
		\includegraphics[width=0.98\linewidth]{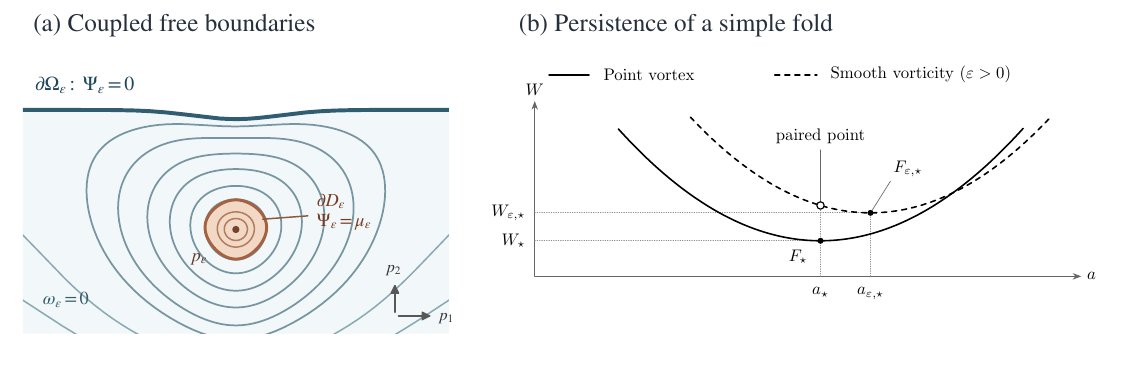}
		\caption{Coupled free boundaries and simple-fold persistence in Theorem~\ref{thm1-3}. (a) Relative streamlines between the free surface $\Psi_\varepsilon=0$ and the active-core boundary $\Psi_\varepsilon=\mu_\varepsilon$; core size and surface deformation are enlarged for visibility. (b) Point-vortex (solid) and smooth-core (dashed) branches near a specified simple fold. Filled points mark the two folds, whereas the open point is paired with $F_\star$ by \eqref{1-21} and need not coincide with $F_{\varepsilon,\star}$.}
		\label{fig:fold-persistence}
	\end{figure}
	
	Part~{\rm(ii)} describes the concentration of vorticity along the regularized solution arc. The smooth cores have fixed circulation $\Gamma$ and are supported in disks of radius $O(\varepsilon)$, with their centers remaining $O(\varepsilon)$-close to the corresponding point-vortex positions. Consequently, \eqref{1-23} gives uniform convergence of the vorticity measures to $\Gamma\delta_{p(t)}$ in the bounded-Lipschitz sense.
	
	Beyond constructing nearby smooth-core waves, part~{\rm(iii)} establishes the persistence of simple folds under two-parameter regularization. 
	Each specified interior simple fold of the point-vortex solution curve corresponds to a unique nearby simple fold of the smooth-core solution curve, with the same turning direction and local two-wave structure. For each prescribed compact regular arc $\mathscr K$, the construction holds for sufficiently small $\varepsilon$, with the admissible core size depending on $\mathscr K$. Thus the theorem does not establish smooth-core regularization along the entire global point-vortex curve for any fixed $\varepsilon>0$. It is also worth noting that these conclusions concern only steady equilibria and do not imply preservation of the point-vortex Hessian index or dynamical stability.

	\smallskip
	
	\subsection{Notation and conventions}
	We follow the conventions of Part~I \cite{DuZouPartI} and collect the notation used in the continuation, fold, and smooth-core analyses.
	
	\begin{itemize}
		\item We use $p=(p_1,p_2)$ for physical points, $z=x+\mathrm{i}y\in\Pi_-$ for conformal coordinates, and $\xi$ for the physical horizontal coordinate. The surface is represented by $Y(x)$ in conformal coordinates and by $\eta_Y(\xi)$ in physical coordinates. We write $z_a:=-\mathrm{i}a$, $p_a:=Z_Y(z_a)=(0,-h)$, and $b:=a^{-1}$, where $a$ and $h$ are the conformal and physical vortex depths. We also use $\vartheta_Y:=1+|D|Y=\Xi_Y'$, $c_Y:=\inf_{\mathbb R}\vartheta_Y$, and $\alpha(Y,a):=Z_Y'(-\mathrm{i}a)=1+E_a(|D|Y)>0$. Our curvature convention gives $D\kappa(0)[h]=-h''$.
		
		\item The Fourier transform and Bessel-potential Sobolev norms follow Part~I. In particular, the Hilbert transform $H$ has symbol $\mathrm{i}\operatorname{sgn}(\zeta)$, so $|D|=-H\partial_x$; see \cite{Stein1970}. The subscripts ${\rm e}$ and ${\rm o}$ denote real-valued even and odd subspaces. For $s>\frac12$, the point-vortex unknown and residual spaces are
		\[
		\mathcal X^s:=H_{\rm e}^{s+2}(\mathbb R)\times\mathbb R,\qquad \mathcal Y^s:=H_{\rm e}^{s}(\mathbb R)\times\mathbb R.
		\]
		
		\item For the relative stream function $\Psi$, we use $v_{\rm rel}=-\nabla_p^\perp\Psi$, where $\nabla_p^\perp\Psi=(-\partial_{p_2}\Psi,\partial_{p_1}\Psi)$, and $\omega=-\Delta_p\Psi$. The fluid velocity is $v=v_{\rm rel}-We_1$, and positive circulation corresponds to counterclockwise rotation. We use $dp$ for physical area measure and $dx$ for integration on the reference boundary.
		
		\item The admissible sets $\mathcal O^s,\mathscr U^s$ and the residuals $\mathscr F=(B,P)$ and $B_{\rm irr}$ are defined in \eqref{1-1}--\eqref{1-4}. In reciprocal-depth coordinates, we set $\widehat{\mathscr F}(Y,b,W):=\mathscr F(Y,b^{-1},W)$. We write $D$ for Fr\'echet derivatives, $A:=D_YB$ for the surface linearization, and $\mathscr L:=D_{(Y,a)}\mathscr F$ for the fixed-speed coupled linearization. On a Bernoulli chart, $\mathscr G_0$, $L:=\partial_a\mathscr G_0$, and $M:=\partial_W\mathscr G_0$ are defined in \eqref{1-9}, with the surface response included in both derivatives.
		
		\item We denote the chosen global point-vortex curve by $\mathscr C_0$ and a compact regular subarc by $\mathscr K$. In the compactness analysis, $Y^{(0)}$ denotes the centered limit, while $Q_j$ and $x_n^{(j)}$ denote the distant profiles and their centers. In Section~\ref{sec7}, $K$ denotes a compact parameter set in a Bernoulli chart, distinct from $\mathscr K$. A dot denotes differentiation with respect to a curve parameter $t$, a subscript $\star$ denotes evaluation at a distinguished configuration, and a prime denotes differentiation in the indicated scalar variable.
		
		\item We retain the Green-function normalization of Part~I,
		\[
		G_\Omega(p,q)=\frac1{2\pi}\log\frac1{|p-q|}-H_\Omega(p,q),\qquad R_\Omega(p):=H_\Omega(p,p).
		\]
		The Kirchhoff--Routh function is denoted by $K_{\Omega,W}$, with $K_{\Omega,W}(a)$ its restriction to the symmetry axis. We distinguish the horizontal momentum $\mathcal P$ from the vortex residual $P$. The energy--momentum action is $\mathcal J_W=E+W\mathcal P$, while $\widetilde{\mathcal J}_W$ uses the physical vortex position and $\Phi$ denotes the reduced action. The forms $\mathfrak q,\mathfrak q_{\rm s}$ and the index $n_-$ refer to the Hessian analysis of Theorem~\ref{thm1-2}.
		
		\item In Section~\ref{sec7}, $\varepsilon$ denotes the physical vortex-core scale, $U$ the radial reference stream function, and $f(U)$ its vorticity profile. We use $\omega_{\varepsilon,t}$ for the vorticity of the wave indexed by $t$, and $p_\varepsilon(t)$ for its selected center; $t$ is not physical time. The normalized smooth-core regular drift is $\mathscr G_\varepsilon$, with $L_\varepsilon:=\partial_a\mathscr G_\varepsilon$ and $M_\varepsilon:=\partial_W\mathscr G_\varepsilon$. The corresponding solution arc is denoted by $\mathscr C_{\varepsilon,\mathscr K}$. The rescaled variable $X$ is one-dimensional in Section~\ref{sec3} and two-dimensional in Section~\ref{sec7}.
		
		\item We write $A\lesssim B$ when $A\le CB$, with $\gtrsim$ and $\asymp$ defined analogously. Constants may depend on the fixed physical parameters, Sobolev regularity, and stated configuration bounds. The limit $W\to0^+$ refers to the initial deep-vortex branch, while $\varepsilon\to0^+$ is taken on a fixed compact arc $\mathscr K$, with $f$, $U$, and $\Gamma$ unchanged.
	\end{itemize}
	
	\smallskip
	
	\subsection{Organization of the paper}
	Section~\ref{sec2} develops the analytic formulation of the coupled point-vortex problem and establishes the Fredholm properties of its linearization. Section~\ref{sec3} continues the small-speed branch of Part~I through analytic global continuation. Section~\ref{sec4} studies the possible loss of compactness and establishes the finite profile decomposition of bounded solution sequences. Section~\ref{sec5} characterizes simple speed folds through the reduced vortex drift, while Section~\ref{sec6} develops their variational interpretation using the energy--momentum Hessian and a Schur-complement reduction. Finally, Section~\ref{sec7} constructs smooth-core solution arcs through a joint two-parameter Lyapunov--Schmidt reduction and proves the persistence of their specified simple folds.
	
	\bigskip

	\section{Analytic formulation and linearization}
    \label{sec2}
    We recall the conformal formulation of the point-vortex problem and develop the analytic and Fredholm properties needed for global continuation.

    \subsection{The coupled point-vortex equations}
    Fix $s>\frac12$ and use the spaces $\mathcal X^s$, $\mathcal Y^s$, and the configuration set $\mathscr U^s$ defined in Section~\ref{sec1}. Continuity of $|D|:H^{s+2}(\mathbb R)\to L^\infty(\mathbb R)$ makes $\mathscr U^s$ open. The Fredholm analysis concerns $0<W<c_{\min}$, although the reconstruction below is valid for every $W>0$.

    For $Y\in\mathcal O^s$, write $Z_Y(z)=z+\Phi_Y(z)$, where
    \[
    \Phi_Y(x+\mathrm{i}y):=-e^{y|D|}HY(x)+\mathrm{i}e^{y|D|}Y(x),\qquad y<0.
    \]
    The graph condition gives $\operatorname{Re}Z_Y'\ge c_Y>0$ on $\overline{\Pi_-}$, and $Z_Y$ is univalent by \cite[Lemma~2.1]{DuZouPartI}; see also \cite{Pommerenke1992}. Its boundary trace and reflection symmetry satisfy
    \begin{equation}
        Z_Y(x)=\Xi_Y(x)+\mathrm{i}Y(x),\qquad \Xi_Y(x):=x-HY(x),\qquad Z_Y(-\overline z)=-\overline{Z_Y(z)}.
        \label{2-1}
    \end{equation}
    For each fixed $Y$, both $Z_Y(z)-z\to0$ and $Z_Y'(z)\to1$ at infinity; no uniformity over families of surfaces is assumed.

    We use the boundary metric $J(Y)$, boundary velocity $q_{a,W}$, and Poisson evaluation $E_a$ from Section~\ref{sec1}. The signed curvature is
    \begin{equation}
        \kappa(Y)=-\frac{(1+|D|Y)Y''-Y'(|D|Y)'}{J(Y)^{3/2}}.
        \label{2-2}
    \end{equation}
    For a vortex at $z_a=-\mathrm{i}a$, with physical position $p_a=Z_Y(z_a)$, the Bernoulli and regular-drift residuals are
    \begin{equation}
        B(Y,a,W):=gY+\sigma\kappa(Y)+\frac{q_{a,W}^2}{2J(Y)}-\frac{W^2}{2}
        \label{2-3}
    \end{equation}
    and
    \begin{equation}
        P(Y,a,W):=W-\frac{\Gamma}{4\pi a}-\frac{\Gamma}{4\pi}\frac{E_a(Y'')}{1+E_a(|D|Y)}.
        \label{2-4}
    \end{equation}
    The graph condition gives $J(Y)\ge c_Y^2$ and $1+E_a(|D|Y)\ge c_Y$. Since $J(Y)^{-1}-1$ and $q_{a,W}-W$ belong to $H^{s+1}(\mathbb R)$, the cancellation
    \[
    \frac{q_{a,W}^2}{2J(Y)}-\frac{W^2}{2}=\frac{q_{a,W}^2-W^2}{2J(Y)}+\frac{W^2}{2}\bigl(J(Y)^{-1}-1\bigr)
    \]
    and Sobolev product estimates give $B(Y,a,W)\in H_{\rm e}^s(\mathbb R)$. The Poisson evaluations in $P$ are bounded for $a>0$. Thus the coupled residual defines a map
    \begin{equation}
        \mathscr F:\mathscr U^s\longrightarrow\mathcal Y^s,\qquad \mathscr F(Y,a,W):=\begin{pmatrix}B(Y,a,W)\\[1mm]P(Y,a,W)\end{pmatrix}.
        \label{2-5}
    \end{equation}
    The following proposition identifies this system with the point-vortex traveling-wave problem.

    \begin{proposition}
        \label{prop2-1}
        For $(Y,a,W)\in\mathscr U^s$, set
        \[
        \psi_a(x,y):=\frac{\Gamma}{4\pi}\log\frac{x^2+(y-a)^2}{x^2+(y+a)^2},\qquad \Psi_{a,W}(x,y):=\psi_a(x,y)+Wy,
        \]
        and define
        \[
        \widetilde\Psi_{Y,a,W}:=\Psi_{a,W}\circ Z_Y^{-1},\qquad v_{\rm rel}:=-\nabla^\perp\widetilde\Psi_{Y,a,W},\qquad v:=v_{\rm rel}-We_1.
        \]
        The reconstructed flow has vorticity $\Gamma\delta_{p_a}$, satisfies the kinematic boundary condition, and obeys $v(p)\to0$ as $|p|\to\infty$ in $\Omega_Y$. Writing the regular relative velocity in complex form as $\mathcal V_{\rm reg}=v_{{\rm reg},1}-\mathrm{i}v_{{\rm reg},2}$, we have
        \begin{equation}
            |v_{\rm rel}(Z_Y(x))|^2=\frac{q_{a,W}(x)^2}{J(Y)(x)},\qquad \mathcal V_{\rm reg}(p_a)=\frac{P(Y,a,W)}{Z_Y'(-\mathrm{i}a)}.
            \label{2-6}
        \end{equation}
        The reconstructed flow is a point-vortex traveling wave with velocity $-We_1$ if and only if
        \begin{equation}
            \mathscr F(Y,a,W)=0.
            \label{2-7}
        \end{equation}
        Conversely, every symmetric point-vortex traveling wave in this conformal graph class, with vortex preimage $-\mathrm{i}a$ and fluid velocity vanishing at infinity, admits this representation.
    \end{proposition}

    \begin{proof}
        The reference stream function vanishes on $\mathbb R$ and satisfies $-\Delta\Psi_{a,W}=\Gamma\delta_{z_a}$. Under the conformal map, $-\Delta_p\widetilde\Psi_{Y,a,W}=\Gamma\delta_{p_a}$, and the reconstructed velocity is divergence-free and tangent to the free surface. Since $\nabla\psi_a\to0$ and $Z_Y'\to1$ at infinity, we also have $v\to0$. On $y=0$, the identities $\partial_x\Psi_{a,W}=0$ and $\partial_y\Psi_{a,W}=q_{a,W}$ give the first formula in \eqref{2-6}. By Bernoulli's law, the capillary pressure condition is then equivalent to $B=0$.

        For the vortex equilibrium, \cite[Lemma~2.3]{DuZouPartI} gives
        \[
        \mathcal V_{\rm reg}(p_a)=\frac1{Z_Y'(-\mathrm{i}a)}\left[W-\frac{\Gamma}{4\pi a}+\frac{\mathrm{i}\Gamma}{4\pi}\frac{Z_Y''(-\mathrm{i}a)}{Z_Y'(-\mathrm{i}a)}\right].
        \]
        Using $Z_Y'(-\mathrm{i}a)=1+E_a(|D|Y)>0$ and $Z_Y''(-\mathrm{i}a)=\mathrm{i}E_a(Y'')$, we obtain the second formula in \eqref{2-6}. Thus $P=0$ is equivalent to the stationarity of the vortex, proving \eqref{2-7}.

        Conversely, let $\Psi^\#$ be the pullback of the relative stream function of a symmetric point-vortex traveling wave, normalized to vanish on $\mathbb R$. Then $R:=\Psi^\#-\psi_a-Wy$ is harmonic in $\Pi_-$, since the point-vortex singularities cancel, and satisfies $R=0$ on $\mathbb R$ and $\nabla R\to0$ at infinity. Odd reflection across $\mathbb R$ extends $R$ to an entire harmonic function with bounded gradient. By Liouville's theorem, $\nabla R=0$, and hence $R=0$. The stated reconstruction follows, and the Bernoulli and vortex equilibrium conditions give $B=P=0$.
    \end{proof}

    \smallskip

    \subsection{Analyticity of the coupled map}
    We record analyticity of the coupled residual and derivative bounds uniform over controlled configurations.

    \begin{lemma}
        \label{lem2-2}
        Let $s>\frac12$. The map $\mathscr F:\mathscr U^s\to\mathcal Y^s$ is real analytic. Fix $R,c_0>0$, $0<a_-<a_+<\infty$, and $[W_-,W_+]\Subset(0,c_{\min})$. For every integer $k\ge1$, there exists a constant $C_k$, depending only on $s,k,R,c_0,a_\pm,W_\pm$ and the fixed physical parameters, such that
        \begin{equation}
            \|\mathscr F(Y,a,W)\|_{\mathcal Y^s}+\|D^k\mathscr F(Y,a,W)\|_{\mathcal L^k(\mathcal X^s\times\mathbb R,\mathcal Y^s)}\le C_k
            \label{2-8}
        \end{equation}
        whenever
        \[
        \|Y\|_{H^{s+2}}\le R,\qquad \inf_{\mathbb R}(1+|D|Y)\ge c_0,\qquad (a,W)\in[a_-,a_+]\times[W_-,W_+].
        \]
    \end{lemma}

    \begin{proof}
        Since $s>\frac12$, Sobolev product estimates \cite{BahouriCheminDanchin2011} show that $J(Y)-1$ and $\mathfrak m_Y:=(1+|D|Y)Y''-Y'(|D|Y)'$ are polynomial maps from $H^{s+2}(\mathbb R)$ into $H^{s+1}(\mathbb R)$ and $H^s(\mathbb R)$, respectively. For $\theta=1,\frac32$, the negative powers of $J$ are analytic in the algebra $\mathbb R\oplus H^{s+1}(\mathbb R)$. Indeed, for any $Y_\circ\in\mathcal O^s$, setting $J_\circ:=J(Y_\circ)$ gives the convergent expansion
        \[
        J(Y)^{-\theta}=J_\circ^{-\theta}\sum_{j=0}^{\infty}\binom{-\theta}{j}\bigl[J_\circ^{-1}(J(Y)-J_\circ)\bigr]^j
        \]
        for $Y$ sufficiently close to $Y_\circ$. Thus $Y\mapsto J(Y)^{-\theta}-1$ is analytic into $H^{s+1}(\mathbb R)$, and $\kappa(Y)=-\mathfrak m_YJ(Y)^{-3/2}$ is analytic into $H^s(\mathbb R)$.

        For the vortex terms, set $r_a(x):=\Gamma a/[\pi(x^2+a^2)]$, so that $q_{a,W}=W-r_a$. The Fourier representations
        \[
        \widehat{r_a}(\zeta)=\Gamma e^{-a|\zeta|},\qquad E_ah=\frac1{2\pi}\int_{\mathbb R}e^{-a|\zeta|}\widehat h(\zeta)\,d\zeta
        \]
        and the bound $|\zeta|^je^{-a|\zeta|}\le j!(2/a_-)^je^{-a_-|\zeta|/2}$ for $a\ge a_->0$ give
        \begin{equation}
            \|\partial_a^jr_a\|_{H^{s+1}}+\|\partial_a^jE_a\|_{\mathcal L(H^s,\mathbb R)}\le Cj!\left(\frac2{a_-}\right)^j,\qquad j\ge0.
            \label{2-9}
        \end{equation}
        Hence $a\mapsto r_a$ and $a\mapsto E_a$ are analytic in their respective spaces.

        Using the kinetic-term cancellation before Proposition~\ref{prop2-1}, we obtain joint analyticity of $B$ into $H_{\rm e}^s(\mathbb R)$. For $P$, the analyticity of $E_a$, together with $1+E_a(|D|Y)>0$, gives joint analyticity of the regular-drift residual. This proves real analyticity of $\mathscr F$ as a map into $\mathcal Y^s$.

        Under the stated uniform bounds, $J(Y)\ge c_0^2$ and $1+E_a(|D|Y)\ge c_0$. Sobolev product and composition estimates therefore control all surface derivatives of each fixed order, while \eqref{2-9} controls the depth derivatives. Polynomial dependence on $W$ and the product and chain rules now give \eqref{2-8}.
    \end{proof}

    \smallskip

    \subsection{The coupled linearization}
    For $(Y,a,W)\in\mathscr U^s$, write $q=q_{a,W}$, $J=J(Y)$, and $\mathfrak m_Y=(1+|D|Y)Y''-Y'(|D|Y)'$, and set $\alpha:=1+E_a(|D|Y)=Z_Y'(-\mathrm{i}a)>0$. Unless indicated otherwise, partial derivatives with respect to $a$ and $W$ keep the surface $Y$ fixed.

    \begin{lemma}
        \label{lem2-3}
        Let $s>\frac12$ and $(Y,a,W)\in\mathscr U^s$. For $(h,\ell,\nu)\in H_{\rm e}^{s+2}(\mathbb R)\times\mathbb R^2$,
        \begin{equation}
            D\mathscr F(Y,a,W)[h,\ell,\nu]=\begin{pmatrix}Ah+B_a\ell+B_W\nu\\[1mm]P_Y[h]+P_a\ell+\nu\end{pmatrix},
            \label{2-10}
        \end{equation}
        where $A=D_YB$ is given by
        \begin{equation}
            Ah=gh+\sigma D\kappa(Y)[h]-\frac{q^2}{J^2}\bigl[(1+|D|Y)|D|h+Y'h'\bigr].
            \label{2-11}
        \end{equation}
        Here
        \begin{equation}
            D\kappa(Y)[h]=-\frac{D\mathfrak m_Y[h]}{J^{3/2}}+\frac{3\mathfrak m_Y}{J^{5/2}}\bigl[(1+|D|Y)|D|h+Y'h'\bigr],
            \label{2-12}
        \end{equation}
        with $D\mathfrak m_Y[h]=(1+|D|Y)h''-Y'(|D|h)'+Y''|D|h-(|D|Y)'h'$. The parameter derivatives of $B$ are
        \begin{equation}
            B_a=\frac{q\,\partial_aq}{J},\qquad B_W=\frac qJ-W,\qquad \partial_aq(x)=\frac{\Gamma(a^2-x^2)}{\pi(x^2+a^2)^2}.
            \label{2-13}
        \end{equation}
        For the vortex residual,
        \begin{equation}
            P_Y[h]=-\frac{\Gamma}{4\pi}\left[\frac{E_a(h'')}{\alpha}-\frac{E_a(Y'')E_a(|D|h)}{\alpha^2}\right]
            \label{2-14}
        \end{equation}
        \begin{equation}
            P_a=\frac{\Gamma}{4\pi a^2}+\frac{\Gamma}{4\pi}\left[\frac{E_a(|D|Y'')}{\alpha}+\frac{\bigl(E_a(Y'')\bigr)^2}{\alpha^2}\right],\qquad P_W=1.
            \label{2-15}
        \end{equation}
        The maps $A:H_{\rm e}^{s+2}(\mathbb R)\to H_{\rm e}^s(\mathbb R)$ and $P_Y:H_{\rm e}^{s+2}(\mathbb R)\to\mathbb R$ are bounded, with $B_a,B_W\in H_{\rm e}^{s+1}(\mathbb R)$, uniformly under the bounds of Lemma~\ref{lem2-2}.
    \end{lemma}

    \begin{proof}
        Differentiating $J$ gives $DJ(Y)[h]=2[(1+|D|Y)|D|h+Y'h']$, and the product rule gives $D\mathfrak m_Y[h]$. These identities yield \eqref{2-11}--\eqref{2-12}, while direct differentiation of $q$ gives \eqref{2-13}. For $P$, the quotient rule and $D_Y\alpha[h]=E_a(|D|h)$ give \eqref{2-14}. Since $\partial_aE_a=-E_a|D|$ and $|D|^2=-\partial_x^2$,
        \[
        \partial_a\alpha=E_a(Y''),\qquad \partial_a(E_a(Y''))=-E_a(|D|Y''),
        \]
        which proves \eqref{2-15} and hence \eqref{2-10}.

        The asserted bounds follow from Lemma~\ref{lem2-2} and positivity of the denominators. In particular, $B_W=(q-W)/J+W(J^{-1}-1)\in H_{\rm e}^{s+1}$. The factor $|\zeta|^3e^{-a|\zeta|}$ in the Fourier representation of $E_a(|D|Y'')$ is bounded as a functional on $H^{s+2}$ for $a>0$, uniformly when $a\ge a_->0$; so no additional regularity of $Y$ is needed.
    \end{proof}

    Separate the principal part of $A$ by writing $A=A_2+A_1$, where
    \begin{equation}
        A_2h=-\frac{\sigma(1+|D|Y)}{J^{3/2}}h''+\frac{\sigma Y'}{J^{3/2}}(|D|h)'.
        \label{2-16}
    \end{equation}
    The remainder $A_1:H_{\rm e}^{s+1}(\mathbb R)\to H_{\rm e}^s(\mathbb R)$ is bounded, uniformly under the same configuration bounds. Both terms in $A_2$ are of second order.

    \begin{lemma}
        \label{lem2-4}
        Let $s>\frac12$ and $(Y,a,W)\in\mathscr U^s$. The homogeneous second-order symbol of $A$ is
        \begin{equation}
            p_2(x,\zeta)=\frac{\sigma}{J(Y)(x)^{3/2}}\left[(1+|D|Y(x))\zeta^2+\mathrm{i}Y'(x)\zeta|\zeta|\right],\qquad \zeta\ne0,
            \label{2-17}
        \end{equation}
        and satisfies
        \begin{equation}
            \operatorname{Re}p_2(x,\zeta)\ge c|\zeta|^2,\qquad x\in\mathbb R,\quad\zeta\ne0,
            \label{2-18}
        \end{equation}
        where $c>0$ depends only on $s,\sigma,R,c_0$ when $\|Y\|_{H^{s+2}}\le R$ and $1+|D|Y\ge c_0>0$.

        Extend $A$ to $H^{s+2}(\mathbb R)$ by the same formula, and let $(\tau_rh)(x):=h(x-r)$. Define
        \begin{equation}
            A_\infty(W):=g-\sigma\partial_x^2-W^2|D|.
            \label{2-19}
        \end{equation}
        Then, for each fixed $h\in H^{s+2}(\mathbb R)$,
        \begin{equation}
            \|\tau_{-r}A\tau_rh-A_\infty(W)h\|_{H^s}\longrightarrow0\qquad\text{as }|r|\to\infty.
            \label{2-20}
        \end{equation}
    \end{lemma}

    \begin{proof}
        The symbol follows from \eqref{2-16}, and
        \[
        \operatorname{Re}p_2(x,\zeta)=\frac{\sigma(1+|D|Y(x))}{J(Y)(x)^{3/2}}|\zeta|^2\ge\frac{\sigma c_Y}{\|J(Y)\|_{L^\infty}^{3/2}}|\zeta|^2,
        \]
        which proves \eqref{2-18} by Sobolev embedding. Writing $Ah=A_2h+\mathfrak c|D|h+\mathfrak d h'+gh$, we obtain
        \[
        \begin{aligned}
            \mathfrak c&=-\frac{\sigma Y''}{J^{3/2}}+\frac{3\sigma\mathfrak m_Y(1+|D|Y)}{J^{5/2}}-\frac{q^2(1+|D|Y)}{J^2},\\
            \mathfrak d&=\frac{\sigma(|D|Y)'}{J^{3/2}}+\frac{3\sigma\mathfrak m_Y Y'}{J^{5/2}}-\frac{q^2Y'}{J^2}.
        \end{aligned}
        \]
        Since $q^2-W^2\in H^{s+1}$, the coefficient differences satisfy
        \[
        \frac{\sigma(1+|D|Y)}{J^{3/2}}-\sigma,\quad \frac{\sigma Y'}{J^{3/2}}\in H^{s+1}(\mathbb R),\qquad \mathfrak c+W^2,\quad\mathfrak d\in H^s(\mathbb R).
        \]
        For $\rho,f\in H^s(\mathbb R)$, $s>\frac12$, density of $C_c^\infty(\mathbb R)$ and the Sobolev algebra estimate give
        \begin{equation}
            \|\rho(\cdot+r)f\|_{H^s}\longrightarrow0\qquad\text{as }|r|\to\infty,\qquad \rho,f\in H^s(\mathbb R).
            \label{2-21}
        \end{equation}
        Each term of $A-A_\infty(W)$ is a coefficient from the displayed spaces multiplied by one of $\partial_x^2$, $\partial_x|D|$, $|D|$, or $\partial_x$. Since these operators commute with translations and map $H^{s+2}$ into $H^s$, \eqref{2-21} gives \eqref{2-20}.
    \end{proof}

    The convergence in \eqref{2-20} is strong but need not hold in operator norm. In particular, the variable second-order terms need not be compact perturbations of $A_\infty(W)$; we handle them by the conformal factorization below.

    The far-field Fourier symbol is
    \begin{equation}
        m_W(\zeta):=g+\sigma\zeta^2-W^2|\zeta|.
        \label{2-22}
    \end{equation}

    \begin{lemma}
        \label{lem2-5}
        Let $r\ge0$ and $0\le W<c_{\min}$. The operator $A_\infty(W):H^{r+2}(\mathbb R)\to H^r(\mathbb R)$ is an isomorphism, with inverse
        \begin{equation}
            \widehat{A_\infty(W)^{-1}f}(\zeta)=\frac{\widehat f(\zeta)}{g+\sigma\zeta^2-W^2|\zeta|}.
            \label{2-23}
        \end{equation}
        For every $0\le W_+<c_{\min}$, there is $C>0$, depending only on $g,\sigma,W_+$, such that
        \begin{equation}
            C^{-1}\|h\|_{H^{r+2}}\le\|A_\infty(W)h\|_{H^r}\le C\|h\|_{H^{r+2}},\qquad 0\le W\le W_+.
            \label{2-24}
        \end{equation}
        The same conclusions hold on the even and odd subspaces, including the isomorphism $H^2\to L^2$ when $r=0$.
    \end{lemma}

    \begin{proof}
        As $c_{\min}^2=2\sqrt{g\sigma}$, we have $c_{\min}^2|\zeta|\le g+\sigma\zeta^2$, and hence
        \begin{equation}
            \left(1-\frac{W^2}{c_{\min}^2}\right)(g+\sigma\zeta^2)\le m_W(\zeta)\le g+\sigma\zeta^2.
            \label{2-25}
        \end{equation}
        For $0\le W\le W_+<c_{\min}$, the multiplier $m_W$ is uniformly comparable to $1+\zeta^2$. Fourier integration gives \eqref{2-24}, and division by $m_W$ maps $H^r$ into $H^{r+2}$, proving \eqref{2-23} and invertibility. Since the multiplier is real and even, its inverse preserves both parity subspaces.
    \end{proof}

    At $W=c_{\min}$, $m_W(\zeta)=\sigma(|\zeta|-\sqrt{g/\sigma})^2$, while for $W>c_{\min}$ it vanishes at $|\zeta|=\rho_\pm:=(W^2\pm\sqrt{W^4-4g\sigma})/(2\sigma)$. In either case, choose a positive zero $\rho_0$ and a nonzero real even $\chi\in C_c^\infty(-1,1)$. With $c_n$ chosen so that $\|h_n\|_{H^{s+2}}=1$, set
    \[
    \widehat h_n(\zeta):=c_n n^{1/2}\bigl[\chi(n(\zeta-\rho_0))+\chi(n(\zeta+\rho_0))\bigr].
    \]
    Then $\|A_\infty(W)h_n\|_{H^s}\to0$, since the Fourier supports shrink to $\pm\rho_0$. As $m_W$ vanishes at only finitely many frequencies, the kernel remains trivial; the bounded inverse theorem therefore rules out closed range. Thus $A_\infty(W)$ is not Fredholm for $W\ge c_{\min}$.

    We turn to a factorization of the variable second-order terms to obtain Fredholmness in the subcritical range.

    \begin{lemma}
        \label{lem2-6}
        Let $s>\frac12$ and $(Y,a,W)\in\mathscr U^s$. Then $A=D_YB(Y,a,W):H_{\rm e}^{s+2}(\mathbb R)\to H_{\rm e}^s(\mathbb R)$ is Fredholm of index zero.
    \end{lemma}

    \begin{proof}
        On the full real Sobolev spaces, define $\mathcal S_Yf:=(1+|D|Y)f-Y'Hf$. For $f\in H^s$, the boundary value $f+\mathrm{i}Hf$ extends holomorphically into $\Pi_-$. Since $Z_Y'$ and $1/Z_Y'$ are bounded and holomorphic, multiplication and division preserve the corresponding Hardy class. Taking real parts, with the $L^2$ normalization ruling out an imaginary constant, gives
        \[
        \mathcal S_Yf+\mathrm{i}H(\mathcal S_Yf)=Z_Y'(f+\mathrm{i}Hf),\qquad \mathcal S_Y^{-1}f=\frac{(1+|D|Y)f+Y'Hf}{J}.
        \]
        Define
        \begin{equation}
            \mathcal T_Yf:=J^{-3/2}\bigl[(1+|D|Y)f+Y'Hf\bigr].
            \label{2-26}
        \end{equation}
        Then $\mathcal T_Y=J^{-1/2}\mathcal S_Y^{-1}$ and
        \[
        \mathcal T_Y^{-1}f=(1+|D|Y)J^{1/2}f-Y'H(J^{1/2}f).
        \]
        Both operators are bounded on $H^s$, uniformly when $\|Y\|_{H^{s+2}}\le R$ and $1+|D|Y\ge c_0>0$.

        Using $(|D|h)'=-Hh''$, we obtain $A_2=\mathcal T_Y(-\sigma\partial_x^2)$ and hence
        \begin{equation}
            A=\mathcal T_YA_\infty(W)+\mathcal K_{Y,a,W},
            \label{2-27}
        \end{equation}
        where
        \[
        \mathcal K_{Y,a,W}:=(\mathfrak c+W^2)|D|+\mathfrak d\,\partial_x+(I-\mathcal T_Y)(gI-W^2|D|).
        \]
        For $s>\frac12$ and $\rho\in H^s(\mathbb R)$, multiplication by $\rho$ is compact from $H^{s+1}(\mathbb R)$ to $H^s(\mathbb R)$: approximate $\rho$ in $H^s$ by compactly supported smooth functions, apply local Rellich compactness, and use the Sobolev product estimate for operator-norm convergence. Since $\mathfrak c+W^2,\mathfrak d\in H^s$, this proves compactness of the first two terms in $\mathcal K_{Y,a,W}$ from $H^{s+2}$ to $H^s$. Moreover,
        \[
        \mathcal T_Y-I=\rho_0I+\rho_1H,\qquad \rho_0:=\frac{1+|D|Y}{J^{3/2}}-1,\qquad \rho_1:=\frac{Y'}{J^{3/2}},
        \]
        with $\rho_0,\rho_1\in H^{s+1}$. The same argument makes $I-\mathcal T_Y:H^{s+1}\to H^s$ compact, and therefore also the last term in $\mathcal K_{Y,a,W}$.

        By Lemma~\ref{lem2-5}, $\mathcal T_YA_\infty(W)$ is an isomorphism $H^{s+2}\to H^s$, so \eqref{2-27} proves that $A$ is Fredholm of index zero. Since the coefficients of $\mathcal T_Y$ have the appropriate parity and $H$ exchanges even and odd functions, $\mathcal T_Y$ and its inverse preserve evenness. The factorization and index conclusion thus hold on the even subspaces.
    \end{proof}

    The vortex equation contributes only finite-dimensional blocks to the surface linearization.

    \begin{lemma}
        \label{lem2-7}
        Let $s>\frac12$ and $(Y,a,W)\in\mathscr U^s$. The fixed-speed linearization
        \begin{equation}
            \mathscr L:=D_{(Y,a)}\mathscr F(Y,a,W)=\begin{pmatrix}A&B_a\\[1mm]P_Y&P_a\end{pmatrix}:\mathcal X^s\longrightarrow\mathcal Y^s
            \label{2-28}
        \end{equation}
        is Fredholm of index zero. The full derivative $\mathscr D:=D\mathscr F(Y,a,W):\mathcal X^s\times\mathbb R\to\mathcal Y^s$ is Fredholm of index one.
    \end{lemma}

    \begin{proof}
        By Lemma~\ref{lem2-6}, the operator $\mathscr L_{\rm ref}(h,\ell):=(Ah,\ell)$ is Fredholm of index zero. Since $\mathscr L-\mathscr L_{\rm ref}$ has finite rank, $\mathscr L$ also has index zero. The operator $\mathscr D_{\rm ref}(h,\ell,\nu):=\mathscr L(h,\ell)$ has index one, and $\mathscr D-\mathscr D_{\rm ref}$ is the rank-one map $(h,\ell,\nu)\mapsto\nu(B_W,1)$. Thus $\mathscr D$ has index one.
    \end{proof}

    If $\mathscr D$ is surjective at a solution, its one-dimensional kernel and the analytic implicit function theorem yield a regular local solution curve. The fixed-speed derivative $\mathscr L$ may nevertheless be singular. The next lemma makes this distinction precise on Bernoulli charts.

    \begin{lemma}
        \label{lem2-8}
        Let $s>\frac12$ and $(Y,a,W)\in\mathscr U^s$, with $A=D_YB(Y,a,W)$ invertible. Define
        \begin{equation}
            L:=P_a-P_YA^{-1}B_a,\qquad M:=1-P_YA^{-1}B_W.
            \label{2-29}
        \end{equation}
        Then $\mathscr L$ is invertible if and only if $L\ne0$. When $L=0$,
        \begin{equation}
            \ker\mathscr L=\operatorname{span}\{(-A^{-1}B_a,1)\},\qquad \operatorname{Ran}\mathscr L=\ker\Lambda,
            \label{2-30}
        \end{equation}
        where $\Lambda(f,\rho):=\rho-P_YA^{-1}f$ is a bounded functional on $\mathcal Y^s$. The full derivative $\mathscr D$ is surjective if and only if $(L,M)\ne(0,0)$, in which case
        \begin{equation}
            \ker\mathscr D=\operatorname{span}\left\{\bigl(-A^{-1}(MB_a-LB_W),M,-L\bigr)\right\}.
            \label{2-31}
        \end{equation}
        Near a coupled solution $(Y_\star,a_\star,W_\star)$ with $A_\star$ invertible, the Bernoulli equation determines a unique local real-analytic map $Y=Y(a,W)$. The nearby coupled solutions are precisely the zeros of
        \begin{equation}
            \mathscr G_0(a,W):=P(Y(a,W),a,W),
            \label{2-32}
        \end{equation}
        with
        \begin{equation}
            \partial_a\mathscr G_0=L,\qquad \partial_W\mathscr G_0=M,
            \label{2-33}
        \end{equation}
        where $L,M$ are evaluated at $(Y(a,W),a,W)$.
    \end{lemma}

    \begin{proof}
        Solving $\mathscr L(h,\ell)=(f,\rho)$ for $h$ gives
        \[
        h=A^{-1}f-A^{-1}B_a\ell,\qquad L\ell=\Lambda(f,\rho).
        \]
        If $L\ne0$, the bounded inverse is
        \[
        \mathscr L^{-1}(f,\rho)=(A^{-1}f,0)+\frac{\Lambda(f,\rho)}{L}(-A^{-1}B_a,1).
        \]
        For $L=0$, solvability requires $\Lambda(f,\rho)=0$, and homogeneous solutions are multiples of $(-A^{-1}B_a,1)$. Since $\Lambda(0,1)=1$, this proves \eqref{2-30} and shows that the range has codimension one.

        Eliminating the surface variable in the full derivative yields
        \[
        h=A^{-1}(f-B_a\ell-B_W\nu),\qquad L\ell+M\nu=\Lambda(f,\rho).
        \]
        The last equation is solvable for every right-hand side if and only if $(L,M)\ne(0,0)$. Its homogeneous solutions satisfy $(\ell,\nu)=t(M,-L)$, giving \eqref{2-31}.

        At a coupled solution with $A_\star$ invertible, Lemma~\ref{lem2-2} and the analytic implicit function theorem give $Y=Y(a,W)$ and \eqref{2-32}. Differentiating the Bernoulli equation yields
        \[
        Y_a=-A^{-1}B_a,\qquad Y_W=-A^{-1}B_W,
        \]
        so the chain rule and $P_W=1$ prove \eqref{2-33}.
    \end{proof}

    Along any differentiable solution curve in a Bernoulli chart,
    \begin{equation}
        L\dot a+M\dot W=0.
        \label{2-34}
    \end{equation}
    Thus $W$ can be used as a local parameter where $L\ne0$, while $a$ can be used where $M\ne0$. In particular, a regular solution curve may have stationary speed when $L=0$ and $M\ne0$; the additional condition $\partial_aL\ne0$ characterizes the simple folds considered in Section~\ref{sec5}.

    \bigskip

    \section{The deep-vortex end and global continuation}
    \label{sec3}

    For the small-speed branch of Part~I, we establish invertibility of the fixed-speed coupled linearization $D_{(Y,a)}\mathscr F$ in the symmetric spaces and determine how vortex depth varies with speed. Reciprocal depth then brings the singular deep-vortex end to a finite boundary point suitable for analytic global continuation.

    \subsection{The deep-vortex end}
    Fix $s>\frac12$, and let $(Y_W,a_W)$, $0<W<W_0$, be the small-speed family of Part~I. Set $a_0(W):=\Gamma/(4\pi W)$ and $\rho_W:=4W^2/g$, and write
    \begin{equation}
    	a_W=a_0(W)(1+\lambda_W),\qquad Y_W(x)=\rho_Wu_W\left(\frac{x}{a_0(W)}\right).
    	\label{3-1}
    \end{equation}
    By \cite[Proposition~3.9 and Corollary~3.10]{DuZouPartI}, after decreasing $W_0$ if necessary,
    \begin{equation}
    	\|u_W\|_{H^{s+2}}\le C_s,\qquad \lambda_W=O(W^3),\qquad a_W=\frac{\Gamma}{4\pi W}+\frac{3}{2g}W^2+O(W^4).
    	\label{3-2}
    \end{equation}
    The norm of $u_W$ is taken in the rescaled coordinate $X=x/a_0(W)$, whereas the linearizations below act in the original coordinate $x$. All constants are uniform for sufficiently small $W$, with $s,g,\sigma,\Gamma$ fixed.

    \begin{lemma}
    	\label{lem3-1}
    	There exist $W_{\rm loc}\in(0,\min\{W_0,c_{\min}\})$ and $C_s>0$ such that, for $0<W<W_{\rm loc}$, $A_W:=D_YB(Y_W,a_W,W):H_{\rm e}^{s+2}(\mathbb R)\to H_{\rm e}^s(\mathbb R)$ is an isomorphism with
    	\begin{equation}
    		\|A_W^{-1}\|_{\mathcal L(H_{\rm e}^s,H_{\rm e}^{s+2})}\le C_s.
    		\label{3-3}
    	\end{equation}
    	The reduced derivatives from \eqref{2-29} satisfy
    	\begin{equation}
    		L(W)=\frac{4\pi}{\Gamma}W^2+O(W^5),\qquad M(W)=1+O(W^3).
    		\label{3-4}
    	\end{equation}
    	Consequently, $D_{(Y,a)}\mathscr F(Y_W,a_W,W)$ is invertible, the map $W\mapsto(Y_W,a_W)$ is real analytic on $(0,W_{\rm loc})$, and
    	\begin{equation}
    		\frac{da_W}{dW}=-\frac{\Gamma}{4\pi W^2}\bigl(1+O(W^3)\bigr)<0.
    		\label{3-5}
    	\end{equation}
    \end{lemma}

    \begin{proof}
    	From \eqref{3-1}--\eqref{3-2} and $a_0\asymp W^{-1}$, the Sobolev scaling estimates give
    	\[
    	\begin{aligned}
    		\|Y_W\|_{H^{s+2}}&=O(W^{3/2}),\\
    		\||D|Y_W\|_{H^{s+1}}+\|Y_W'\|_{H^{s+1}}&=O(W^{5/2}),\\
    		\|Y_W''\|_{H^s}+\|(|D|Y_W)'\|_{H^s}&=O(W^{7/2}).
    	\end{aligned}
    	\]
    	For $r_\lambda(X):=1-4(1+\lambda)/(X^2+(1+\lambda)^2)$, we have $q_{a_W,W}(a_0X)=Wr_{\lambda_W}(X)$. Since $r_{\lambda_W}^2-1$ is uniformly bounded in $H^{s+1}$, $\|q_{a_W,W}^2-W^2\|_{H^s}=O(W^{3/2})$. Substituting into \eqref{2-11}--\eqref{2-12} gives
    	\begin{equation}
    		\|A_W-A_\infty(W)\|_{\mathcal L(H_{\rm e}^{s+2},H_{\rm e}^s)}\le C_sW^{3/2}.
    		\label{3-6}
    	\end{equation}
    	By Lemma~\ref{lem2-5}, $T_W:=A_\infty(W)^{-1}(A_W-A_\infty(W))$ has norm less than $1/2$ on $H_{\rm e}^{s+2}$ for small $W$. A Neumann series for $I+T_W$ proves \eqref{3-3}.

    	Set $\alpha_W:=1+E_{a_W}(|D|Y_W)$. The scaling identity
    	\[
    	E_{a_W}(|D|^kY_W)=\rho_Wa_0^{-k}\bigl(e^{-(1+\lambda_W)|D_X|}|D_X|^ku_W\bigr)(0)
    	\]
    	and the uniform positivity of $1+\lambda_W$ give $\alpha_W=1+O(W^3)$, $E_{a_W}(Y_W'')=O(W^4)$, and $E_{a_W}(|D|Y_W'')=O(W^5)$. Hence \eqref{2-15} yields
    	\begin{equation}
    		P_a=\frac{\Gamma}{4\pi a_W^2}+O(W^5)=\frac{4\pi}{\Gamma}W^2+O(W^5).
    		\label{3-7}
    	\end{equation}
    	The Fourier representation and Cauchy--Schwarz yield
\[
|E_a(|D|^kh)|\le C_ka^{-k-1/2}\|h\|_{L^2},\qquad a>0,\quad k\ge0.
\]
Consequently, \eqref{2-14} extends $P_Y$ to $L^2$ and gives $\|P_Y\|_{\mathcal L(L^2,\mathbb R)}\le C(W^{5/2}+W^4a_W^{-3/2})\le CW^{5/2}$. When differentiating $q_{a,W}$ with respect to $a$, we hold $W$ and $a_0(W)$ fixed. Since $\partial_aq_{a,W}(a_0X)=(W/a_0)\partial_\lambda r_\lambda(X)$ at $\lambda=\lambda_W$, the uniform $H^s$ bounds for $r_\lambda\partial_\lambda r_\lambda$ and $r_\lambda-1$, together with \eqref{2-13} and $B_W=(q-W)/J+W(J^{-1}-1)$, give
    	\[
    	\|B_a\|_{H^s}=O(W^{5/2}),\qquad \|B_W\|_{H^s}=O(W^{1/2}).
    	\]
    	Together with \eqref{3-3}, these estimates yield
    	\[
    	|P_YA_W^{-1}B_a|=O(W^5),\qquad |P_YA_W^{-1}B_W|=O(W^3).
    	\]
    	Equations \eqref{2-29} and \eqref{3-7} now prove \eqref{3-4}.

    	After decreasing $W_{\rm loc}$, we have $L(W)>0$ and $M(W)>0$. Lemma~\ref{lem2-8} then makes the fixed-speed coupled derivative invertible. By the analytic implicit function theorem and the local uniqueness of the continuous Part~I family, $W\mapsto(Y_W,a_W)$ is real analytic. Differentiating the coupled equation along this branch, rather than differentiating the remainder in \eqref{3-2}, gives $da_W/dW=-M(W)/L(W)$ by \eqref{2-34}, and \eqref{3-4} proves \eqref{3-5}.
    \end{proof}

    Estimate \eqref{3-3} controls the surface inverse uniformly, but $L(W)\to0$ precludes the analogous bound for the full fixed-speed inverse in the depth coordinate $a$. Set $b_W:=a_W^{-1}$. Then \eqref{3-2} and the surface estimate above give
    \begin{equation}
    	\|Y_W\|_{H^{s+2}}=O(W^{3/2}),\qquad b_W=\frac{4\pi}{\Gamma}W+O(W^4).
    	\label{3-8}
    \end{equation}
    Thus $(Y_W,b_W,W)\to(0,0,0)$ in $H_{\rm e}^{s+2}(\mathbb R)\times\mathbb R^2$ as $W\to0^+$; the deep-vortex end is now a finite boundary point.

   	\smallskip

   	\subsection{Reciprocal depth and the initial arc}
   	Introduce the reciprocal depth $b:=a^{-1}$ and the admissible set $\widehat{\mathscr U}^{\,s}:=\mathcal O^s\times(0,\infty)\times(0,c_{\min})$ with coordinates $(Y,b,W)$. The analytic change of variables $\mathfrak C(Y,b,W):=(Y,b^{-1},W)$ gives
   	\begin{equation}
   		\widehat{\mathscr F}(Y,b,W):=\mathscr F(Y,b^{-1},W)=\begin{pmatrix}\widehat B(Y,b,W)\\[1mm]\widehat P(Y,b,W)\end{pmatrix}.
   		\label{3-9}
   	\end{equation}
   	With $\widehat q_{b,W}(x):=W-\Gamma b/[\pi(1+b^2x^2)]$, the transformed residuals are
   	\[
   	\begin{aligned}
   		\widehat B(Y,b,W)&=gY+\sigma\kappa(Y)+\frac{\widehat q_{b,W}^2}{2J(Y)}-\frac{W^2}{2},\\
   		\widehat P(Y,b,W)&=W-\frac{\Gamma b}{4\pi}-\frac{\Gamma}{4\pi}\frac{E_{1/b}(Y'')}{1+E_{1/b}(|D|Y)}.
   	\end{aligned}
   	\]

   	\begin{lemma}
   		\label{lem3-2}
   		The map $\widehat{\mathscr F}:\widehat{\mathscr U}^{\,s}\to\mathcal Y^s$ is real analytic, and $D_{(Y,b)}\widehat{\mathscr F}:\mathcal X^s\to\mathcal Y^s$ is Fredholm of index zero throughout $\widehat{\mathscr U}^{\,s}$. After decreasing $W_{\rm loc}$ if necessary, the initial arc
   		\begin{equation}
   			\widehat{\mathscr C}_{\rm loc}:=\{(Y_W,b_W,W):0<W<W_{\rm loc}\},\qquad b_W:=a_W^{-1},
   			\label{3-10}
   		\end{equation}
   		is real analytic and consists of zeros of $\widehat{\mathscr F}$, with $(Y_W,b_W,W)\to(0,0,0)\in\partial\widehat{\mathscr U}^{\,s}$ as $W\to0^+$. Along this arc, $\widehat{\mathscr L}_W:=D_{(Y,b)}\widehat{\mathscr F}(Y_W,b_W,W)$ satisfies
   		\begin{equation}
   			\widehat{\mathscr L}_W\longrightarrow\widehat{\mathscr L}_0:=\begin{pmatrix}g-\sigma\partial_x^2&0\\[1mm]0&-\Gamma/(4\pi)\end{pmatrix}\qquad\text{in }\mathcal L(\mathcal X^s,\mathcal Y^s),
   			\label{3-11}
   		\end{equation}
   		and
   		\begin{equation}
   			\|\widehat{\mathscr L}_W^{-1}\|_{\mathcal L(\mathcal Y^s,\mathcal X^s)}\le C_s,\qquad 0<W<W_{\rm loc}.
   			\label{3-12}
   		\end{equation}
   	\end{lemma}

   	\begin{proof}
   		Lemma~\ref{lem2-2} and the analytic change of variables $\mathfrak C$ give analyticity of the transformed residual. For $b>0$, the chain rule gives
   		\begin{equation}
   			D_{(Y,b)}\widehat{\mathscr F}(Y,b,W)=D_{(Y,a)}\mathscr F(Y,b^{-1},W)\mathfrak T_b,\qquad \mathfrak T_b(h,\ell):=(h,-b^{-2}\ell).
   			\label{3-13}
   		\end{equation}
   		Since $\mathfrak T_b$ is an isomorphism, Lemma~\ref{lem2-7} proves the Fredholm assertion, while Lemma~\ref{lem3-1} and \eqref{3-8} give the analytic initial arc and its boundary limit.

   		Along this arc,
   		\[
   		\widehat{\mathscr L}_W=\begin{pmatrix}A_W&-b_W^{-2}B_a\\[1mm]P_Y&-b_W^{-2}P_a\end{pmatrix},\qquad b_W\asymp W.
   		\]
   		By \eqref{3-6}--\eqref{3-7} and the estimates in Lemma~\ref{lem3-1},
   		\[
   		\begin{aligned}
   			\|A_W-(g-\sigma\partial_x^2)\|_{\mathcal L(H_{\rm e}^{s+2},H_{\rm e}^s)}&=O(W^{3/2}),\\
   			\|b_W^{-2}B_a\|_{H^s}&=O(W^{1/2}),\\
   			\|P_Y\|_{\mathcal L(H_{\rm e}^{s+2},\mathbb R)}&=O(W^{5/2}),\\
   			-b_W^{-2}P_a&=-\frac{\Gamma}{4\pi}+O(W^3).
   		\end{aligned}
   		\]
   		These bounds prove \eqref{3-11}. Since $\widehat{\mathscr L}_0$ is an isomorphism by Lemma~\ref{lem2-5} at $W=0$ and $\Gamma>0$, a Neumann series gives \eqref{3-12}.
   	\end{proof}

   	The Fredholm property also supplies the local properness required for analytic continuation. At any admissible configuration, let $T=D_{(Y,b)}\widehat{\mathscr F}$ and choose decompositions $\mathcal X^s=\ker T\oplus\mathcal E$ and $\mathcal Y^s=\operatorname{Ran}T\oplus\mathcal F$, with $\ker T$ and $\mathcal F$ finite dimensional. After projection onto $\operatorname{Ran}T$, the implicit function theorem solves the $\mathcal E$ component locally in terms of the kernel coordinate, $W$, and the projected residual. Thus bounded sequences in a sufficiently small closed neighborhood with convergent residuals and speeds have convergent subsequences. This is a local compactness statement, not a global compactness assertion for bounded solution families.

   	We may therefore apply the analytic global continuation theorem of Haziot and Wheeler \cite[Theorem~B.1]{HaziotWheeler2023}, which permits an initial arc approaching a boundary point where the equation is undefined. In particular, $\widehat{\mathscr L}_0$ in \eqref{3-11} is only an operator limit along the arc, not a derivative of $\widehat{\mathscr F}$ at $(0,0,0)$; no analytic extension is asserted there. Global loss of compactness is analyzed in Section~\ref{sec4}.

   	\smallskip

   	\subsection{Global continuation}
   	Equip $\mathcal X^s\times\mathbb R$ with the sum norm and define
   	\begin{equation}
   		\mathcal N(Y,b,W):=\|Y\|_{H^{s+2}}+b+W+\frac1{\operatorname{dist}((Y,b,W),\partial\widehat{\mathscr U}^{\,s})}.
   		\label{3-14}
   	\end{equation}

   	\begin{proposition}
   		\label{prop3-3}
   		There exists a continuous solution curve
   		\begin{equation}
   			\widehat\gamma:(0,\infty)\longrightarrow\widehat{\mathscr U}^{\,s},\qquad \widehat\gamma(t)=(Y(t),b(t),W(t)),\qquad \widehat{\mathscr F}(\widehat\gamma(t))=0,
   			\label{3-15}
   		\end{equation}
   		whose image $\widehat{\mathscr C}_0:=\widehat\gamma((0,\infty))$ contains $\widehat{\mathscr C}_{\rm loc}$. The parametrization may be chosen to follow the initial arc near $t=0$, with $\widehat\gamma(t)\to(0,0,0)$ as $t\downarrow0$. Near each $t_0>0$, the curve admits a local real-analytic reparametrization, and $\widehat\gamma(t)\notin\widehat{\mathscr C}_{\rm loc}$ for all sufficiently large $t$.

   		Writing $\mathcal N(t):=\mathcal N(\widehat\gamma(t))$, one of the following alternatives holds:
   		\begin{equation}
   			\mathcal N(t)\longrightarrow\infty\qquad\text{as }t\to\infty,
   			\label{3-16}
   		\end{equation}
   		or there exists $t_n\to\infty$ such that
   		\begin{equation}
   			\sup_n\mathcal N(t_n)<\infty,
   			\label{3-17}
   		\end{equation}
   		but $\{(Y(t_n),b(t_n))\}$ has no convergent subsequence in $\mathcal X^s$. In particular, $\{Y(t_n)\}$ has no strongly convergent subsequence in $H^{s+2}(\mathbb R)$.
   	\end{proposition}

   	\begin{proof}
   		Apply \cite[Theorem~B.1]{HaziotWheeler2023} with unknown $(Y,b)$ and parameter $W$. By Lemma~\ref{lem3-2}, the transformed map is analytic, its fixed-speed derivative is Fredholm of index zero, and the initial analytic arc approaches the boundary point $(0,0,0)$ with invertible fixed-speed derivative. The local properness established above verifies the remaining compactness hypothesis of the continuation theorem. The theorem yields the stated curve and alternatives, oriented from the initial boundary end. Away from the initial arc, its parameter $t$ need not agree with $W$.
   	\end{proof}

   	Returning to conformal depth, define
   	\begin{equation}
   		a(t):=b(t)^{-1},\qquad \gamma(t):=(Y(t),a(t),W(t)),\qquad \mathscr C_0:=\{\gamma(t):t>0\}.
   		\label{3-18}
   	\end{equation}
   	Then $\mathscr F(\gamma(t))=0$, with the initial asymptotics from \eqref{3-2} and \eqref{3-8}. The chosen continuation $\mathscr C_0$ need not exhaust the connected component of the solution set.

   	For the sum norm,
   	\[
   	\operatorname{dist}((Y,b,W),\partial\widehat{\mathscr U}^{\,s})=\min\{\operatorname{dist}(Y,\partial\mathcal O^s),b,W,c_{\min}-W\}.
   	\]
   	Since $|D|:H^{s+2}\to L^\infty$ is continuous, $\operatorname{dist}(Y,\partial\mathcal O^s)\ge c_Y/C_s$. Thus alternative \eqref{3-16} forces, along some sequence, divergence of $\|Y\|_{H^{s+2}}$, $a\to0$ or $a\to\infty$, $W\to0$ or $W\to c_{\min}$, or $c_Y\to0$. The first alternative allows the wave speed to approach the critical value $c_{\min}$; the present continuation theory does not determine whether such a limit occurs along the global solution curve.

   	Under \eqref{3-17}, put $Y_n:=Y(t_n)$, $b_n:=b(t_n)$, $a_n:=b_n^{-1}$, and $W_n:=W(t_n)$. The bound on $\mathcal N$ gives bounded $\|Y_n\|_{H^{s+2}}$, compact ranges for $a_n,W_n$, and a positive lower bound $\delta_0$ on $\operatorname{dist}(Y_n,\partial\mathcal O^s)$. To control the graph condition, if $0<c_Y<1$ then $\widetilde Y:=Y/(1-c_Y)$ belongs to $\partial\mathcal O^s$, since $\inf(1+|D|\widetilde Y)=0$ and $r\widetilde Y\in\mathcal O^s$ for $0\le r<1$. Consequently,
   	\[
   	\operatorname{dist}(Y,\partial\mathcal O^s)\le\frac{c_Y}{1-c_Y}\|Y\|_{H^{s+2}}.
   	\]
   	Together with the uniform surface bound and $\delta_0>0$, this yields $c_{Y_n}\ge c_0>0$ (trivially when $c_{Y_n}\ge1$). Hence for some $R,c_0>0$, $0<a_-<a_+<\infty$, and $[W_-,W_+]\Subset(0,c_{\min})$,
   	\begin{equation}
   		\|Y_n\|_{H^{s+2}}\le R,\qquad 1+|D|Y_n\ge c_0,\qquad a_n\in[a_-,a_+],\qquad W_n\in[W_-,W_+].
   		\label{3-19}
   	\end{equation}
   	After passing to a subsequence, $a_n\to a_\infty>0$ and $W_n\to W_\infty\in(0,c_{\min})$, whereas $Y_n$ has no strongly convergent subsequence in $H_{\rm e}^{s+2}(\mathbb R)$. We analyze this remaining loss of compactness in Section~\ref{sec4}.

    \bigskip

    \section{Compactness and global alternatives}
    \label{sec4}
    We consider bounded sequences of coupled point-vortex solutions. Uniform surface and graph bounds, together with a compact subcritical speed range, control both the vortex depth and its distance from the free surface. Strong compactness may still fail through surface profiles escaping to horizontal infinity. We describe this loss of compactness by a finite profile decomposition with a strongly vanishing remainder.

    \subsection{Control of the vortex position}
    The estimates below rely only on vortex equilibrium $P=0$; neither Bernoulli's equation nor the subcritical speed restriction is needed.

    \begin{lemma}
        \label{lem4-1}
        Let $s>\frac12$ and $Y\in\mathcal O^s$ satisfy
        \[
        0<c_0\le\vartheta_Y(x)\le C_0<\infty,\qquad x\in\mathbb R.
        \]
        If $a,W>0$ satisfy $P(Y,a,W)=0$, then
        \begin{equation}
            \frac{\Gamma}{4\pi}\frac{c_0}{C_0}\le aW\le\frac{\Gamma}{4\pi}\frac{C_0}{c_0},\qquad aW<\frac{\Gamma}{2\pi}.
            \label{4-1}
        \end{equation}
        The physical vortex $p_a:=Z_Y(-\mathrm{i}a)$ has Euclidean distance $d_Y(a):=\operatorname{dist}(p_a,\partial\Omega_Y)$ from the free surface, satisfying
        \begin{equation}
            c_0a\le d_Y(a)\le C_0a.
            \label{4-2}
        \end{equation}
    \end{lemma}

    \begin{proof}
        For fixed $Y$, write $\alpha(a):=\alpha(Y,a)$. Since $\alpha'(a)=E_a(Y'')$, the equation $P=0$ gives
        \begin{equation}
            aW=\frac{\Gamma}{4\pi}\frac{\alpha(a)+a\alpha'(a)}{\alpha(a)}.
            \label{4-3}
        \end{equation}
        The Poisson representation yields
        \[
        \begin{aligned}
            \alpha(a)&=\int_{\mathbb R}\frac{a}{\pi(a^2+x^2)}\vartheta_Y(x)\,dx,\\
            \alpha(a)+a\alpha'(a)&=\int_{\mathbb R}\frac{2ax^2}{\pi(a^2+x^2)^2}\vartheta_Y(x)\,dx.
        \end{aligned}
        \]
        Both kernels are nonnegative and integrate to one, so $c_0\le\alpha(a),\alpha(a)+a\alpha'(a)\le C_0$, proving the two-sided bound in \eqref{4-1}. Moreover,
        \[
        2\alpha(a)-\bigl(\alpha(a)+a\alpha'(a)\bigr)=\int_{\mathbb R}\frac{2a^3}{\pi(a^2+x^2)^2}\vartheta_Y(x)\,dx>0.
        \]
        Together with \eqref{4-3}, this gives the strict bound $aW<\Gamma/(2\pi)$.

        For the distance estimate, the Poisson representation gives $\operatorname{Re}Z_Y'\ge c_0$ on $\overline{\Pi_-}$. Since the lower half-plane is convex, for any distinct $z_1,z_2\in\overline{\Pi_-}$,
        \[
        \operatorname{Re}\frac{Z_Y(z_2)-Z_Y(z_1)}{z_2-z_1}=\int_0^1\operatorname{Re}Z_Y'\bigl(z_1+t(z_2-z_1)\bigr)\,dt\ge c_0.
        \]
        Taking $z_1=-\mathrm{i}a$ and $z_2=x\in\mathbb R$ gives $|Z_Y(x)-p_a|\ge c_0|x+\mathrm{i}a|\ge c_0a$, proving the lower bound in \eqref{4-2}. For the upper bound, symmetry and the Poisson estimate give $\alpha(r)\in[c_0,C_0]$, so
        \[
        d_Y(a)\le|Z_Y(-\mathrm{i}a)-Z_Y(0)|=\int_0^a\alpha(r)\,dr\le C_0a.\qedhere
        \]
    \end{proof}

    For a family satisfying $\|Y\|_{H^{s+2}}\le R$, $\vartheta_Y\ge c_0>0$, and $W\in[W_-,W_+]\Subset(0,\infty)$, Sobolev embedding gives $\vartheta_Y\le C_0:=1+C_sR$. Setting
    \[
    a_-:=\frac{\Gamma c_0}{4\pi C_0W_+},\qquad a_+:=\frac{\Gamma}{2\pi W_-},
    \]
    Lemma~\ref{lem4-1} yields
    \begin{equation}
        a_-\le a<a_+,\qquad c_0a_-\le\operatorname{dist}(p_a,\partial\Omega_Y)\le C_0a_+.
        \label{4-4}
    \end{equation}
    These estimates require no uniform spatial decay of the surfaces. Under the stated bounds they rule out both vortex--surface collision and escape to infinite depth. Independently of uniform surface bounds, the strict inequality in \eqref{4-1} also implies $W_n\to0$ whenever $a_n\to\infty$ along a sequence of vortex equilibria. We next consider centered limits of bounded surface profiles.

    \smallskip

    \subsection{Centered limits of bounded solutions}
    We first pass to a limit without translating the surfaces, so the symmetry axis stays fixed. Bernoulli's equation provides one additional derivative and hence local strong convergence without uniform tail bounds.

    \begin{lemma}
        \label{lem4-2}
        Let $s>\frac12$, and let $(Y_n,a_n,W_n)\in\mathscr U^s$ satisfy $\mathscr F(Y_n,a_n,W_n)=0$. Suppose that, for some $R,c_0>0$ and $[W_-,W_+]\Subset(0,c_{\min})$,
        \begin{equation}
            \|Y_n\|_{H^{s+2}}\le R,\qquad 1+|D|Y_n\ge c_0,\qquad W_n\in[W_-,W_+].
            \label{4-5}
        \end{equation}
        Then
        \begin{equation}
            \sup_n\|Y_n\|_{H^{s+3}}\le C,
            \label{4-6}
        \end{equation}
        where $C$ depends only on the stated bounds, $s$, and the fixed physical parameters. After passing to a subsequence, there exist $a_\infty>0$, $W_\infty\in[W_-,W_+]$, and $Y^{(0)}\in\mathcal O^s\cap H^{s+3}(\mathbb R)$ such that $a_n\to a_\infty$, $W_n\to W_\infty$, and
        \begin{equation}
            Y_n\rightharpoonup Y^{(0)}\quad\text{in }H^{s+3}(\mathbb R),\qquad Y_n\to Y^{(0)}\quad\text{in }H_{\rm loc}^{s+2}(\mathbb R).
            \label{4-7}
        \end{equation}
        Moreover, $\|Y^{(0)}\|_{H^{s+2}}\le R$, $1+|D|Y^{(0)}\ge c_0$, and
        \begin{equation}
            \mathscr F(Y^{(0)},a_\infty,W_\infty)=0,\qquad Y^{(0)}\not\equiv0.
            \label{4-8}
        \end{equation}
        The physical vortex positions also converge: $p_n:=Z_{Y_n}(-\mathrm{i}a_n)\to p_\infty:=Z_{Y^{(0)}}(-\mathrm{i}a_\infty)$.
    \end{lemma}

    \begin{proof}
        By Sobolev embedding and Lemma~\ref{lem4-1}, the bounds in \eqref{4-5} keep $a_n$ in a compact subinterval of $(0,\infty)$. After extraction, $a_n\to a_\infty>0$ and $W_n\to W_\infty\in[W_-,W_+]$.

        Since $\kappa(Y)=-\mathcal T_YY''$ by \eqref{2-26}, the Bernoulli equation becomes
        \[
        \sigma\mathcal T_{Y_n}Y_n''=gY_n+\frac{q_{a_n,W_n}^2}{2J(Y_n)}-\frac{W_n^2}{2}.
        \]
        The right-hand side is uniformly bounded in $H^{s+1}$, using $J(Y_n)^{-1}-1$, $q_{a_n,W_n}^2-W_n^2\in H^{s+1}$ and the cancellation of the constant terms. By the inverse formula for $\mathcal T_Y$ in Lemma~\ref{lem2-6}, Sobolev product and composition estimates bound $\mathcal T_{Y_n}^{-1}$ uniformly on $H^{s+1}$. Hence $\|Y_n''\|_{H^{s+1}}\le C$, which, together with the $L^2$ bound, proves \eqref{4-6}.

        Weak compactness and the local Rellich theorem give \eqref{4-7}. To pass to the limit in the nonlocal curvature term, global weak convergence also gives $|D|Y_n\rightharpoonup|D|Y^{(0)}$ in $H^{s+2}$; the uniform bound and local compactness then imply $|D|Y_n\to|D|Y^{(0)}$ in $H_{\rm loc}^{s+1}$. As $s>\frac12$, these convergences give $Y_n\to Y^{(0)}$ in $C_{\rm loc}^2$ and $|D|Y_n\to|D|Y^{(0)}$ in $C_{\rm loc}^1$. Thus $1+|D|Y^{(0)}\ge c_0$, and passage to the limit in the Bernoulli equation gives $B(Y^{(0)},a_\infty,W_\infty)=0$.

        For the vortex equation, operator-norm continuity of $E_a$ for $a>0$ and global weak convergence yield $E_{a_n}(Y_n'')\to E_{a_\infty}((Y^{(0)})'')$ and $E_{a_n}(|D|Y_n)\to E_{a_\infty}(|D|Y^{(0)})$. The denominators are bounded below by $c_0$, so $P(Y^{(0)},a_\infty,W_\infty)=0$. If $Y^{(0)}=0$, this identity forces $W_\infty=\Gamma/(4\pi a_\infty)$, whence $q_{a_\infty,W_\infty}(0)=-3W_\infty$ and
        \[
        B(0,a_\infty,W_\infty)(0)=4W_\infty^2>0,
        \]
        contradicting the limiting Bernoulli equation. Thus the centered wave is nontrivial. Finally, $p_n=-\mathrm{i}a_n+\mathrm{i}E_{a_n}Y_n$, and the same convergence of Poisson evaluations gives $p_n\to p_\infty$. Lemma~\ref{lem4-1} keeps the limiting vortex strictly below the free surface.
    \end{proof}

    Set $R_n:=Y_n-Y^{(0)}$. By \eqref{4-7},
    \begin{equation}
        R_n\rightharpoonup0\quad\text{in }H^{s+2}(\mathbb R),\qquad R_n\to0\quad\text{in }H_{\rm loc}^{s+2}(\mathbb R).
        \label{4-9}
    \end{equation}
    If the original sequence has no strongly convergent subsequence in $H^{s+2}$, we may extract further so that $\|R_n\|_{H^{s+2}}\ge\varepsilon_0>0$. Any remaining loss of compactness must therefore occur outside every fixed bounded interval. This leads to the translated-profile analysis below.

    \smallskip

    \subsection{Nonvanishing and translated profiles}
    We examine whether the centered remainder $R_n=Y_n-Y^{(0)}$ converges strongly. For these controlled solutions, vanishing local $L^2$ mass forces global convergence. Otherwise, translation reveals a nontrivial irrotational profile.

    \begin{lemma}
        \label{lem4-3}
        Under the assumptions of Lemma~\ref{lem4-2}, retain the extracted sequence and centered limit, and set $R_n:=Y_n-Y^{(0)}$. If
        \begin{equation}
            \eta_n:=\sup_{x\in\mathbb R}\int_{x-1}^{x+1}|R_n(y)|^2\,dy\longrightarrow0,
            \label{4-10}
        \end{equation}
        then $\|R_n\|_{H^{s+2}}\to0$. Otherwise, if $\{Y_n\}$ has no strongly convergent subsequence in $H^{s+2}(\mathbb R)$, we may extract further and find $x_n\to+\infty$ and a nonzero real-valued $Q\in H^{s+3}(\mathbb R)$ such that
        \begin{equation}
            Y_n(\cdot+x_n)\rightharpoonup Q\quad\text{in }H^{s+3}(\mathbb R),\qquad Y_n(\cdot+x_n)\to Q\quad\text{in }H_{\rm loc}^{s+2}(\mathbb R).
            \label{4-11}
        \end{equation}
        The profile satisfies $1+|D|Q\ge c_0$ and
        \begin{equation}
            gQ+\sigma\kappa(Q)+\frac{W_\infty^2}{2J(Q)}-\frac{W_\infty^2}{2}=0.
            \label{4-12}
        \end{equation}
        By reflection symmetry,
        \begin{equation}
            Y_n(\cdot-x_n)\longrightarrow Q(-\,\cdot)\quad\text{in }H_{\rm loc}^{s+2}(\mathbb R).
            \label{4-13}
        \end{equation}
    \end{lemma}

    \begin{proof}
        Assume \eqref{4-10}. The uniform $H^{s+3}$ bound from Lemma~\ref{lem4-2} gives
        \begin{equation}
            \|R_n\|_{W^{2,\infty}}+\||D|R_n\|_{W^{1,\infty}}\longrightarrow0.
            \label{4-14}
        \end{equation}
        To see this, for arbitrary centers $z_n$ the translates $f_n:=R_n(\cdot+z_n)$ remain bounded in $H^{s+3}$ and satisfy $\|f_n\|_{L^2(-R,R)}^2\le C_R\eta_n\to0$ for every fixed $R$. Every weak limit is therefore zero. Local compactness, together with global weak convergence for $|D|f_n$, gives $f_n\to0$ in $H_{\rm loc}^{s+2}$ and $|D|f_n\to0$ in $H_{\rm loc}^{s+1}$ after extraction. If \eqref{4-14} failed, a suitable choice of $z_n$ would contradict these limits at the origin by Sobolev embedding.

        Set $Y_{n,\theta}:=Y^{(0)}+\theta R_n$ for $0\le\theta\le1$. The graph condition persists along this segment because $1+|D|Y_{n,\theta}\ge c_0$. Writing $A^{(0)}:=D_YB(Y^{(0)},a_\infty,W_\infty)$, the fundamental theorem of calculus gives
        \[
        \overline A_n:=\int_0^1D_YB(Y_{n,\theta},a_n,W_n)\,d\theta,\qquad \overline A_nR_n=-B(Y^{(0)},a_n,W_n)\longrightarrow0\quad\text{in }H^s.
        \]
        By \eqref{4-14}, parameter convergence, and \eqref{2-11}--\eqref{2-12}, $\|\overline A_n-A^{(0)}\|_{\mathcal L(H^2,L^2)}\to0$. Since $R_n$ is bounded in $H^2$,
        \begin{equation}
            A^{(0)}R_n\longrightarrow0\quad\text{in }L^2(\mathbb R).
            \label{4-15}
        \end{equation}
        No invertibility of $A^{(0)}$ is assumed. By \eqref{2-27},
        \[
        A^{(0)}=\mathcal T_{Y^{(0)}}A_\infty(W_\infty)+\mathcal K^{(0)},
        \]
        where $\mathcal T_{Y^{(0)}}$ is an isomorphism on $L^2$ and $\mathcal K^{(0)}:H^2\to L^2$ is compact. Indeed, the coefficients of the order-at-most-one remainder belong to $C_0(\mathbb R)$, and multiplication by such coefficients is compact from $H^1$ to $L^2$ by cutoff and local Rellich compactness. Since $R_n\rightharpoonup0$ in $H^2$, we have $\mathcal K^{(0)}R_n\to0$ in $L^2$. Applying $\mathcal T_{Y^{(0)}}^{-1}$ to \eqref{4-15} and using Lemma~\ref{lem2-5} yields $\|R_n\|_{H^2}\to0$. Interpolation with the uniform $H^{s+3}$ bound gives
        \[
        \|R_n\|_{H^{s+2}}\le C\|R_n\|_{H^2}^{1/(s+1)}\|R_n\|_{H^{s+3}}^{s/(s+1)}\longrightarrow0.
        \]

        If strong compactness fails, the preceding argument gives nonvanishing along a further subsequence: for some $\eta>0$ and centers $x_n$,
        \begin{equation}
            \int_{x_n-1}^{x_n+1}|R_n(x)|^2\,dx\ge\eta.
            \label{4-16}
        \end{equation}
        Local convergence in \eqref{4-9} gives $|x_n|\to\infty$, and evenness allows $x_n\to+\infty$. Weak and local compactness of $Y_n(\cdot+x_n)$ give \eqref{4-11}. Since $Y^{(0)}(\cdot+x_n)\to0$ locally, \eqref{4-16} implies $\int_{-1}^1|Q|^2\,dx\ge\eta>0$. Translation of the Bernoulli equation gives
        \begin{equation}
            g\widetilde Y_n+\sigma\kappa(\widetilde Y_n)+\frac{q_{a_n,W_n}(x+x_n)^2}{2J(\widetilde Y_n)}-\frac{W_n^2}{2}=0,\qquad \widetilde Y_n:=Y_n(\cdot+x_n).
            \label{4-17}
        \end{equation}
        The parameter bounds imply $q_{a_n,W_n}(\cdot+x_n)\to W_\infty$ in $C_{\rm loc}^k$ for every fixed $k$. Global weak convergence identifies the nonlocal limit $|D|\widetilde Y_n\to|D|Q$ in $H_{\rm loc}^{s+1}$, as in Lemma~\ref{lem4-2}. Passing to the limit in \eqref{4-17} gives the graph bound and \eqref{4-12}. Finally, evenness gives $Y_n(x-x_n)=\widetilde Y_n(-x)$ and hence \eqref{4-13}.
    \end{proof}

    For subsequent extractions, use the irrotational residual
    \begin{equation}
        B_{\rm irr}(Q,W):=gQ+\sigma\kappa(Q)+\frac{W^2}{2J(Q)}-\frac{W^2}{2}.
        \label{4-18}
    \end{equation}
    It belongs to $L^2$ for real $Q\in H^2$ with $\inf(1+|D|Q)>0$; individual profiles need not be even.

    \begin{lemma}
        \label{lem4-4}
        For every $0<W_+<c_{\min}$ there exists $\eta_*>0$, depending only on $g,\sigma,W_+$, such that every nonzero real solution $Q\in H^2(\mathbb R)$ of $B_{\rm irr}(Q,W)=0$, with $0\le W\le W_+$ and $\inf_{\mathbb R}(1+|D|Q)>0$, satisfies
        \begin{equation}
            \|Q\|_{H^2}\ge\eta_*.
            \label{4-19}
        \end{equation}
    \end{lemma}

    \begin{proof}
        Separate the flat linear part:
        \begin{equation}
            B_{\rm irr}(Q,W)=A_\infty(W)Q+\mathfrak N_W(Q),
            \label{4-20}
        \end{equation}
        where $\mathfrak N_W(Q):=\sigma(\kappa(Q)+Q'')+W^2[(2J(Q))^{-1}-1/2+|D|Q]$. Set $r:=|D|Q$ and $\mathfrak e(Q):=\|r\|_{L^\infty}+\|Q'\|_{L^\infty}\le C_{\rm emb}\|Q\|_{H^2}$. When $\mathfrak e(Q)\le1/4$, the graph bound and $J\ge9/16$ hold. The identities
        \[
        \kappa(Q)+Q''=\left(1-\frac{1+r}{J^{3/2}}\right)Q''+\frac{Q'}{J^{3/2}}(|D|Q)',\qquad \frac1{2J}-\frac12+r=\frac{3r^2-(Q')^2+2r(r^2+(Q')^2)}{2J}
        \]
        show, using $\|(|D|Q)'\|_{L^2}=\|Q''\|_{L^2}$, that
        \begin{equation}
            \|\mathfrak N_W(Q)\|_{L^2}\le C_{\rm N}\mathfrak e(Q)\|Q\|_{H^2},\qquad 0\le W\le W_+,\quad \mathfrak e(Q)\le\frac14.
            \label{4-21}
        \end{equation}
        By \eqref{2-25}, $\|A_\infty(W)Q\|_{L^2}\ge c_{\rm f}\|Q\|_{H^2}$, where $c_{\rm f}:=(1-W_+^2/c_{\min}^2)\min\{g,\sigma\}>0$. If $\mathfrak e(Q)\le r_*:=\min\{1/4,c_{\rm f}/(2C_{\rm N})\}$ and $B_{\rm irr}(Q,W)=0$, then $c_{\rm f}\|Q\|_{H^2}\le(c_{\rm f}/2)\|Q\|_{H^2}$, so $Q=0$. Thus every nonzero solution has $\mathfrak e(Q)>r_*$, and \eqref{4-19} follows with $\eta_*:=r_*/C_{\rm emb}$.
    \end{proof}

    Let $Q$ and $x_n$ be the first extracted profile and center, and set $R_n^{(1)}:=Y_n-Y^{(0)}-Q(\cdot-x_n)-Q(-\,\cdot-x_n)$. Separation of the profiles and weak convergence yield
    \begin{equation}
        \|Y_n\|_{H^2}^2=\|Y^{(0)}\|_{H^2}^2+2\|Q\|_{H^2}^2+\|R_n^{(1)}\|_{H^2}^2+o(1).
        \label{4-22}
    \end{equation}
    Each reflected pair contributes at least $2\eta_*^2$ to the squared norm. To iterate the decomposition, we must preserve the graph condition and control the nonlinear residual after subtraction.

    \begin{lemma}
        \label{lem4-5}
        Let $(Y_n,a_n,W_n)$ and $Y^{(0)}$ be as in Lemma~\ref{lem4-2}. Suppose that nonzero real profiles $Q_1,\ldots,Q_N\in H^{s+3}(\mathbb R)$, for some $N\ge0$, satisfy $B_{\rm irr}(Q_j,W_\infty)=0$ and $1+|D|Q_j\ge c_0$. Let positive centers obey
        \begin{equation}
            x_n^{(j)}\longrightarrow+\infty,\qquad |x_n^{(j)}-x_n^{(k)}|\longrightarrow\infty\quad\text{for }j\ne k,
            \label{4-23}
        \end{equation}
        and assume $Y_n(\cdot+x_n^{(j)})\rightharpoonup Q_j$ in $H^{s+3}$ and $Y_n(\cdot+x_n^{(j)})\to Q_j$ in $H_{\rm loc}^{s+2}$ for each $j$. Define
        \[
        S_n^{(N)}(x):=Y^{(0)}(x)+\sum_{j=1}^N\bigl[Q_j(x-x_n^{(j)})+Q_j(-x-x_n^{(j)})\bigr],
        \]
        and
        \begin{equation}
            R_n^{(N)}:=Y_n-S_n^{(N)}.
            \label{4-24}
        \end{equation}
        Then $R_n^{(N)}$ is even, bounded in $H^{s+3}$, and tends to zero in $H_{\rm loc}^{s+2}$ without translation and after translation by each $\pm x_n^{(j)}$. For all sufficiently large $n$,
        \begin{equation}
            1+|D|R_n^{(N)}\ge c_\sharp,\qquad c_\sharp:=\frac12\min\{1,c_0\}>0,
            \label{4-25}
        \end{equation}
        and
        \begin{equation}
            \|B_{\rm irr}(R_n^{(N)},W_\infty)\|_{L^2}\longrightarrow0.
            \label{4-26}
        \end{equation}
        For $N=0$, the sum is empty.
    \end{lemma}

    \begin{proof}
        Put $Q_0:=Y^{(0)}$, $Q_{-j}(x):=Q_j(-x)$, and
        \[
        z_{0,n}:=0,\qquad z_{j,n}:=x_n^{(j)},\qquad z_{-j,n}:=-x_n^{(j)},\qquad Q_{j,n}:=Q_j(\cdot-z_{j,n}).
        \]
        Write $S_n:=\sum_{j=-N}^NQ_{j,n}$ and $R_n:=R_n^{(N)}$. Subtracting finitely many profiles preserves both evenness and the uniform $H^{s+3}$ bound. By the hypotheses, reflection symmetry, and the nonlocal convergence argument of Lemma~\ref{lem4-2},
        \[
        Y_n(\cdot+z_{j,n})\to Q_j\quad\text{in }H_{\rm loc}^{s+2},\qquad (|D|Y_n)(\cdot+z_{j,n})\to|D|Q_j\quad\text{in }H_{\rm loc}^{s+1}.
        \]
        All other profiles vanish locally in these coordinates. Consequently, the corresponding translates of $R_n$ and $|D|R_n$ also converge locally to zero.

        To establish \eqref{4-25}, choose $R_0$ large enough that
        \[
        \sum_{j=-N}^N\sup_{|x|>R_0}\bigl||D|Q_j(x)\bigr|\le\frac{c_0}{2},\qquad \mathcal U_{n,R_0}:=\bigcup_{j=-N}^N(z_{j,n}-R_0,z_{j,n}+R_0).
        \]
        Such a choice is possible because $|D|Q_j\in C_0(\mathbb R)$. Outside $\mathcal U_{n,R_0}$, the graph bound for $Y_n$ gives $1+|D|R_n\ge c_0/2$; on $\mathcal U_{n,R_0}$, the local convergence of $|D|R_n$ gives $1+|D|R_n\ge1/2$ for large $n$. Hence \eqref{4-25} holds, with $c_\sharp$ independent of $N$.

        For the nonlinear residual, set $\mathfrak N:=\mathfrak N_{W_\infty}$ as in \eqref{4-20}, and define
        \begin{equation}
            \mathfrak E_n:=\mathfrak N(Y_n)-\mathfrak N(R_n)-\sum_{j=-N}^N\mathfrak N(Q_{j,n})\longrightarrow0\quad\text{in }L^2(\mathbb R).
            \label{4-27}
        \end{equation}
        To prove this, set $\mathbf z(f):=(|D|f,f')$ and $\mathbf w(f):=(f'',(|D|f)')$. Then
        \[
        \mathfrak N(f)=\mathbf c(\mathbf z(f))\cdot\mathbf w(f)+k(\mathbf z(f)),
        \]
        where, for $\mathbf z=(z_1,z_2)$ and $J_{\mathbf z}:=(1+z_1)^2+z_2^2$,
        \[
        \mathbf c(\mathbf z):=\sigma\left(1-\frac{1+z_1}{J_{\mathbf z}^{3/2}},\frac{z_2}{J_{\mathbf z}^{3/2}}\right),\qquad k(\mathbf z):=W_\infty^2\left(\frac1{2J_{\mathbf z}}-\frac12+z_1\right).
        \]
        These functions vanish at the origin and have uniformly bounded first derivatives under the Sobolev and graph bounds. For fixed $R_0$, the intervals in $\mathcal U_{n,R_0}$ are disjoint for large $n$; there $Y_n$ converges to the corresponding profile while $R_n$ and the others vanish locally, so $\|\mathfrak E_n\|_{L^2(\mathcal U_{n,R_0})}\to0$.

        On $\mathcal E_{n,R_0}:=\mathbb R\setminus\mathcal U_{n,R_0}$, let $\varepsilon_N(R_0)$ be the sum, over $-N\le j\le N$, of the $L^\infty$ and $L^2$ tail norms of $\mathbf z(Q_j)$ and the $L^2$ tail norms of $\mathbf w(Q_j)$ outside $(-R_0,R_0)$. Then $\varepsilon_N(R_0)\to0$ as $R_0\to\infty$, and
        \[
        \|\mathbf z(S_n)\|_{L^\infty(\mathcal E_{n,R_0})}+\|\mathbf z(S_n)\|_{L^2(\mathcal E_{n,R_0})}+\|\mathbf w(S_n)\|_{L^2(\mathcal E_{n,R_0})}\le\varepsilon_N(R_0).
        \]
        Every intermediate surface $R_n+\theta S_n$, $0\le\theta\le1$, satisfies
        \[
        1+|D|(R_n+\theta S_n)=(1-\theta)(1+|D|R_n)+\theta(1+|D|Y_n)\ge\min\{c_\sharp,c_0\}>0.
        \]
        The mean value theorem therefore gives, on $\mathcal E_{n,R_0}$,
        \[
        |\mathfrak N(Y_n)-\mathfrak N(R_n)|\le C_N\bigl(|\mathbf z(S_n)|\,|\mathbf w(R_n)|+|\mathbf w(S_n)|+|\mathbf z(S_n)|\bigr),
        \]
        and hence
        \[
        \|\mathfrak N(Y_n)-\mathfrak N(R_n)\|_{L^2(\mathcal E_{n,R_0})}+\sum_{j=-N}^N\|\mathfrak N(Q_{j,n})\|_{L^2(\mathcal E_{n,R_0})}\le C_N\varepsilon_N(R_0).
        \]
        Taking first $n\to\infty$ and then $R_0\to\infty$ proves \eqref{4-27}. No smallness of $R_n$ away from the extracted profiles is required.

        Parameter continuity and $B(Y_n,a_n,W_n)=0$ imply $B(Y_n,a_\infty,W_\infty)\to0$ in $L^2$. With $F_\infty:=(q_{a_\infty,W_\infty}^2-W_\infty^2)/2\in L^2$, the Bernoulli residual satisfies
        \[
        B(Y,a_\infty,W_\infty)=B_{\rm irr}(Y,W_\infty)+F_\infty/J(Y).
        \]
        The centered profile satisfies the full Bernoulli equation, while every distant profile satisfies $B_{\rm irr}(Q_j,W_\infty)=0$. Thus the additivity of $A_\infty(W_\infty)$ and \eqref{4-27} give
        \begin{equation}
            B_{\rm irr}(R_n,W_\infty)=F_\infty\bigl[J(Y^{(0)})^{-1}-J(Y_n)^{-1}\bigr]+o_{L^2}(1).
            \label{4-28}
        \end{equation}
        The bracket converges locally uniformly to zero and remains uniformly bounded by the graph conditions. Dominated convergence against $|F_\infty|^2$ proves \eqref{4-26}.
    \end{proof}

    For each fixed $N$, the remainder tends weakly to zero in $H^{s+3}$ after translation by every extracted center, including zero. Consequently, the translated profiles and the remainder are asymptotically orthogonal in $H^2$ and $H^{s+3}$. Expanding \eqref{4-24} yields
    \begin{equation}
        \|Y_n\|_{H^r}^2=\|Y^{(0)}\|_{H^r}^2+2\sum_{j=1}^N\|Q_j\|_{H^r}^2+\|R_n^{(N)}\|_{H^r}^2+o(1),\qquad r=2,\ s+3.
        \label{4-29}
    \end{equation}
    This norm splitting controls each step of the extraction.

    \begin{proposition}
        \label{prop4-6}
        Let $s>\frac12$, and let $(Y_n,a_n,W_n)\in\mathscr U^s$ be coupled solutions satisfying \eqref{4-5}. After extraction, there exist a nontrivial centered solution $(Y^{(0)},a_\infty,W_\infty)$ as in Lemma~\ref{lem4-2}, a finite integer $N\ge0$, and nonzero real profiles $Q_1,\ldots,Q_N\in H^{s+3}(\mathbb R)$ satisfying $B_{\rm irr}(Q_j,W_\infty)=0$ and $1+|D|Q_j\ge c_0$. We have $a_n\to a_\infty>0$, $W_n\to W_\infty\in[W_-,W_+]$, and there are positive centers $x_n^{(j)}$, $1\le j\le N$, satisfying \eqref{4-23} (an empty family when $N=0$), such that
        \begin{equation}
            Y_n(x)=Y^{(0)}(x)+\sum_{j=1}^N\bigl[Q_j(x-x_n^{(j)})+Q_j(-x-x_n^{(j)})\bigr]+R_n^{(N)}(x),
            \label{4-30}
        \end{equation}
        with
        \begin{equation}
            \|R_n^{(N)}\|_{H^{s+2}}\longrightarrow0.
            \label{4-31}
        \end{equation}
        For each $j$, $Y_n(\cdot+x_n^{(j)})$ converges weakly to $Q_j$ in $H^{s+3}$ and strongly in $H_{\rm loc}^{s+2}$. Moreover,
        \begin{equation}
            \|Y_n\|_{H^r}^2=\|Y^{(0)}\|_{H^r}^2+2\sum_{j=1}^N\|Q_j\|_{H^r}^2+o(1),\qquad r=2,\ s+2.
            \label{4-32}
        \end{equation}
        For $\eta_*$ from Lemma~\ref{lem4-4},
        \begin{equation}
            2N\eta_*^2\le\limsup_{n\to\infty}\|Y_n\|_{H^2}^2-\|Y^{(0)}\|_{H^2}^2\le R^2.
            \label{4-33}
        \end{equation}
        Thus $N=0$ if and only if the extracted subsequence converges strongly to $Y^{(0)}$ in $H^{s+2}(\mathbb R)$.
    \end{proposition}

    \begin{proof}
        By Lemma~\ref{lem4-2}, the centered limit exists and $\{Y_n\}$ is uniformly bounded in $H^{s+3}$. After extraction, let $\|Y_n\|_{H^2}^2\to E_2\le R^2$. Suppose $k$ reflected pairs have been extracted and let $R_n^{(k)}$ be given by \eqref{4-24}, initially with $k=0$. Lemma~\ref{lem4-5} provides a uniform $H^{s+3}$ bound, the graph condition, local convergence to zero near every extracted center, and $B_{\rm irr}(R_n^{(k)},W_\infty)\to0$ in $L^2$.

        After passing to a subsequence, let
        \[
        \mu_n^{(k)}:=\sup_{x\in\mathbb R}\int_{x-1}^{x+1}|R_n^{(k)}(y)|^2\,dy
        \]
        have a limit. If $\mu_n^{(k)}\to0$, the translation argument leading to \eqref{4-14}, which uses only the uniform $H^{s+3}$ bound, gives $\mathfrak e(R_n^{(k)})\to0$. By Lemma~\ref{lem2-5} and \eqref{4-20}--\eqref{4-21},
        \[
        \bigl(1-C\mathfrak e(R_n^{(k)})\bigr)\|R_n^{(k)}\|_{H^2}\le C\|B_{\rm irr}(R_n^{(k)},W_\infty)\|_{L^2}\longrightarrow0.
        \]
        The left coefficient is eventually positive, so $R_n^{(k)}\to0$ in $H^2$ and, by interpolation, in $H^{s+2}$. We stop with $N=k$.

        Otherwise, the local mass stays positive. By evenness, choose $\eta_k>0$ and centers $y_n\ge0$ with $\int_{y_n-1}^{y_n+1}|R_n^{(k)}|^2\,dx\ge\eta_k$. Local convergence around zero and the previously extracted centers forces $y_n\to+\infty$ and $|y_n-x_n^{(j)}|\to\infty$ for $1\le j\le k$. The translates $R_n^{(k)}(\cdot+y_n)$ converge weakly in $H^{s+3}$ and strongly in $H_{\rm loc}^{s+2}$ to a nonzero real profile $Q_{k+1}$. Translation of \eqref{4-26}, the graph bound, and the nonlocal convergence argument of Lemma~\ref{lem4-3} show that $B_{\rm irr}(Q_{k+1},W_\infty)=0$. Since all previously subtracted profiles vanish locally around $y_n$, the original translates $Y_n(\cdot+y_n)$ have the same limit; hence $1+|D|Q_{k+1}\ge c_0$. Symmetry supplies the reflection, and setting $x_n^{(k+1)}:=y_n$ gives the next pair satisfying Lemma~\ref{lem4-5}.

        Only finitely many extractions are possible. Indeed, Lemma~\ref{lem4-4} and \eqref{4-29} imply
        \[
        E_2\ge\|Y^{(0)}\|_{H^2}^2+2\sum_{j=1}^k\|Q_j\|_{H^2}^2\ge\|Y^{(0)}\|_{H^2}^2+2k\eta_*^2,
        \]
        which bounds $k$ independently of the extraction. At the final stage, vanishing gives \eqref{4-31}, while asymptotic orthogonality yields \eqref{4-32} and \eqref{4-33}. Finally,
        \[
        \|Y_n-Y^{(0)}\|_{H^{s+2}}^2=2\sum_{j=1}^N\|Q_j\|_{H^{s+2}}^2+o(1),
        \]
        so strong convergence to the centered limit holds precisely when $N=0$.
    \end{proof}

    Thus a noncompact sequence from Proposition~\ref{prop3-3} necessarily contains at least one reflected pair. We next characterize compactness in terms of uniform surface tails.

    \smallskip

    \subsection{A compactness criterion}
    For uniformly controlled coupled solutions, Proposition~\ref{prop4-6} shows that strong compactness fails precisely when distant profiles occur. This criterion can be expressed through uniform decay of the surface tails.

    \begin{corollary}
        \label{cor4-7}
        Let $s>\frac12$, and let $\mathfrak S\subset\mathscr U^s$ be a family of solutions of $\mathscr F=0$ satisfying, for some $R,c_0>0$ and $[W_-,W_+]\Subset(0,c_{\min})$,
        \[
        \|Y\|_{H^{s+2}}\le R,\qquad 1+|D|Y\ge c_0,\qquad W\in[W_-,W_+]
        \]
        for every $(Y,a,W)\in\mathfrak S$. Then $\mathfrak S$ is relatively compact in $\mathcal X^s\times\mathbb R$ if and only if
        \begin{equation}
            \lim_{\rho\to\infty}\sup_{(Y,a,W)\in\mathfrak S}\int_{|x|>\rho}|Y(x)|^2\,dx=0.
            \label{4-34}
        \end{equation}
        Equivalently, relative compactness holds if and only if
        \begin{equation}
            \lim_{\rho\to\infty}\sup_{(Y,a,W)\in\mathfrak S}\sup_{|x|>\rho}|Y(x)|=0.
            \label{4-35}
        \end{equation}
        Under either condition, the closure of $\mathfrak S$ is compact in the admissible solution set; in particular, $\mathfrak S$ is compact if it is closed.
    \end{corollary}

    \begin{proof}
        If $\mathfrak S$ is relatively compact, its surface projection is relatively compact in $L^2(\mathbb R)$ and $C_0(\mathbb R)$. Finite coverings in these spaces give the uniform tails in \eqref{4-34} and \eqref{4-35}.

        Conversely, take a sequence $(Y_n,a_n,W_n)\in\mathfrak S$ and apply Proposition~\ref{prop4-6}. If a nonzero distant profile $Q_j$ occurs, choose $\ell>0$ with $\int_{-\ell}^{\ell}|Q_j|^2\,dx>0$. By local convergence after translation,
        \[
        \int_{x_n^{(j)}-\ell}^{x_n^{(j)}+\ell}|Y_n(x)|^2\,dx\longrightarrow\int_{-\ell}^{\ell}|Q_j(x)|^2\,dx>0,
        \]
        contradicting \eqref{4-34} since $x_n^{(j)}\to+\infty$. Likewise, choosing $x_*$ with $Q_j(x_*)\ne0$ gives $Y_n(x_n^{(j)}+x_*)\to Q_j(x_*)\ne0$, contradicting \eqref{4-35}. Hence either tail condition forces $N=0$, and Proposition~\ref{prop4-6} yields strong convergence of the surface and convergence of the depth and speed. By Lemma~\ref{lem4-1} and continuity of $\mathscr F$ and $|D|:H^{s+2}\to L^\infty$, all limits satisfy the equation and the graph bound within the admissible set. This proves the assertions.
    \end{proof}

    The equivalence uses the wave equations through Proposition~\ref{prop4-6}. For arbitrary bounded Sobolev families, uniform decay in supremum norm need not control the $L^2$ tails.

    Even without a tail condition, the profile decomposition quantifies the $L^2$ mass escaping to horizontal infinity. Norm splitting and local convergence to the centered profile give
    \begin{equation}
        \lim_{\rho\to\infty}\lim_{n\to\infty}\int_{|x|>\rho}|Y_n(x)|^2\,dx=2\sum_{j=1}^N\|Q_j\|_{L^2}^2.
        \label{4-36}
    \end{equation}
    The limit in $n$ precedes the limit in $\rho$, since each fixed $Y_n$ has vanishing $L^2$ tails.

    Under either compactness condition, the conformal normalization at infinity is also uniform over $\mathfrak S$. Setting $\Phi_Y:=Z_Y-\operatorname{id}$, the Poisson representation and Sobolev embedding yield
    \[
    \|\Phi_Y-\Phi_{\widetilde Y}\|_{L^\infty(\overline{\Pi_-})}+\|\Phi_Y'-\Phi_{\widetilde Y}'\|_{L^\infty(\overline{\Pi_-})}\le C_s\|Y-\widetilde Y\|_{H^{s+2}}.
    \]
    Since $\Phi_Y$ and $\Phi_Y'$ vanish at infinity for each fixed $Y$, a finite covering of the compact closure of the surface family gives
    \begin{equation}
        \lim_{\rho\to\infty}\sup_{(Y,a,W)\in\mathfrak S}\sup_{\substack{z\in\overline{\Pi_-}\\|z|>\rho}}\bigl(|Z_Y(z)-z|+|Z_Y'(z)-1|\bigr)=0.
        \label{4-37}
    \end{equation}
    This uniformity is a consequence of compactness; it was not assumed in controlling the vortex position or extracting profiles. Combined with Lemma~\ref{lem4-1}, these facts lead to the refined global alternatives.

    \smallskip

    \subsection{Refined global alternatives}
    Lemma~\ref{lem4-1} shows that the surface, graph, and speed bounds also control the vortex depth. Along the curve $\gamma(t)=(Y(t),a(t),W(t))$ in \eqref{3-18}, set $c(t):=\inf_{\mathbb R}(1+|D|Y(t))$ and define
    \begin{equation}
        \mathcal N_{\rm s}(t):=\|Y(t)\|_{H^{s+2}}+\frac1{c(t)}+\frac1{W(t)}+\frac1{c_{\min}-W(t)}.
        \label{4-38}
    \end{equation}

    \begin{proposition}
        \label{prop4-8}
        For the global curve of Proposition~\ref{prop3-3}, either
        \begin{equation}
            \mathcal N_{\rm s}(t)\longrightarrow\infty\qquad\text{as }t\to\infty,
            \label{4-39}
        \end{equation}
        or there exists a sequence $t_n\to\infty$ such that
        \begin{equation}
            \sup_n\mathcal N_{\rm s}(t_n)<\infty,
            \label{4-40}
        \end{equation}
        and, after extraction, $(Y_n,a_n,W_n):=\gamma(t_n)$ admits the decomposition of Proposition~\ref{prop4-6} with $1\le N<\infty$. In particular, $a_n\to a_\infty>0$, $W_n\to W_\infty\in(0,c_{\min})$, and the centered limit is a nontrivial coupled solution. The nonzero distant profiles satisfy $B_{\rm irr}(Q_j,W_\infty)=0$, their centers obey \eqref{4-23}, and the remainder converges strongly to zero in $H^{s+2}(\mathbb R)$ as in \eqref{4-31}.
    \end{proposition}

    \begin{proof}
        We compare $\mathcal N_{\rm s}$ with $\mathcal N$ from \eqref{3-14}. Write $c_Y:=\inf_{\mathbb R}(1+|D|Y)$, $\delta_Y:=\operatorname{dist}(Y,\partial\mathcal O^s)$, and $b:=a^{-1}$. If $\mathcal N_{\rm s}\le K$, then $\|Y\|_{H^{s+2}}\le K$, $c_Y\ge K^{-1}$, and $W,c_{\min}-W\ge K^{-1}$. By Sobolev embedding, $1+|D|Y\le C_0:=1+C_sK$. Lemma~\ref{lem4-1} gives
        \[
        \frac{2\pi}{\Gamma}W<b\le\frac{4\pi}{\Gamma}\frac{C_0}{c_Y}W.
        \]
        Thus $b$ is bounded above and away from zero. Since $\delta_Y\ge c_Y/C_s$ and
        \[
        \operatorname{dist}((Y,b,W),\partial\widehat{\mathscr U}^{\,s})=\min\{\delta_Y,b,W,c_{\min}-W\},
        \]
        we obtain $\mathcal N\le C_K$.

        Conversely, if $\mathcal N\le K$, then $\|Y\|_{H^{s+2}}\le K$, $\delta_Y\ge K^{-1}$, and $W,c_{\min}-W\ge K^{-1}$. As $|D|Y$ vanishes at infinity, $0<c_Y\le1$. For $c_Y<1$, the profile $\widetilde Y:=Y/(1-c_Y)$ lies on $\partial\mathcal O^s$, so
        \[
        \delta_Y\le\frac{c_Y}{1-c_Y}\|Y\|_{H^{s+2}},\qquad \frac1{c_Y}\le1+\frac{\|Y\|_{H^{s+2}}}{\delta_Y}\le1+K^2.
        \]
        The final bound also holds for $c_Y=1$, giving $\mathcal N_{\rm s}\le C_K$. Hence the two control quantities are bounded on precisely the same subsets of the solution curve.

        If \eqref{3-16} holds but \eqref{4-39} fails, a sequence with bounded $\mathcal N_{\rm s}$ contradicts the first comparison. Otherwise, Proposition~\ref{prop3-3} supplies a bounded-$\mathcal N$ sequence with no convergent subsequence in $\mathcal X^s$ for $(Y_n,b_n)$. The second comparison yields \eqref{4-40}, and Proposition~\ref{prop4-6} applies. If $N=0$, its strong remainder convergence and $a_n\to a_\infty>0$ imply convergence of $(Y_n,b_n)$, a contradiction. Hence $N\ge1$.
    \end{proof}

    Uniform control of the surface tails rules out the noncompact alternative. For $K>0$, set $\mathfrak S_K:=\{\gamma(t):t>0,\ \mathcal N_{\rm s}(t)\le K\}$, and suppose that every nonempty $\mathfrak S_K$ satisfies
    \begin{equation}
        \lim_{\rho\to\infty}\sup_{(Y,a,W)\in\mathfrak S_K}\sup_{|x|>\rho}|Y(x)|=0.
        \label{4-41}
    \end{equation}
    By Corollary~\ref{cor4-7}, each $\mathfrak S_K$ is relatively compact, contradicting the sequence in \eqref{4-40}. Thus only \eqref{4-39} remains. For these controlled solution families the uniform $L^2$ tail condition \eqref{4-34} is equivalent. Neither condition is assumed in Proposition~\ref{prop4-8}.

    \begin{proof}[Proof of Theorem~\ref{thm1-1}]
        Proposition~\ref{prop3-3} and the change of variables \eqref{3-18} give the solution curve with its initial arc, local analytic reparametrizations, and eventual departure from the initial arc. Estimates \eqref{3-2} and \eqref{3-8} give the initial asymptotics \eqref{1-5}, while Proposition~\ref{prop2-1} supplies the physical reconstruction.

        The control quantity \eqref{4-38} agrees with \eqref{1-6}, and Proposition~\ref{prop4-8} gives the two alternatives. In the noncompact case, Lemma~\ref{lem4-2} and Proposition~\ref{prop4-6} provide the nontrivial centered solution, finitely many reflected irrotational profile pairs, their regularity and separation, and the strong remainder convergence in \eqref{1-8}. By \eqref{4-32} and $N\ge1$,
        \[
        \lim_{n\to\infty}\|Y_n-Y^{(0)}\|_{H^{s+2}}^2=2\sum_{j=1}^N\|Q_j\|_{H^{s+2}}^2>0.
        \]
        Since $Y_n\to Y^{(0)}$ locally, the sequence has no globally strongly convergent subsequence. The proof is complete.
    \end{proof}

    \bigskip

    \section{Effective drift and folds}
    \label{sec5}
    We use the reduced drift to characterize possible simple speed folds and count the nearby waves at fixed speed. Section~\ref{sec6} develops their variational interpretation.

    \subsection{Second derivatives of the reduced drift}
    Near a coupled solution $(Y_\star,a_\star,W_\star)$ with $A_\star$ invertible, Lemma~\ref{lem2-8} gives a real-analytic Bernoulli chart $Y=Y(a,W)$, with $A$ invertible throughout a smaller neighborhood. Set
    \[
    \mathscr G_0(a,W):=P(Y(a,W),a,W),\qquad L:=\partial_a\mathscr G_0,\qquad M:=\partial_W\mathscr G_0.
    \]
    The zeros of $\mathscr G_0$ are exactly the coupled solutions in this chart. Derivatives of $B,P,\mathscr F$ below treat $(Y,a,W)$ as independent variables, while those of $\mathscr G_0$ include the surface response. The functional from Lemma~\ref{lem2-8}, evaluated on this chart, eliminates second derivatives of the surface response:
    \[
    \Lambda(f,\rho):=\rho-P_YA^{-1}f,\qquad (f,\rho)\in\mathcal Y^s.
    \]

    \begin{lemma}
        \label{lem5-1}
        Define
        \begin{equation}
            e_a:=(Y_a,1,0)=(-A^{-1}B_a,1,0),\qquad e_W:=(Y_W,0,1)=(-A^{-1}B_W,0,1).
            \label{5-1}
        \end{equation}
        Then, for $i,j\in\{a,W\}$,
        \begin{equation}
            \partial_i\partial_j\mathscr G_0=\Lambda\bigl(D^2\mathscr F[e_i,e_j]\bigr).
            \label{5-2}
        \end{equation}
        In particular, $\partial_aL=\Lambda(D^2\mathscr F[e_a,e_a])$ and $\partial_WL=\partial_aM=\Lambda(D^2\mathscr F[e_a,e_W])$.
    \end{lemma}

    \begin{proof}
        Since $\mathscr F(Y(a,W),a,W)=(0,\mathscr G_0(a,W))$, twice differentiating yields
        \[
        D^2\mathscr F[e_i,e_j]+D_Y\mathscr F[Y_{ij}]=\begin{pmatrix}0\\ \partial_i\partial_j\mathscr G_0\end{pmatrix}.
        \]
        As $\Lambda(D_Y\mathscr F[h])=\Lambda(Ah,P_Y[h])=0$ for every $h$, applying $\Lambda$ gives \eqref{5-2}.
    \end{proof}

    In particular, the second depth derivative is
    \begin{equation}
        \begin{aligned}\partial_aL={}&P_{aa}+2P_{Ya}[Y_a]+P_{YY}[Y_a,Y_a]\\ &-P_YA^{-1}\bigl(B_{aa}+2B_{Ya}[Y_a]+B_{YY}[Y_a,Y_a]\bigr).\end{aligned}
        \label{5-3}
    \end{equation}
    The turning coefficient therefore includes the surface response; it is not simply $P_{aa}$.

    At a coupled solution with $L_\star=0$, Lemma~\ref{lem2-8} gives
    \[
    \ker\mathscr L_\star=\operatorname{span}\{\zeta_\star\},\qquad \zeta_\star:=(-A_\star^{-1}B_{a,\star},1),\qquad \operatorname{Ran}\mathscr L_\star=\ker\Lambda_\star.
    \]
    Moreover,
    \begin{equation}
        \Lambda_\star(\mathscr F_{W,\star})=M_\star,\qquad \Lambda_\star\bigl(D_{(Y,a)}^2\mathscr F_\star[\zeta_\star,\zeta_\star]\bigr)=(\partial_aL)_\star.
        \label{5-4}
    \end{equation}
    Hence $M_\star\ne0$ makes the speed direction transverse to $\operatorname{Ran}\mathscr L_\star$, while $(\partial_aL)_\star\ne0$ is the corresponding quadratic nondegeneracy along its kernel.

    These nondegeneracy conditions are unchanged when expressed in terms of the physical regular drift. Indeed, put $\alpha(a,W):=Z_{Y(a,W)}'(-\mathrm{i}a)>0$ and $\mathscr V_0:=\mathscr G_0/\alpha$, the regular relative horizontal velocity at the vortex by Proposition~\ref{prop2-1}. At an equilibrium,
    \[
    \partial_a\mathscr V_0=\frac{L}{\alpha},\qquad \partial_W\mathscr V_0=\frac{M}{\alpha};
    \]
    if also $L=0$, then $\partial_a^2\mathscr V_0=(\partial_aL)/\alpha$. Thus the conformal factor does not change the fold conditions, though all depth derivatives are taken in the conformal coordinate $a$.

    \begin{proposition}
        \label{prop5-2}
        Let $(Y_\star,a_\star,W_\star)$ be a coupled solution with $A_\star$ invertible and
        \begin{equation}
            L_\star=0,\qquad M_\star\ne0,\qquad (\partial_aL)_\star\ne0.
            \label{5-5}
        \end{equation}
        Set $\chi_\star:=-(\partial_aL)_\star/M_\star$. All sufficiently nearby coupled solutions lie on a regular real-analytic curve
        \[
        \gamma_\star(a):=(Y(a,w(a)),a,w(a)),\qquad w(a_\star)=W_\star,
        \]
        with
        \begin{equation}
            w(a)-W_\star=\frac{\chi_\star}{2}(a-a_\star)^2+O(|a-a_\star|^3)
            \label{5-6}
        \end{equation}
        and
        \begin{equation}
            Y(a,w(a))=Y_\star-A_\star^{-1}B_{a,\star}(a-a_\star)+O_{H^{s+2}}(|a-a_\star|^2).
            \label{5-7}
        \end{equation}
        In a sufficiently small configuration neighborhood, each nearby speed with $\chi_\star(W-W_\star)>0$ has exactly two solutions, at conformal depths
        \begin{equation}
            a_\pm(W)=a_\star\pm\sqrt{\frac{2(W-W_\star)}{\chi_\star}}+O(|W-W_\star|).
            \label{5-8}
        \end{equation}
        There are no nearby solutions when $\chi_\star(W-W_\star)<0$, and only the fold solution when $W=W_\star$. Along the curve,
        \begin{equation}
            L(a,w(a))=(\partial_aL)_\star(a-a_\star)+O(|a-a_\star|^2).
            \label{5-9}
        \end{equation}
        Thus $L$ changes sign across the fold. The fixed-speed linearization is invertible away from the fold, and the full derivative is surjective throughout the local curve.
    \end{proposition}

    \begin{proof}
        Since $M_\star\ne0$, the analytic implicit function theorem solves $\mathscr G_0(a,W)=0$ locally for $W=w(a)$. Uniqueness on the Bernoulli chart identifies all nearby coupled solutions. The curve is regular because its depth component is $a$. Differentiating the reduced equation gives $w'=-L/M$ and hence $w'(a_\star)=0$ and $w''(a_\star)=-(\partial_aL)_\star/M_\star=\chi_\star$. Taylor expansion proves \eqref{5-6}, while $Y_{a,\star}=-A_\star^{-1}B_{a,\star}$ gives \eqref{5-7}.

        For the exact fixed-speed count, write $\ell:=a-a_\star$ and $w(a_\star+\ell)-W_\star=\ell^2k(\ell)$, where $k$ is analytic and $k(0)=\chi_\star/2$. The analytic coordinate $\tau:=\ell\sqrt{2k(\ell)/\chi_\star}$ satisfies $\tau'(0)=1$ and puts the speed relation in the form
        \begin{equation}
            W-W_\star=\frac{\chi_\star}{2}\tau^2.
            \label{5-10}
        \end{equation}
        In a sufficiently small symmetric $\tau$-neighborhood this gives the asserted solution count. Since $\ell=\tau+O(\tau^2)$, it also gives \eqref{5-8}.

        Finally, $w'(a_\star)=0$ yields \eqref{5-9}. Shrinking the neighborhood keeps $M\ne0$ and $A$ invertible, with $L\ne0$ off the fold. Lemma~\ref{lem2-8} then gives the linearization claims; at the fold the tangent $(-A_\star^{-1}B_{a,\star},1,0)$ spans the kernel of the full derivative.
    \end{proof}

    The criterion is conditional: it does not assert that the global curve contains a fold. When it does, the speed has a strict local minimum if $\chi_\star>0$ and a strict local maximum if $\chi_\star<0$; the sign change of $L$ does not by itself determine dynamical stability.

    For the regularization in Section~\ref{sec7}, we express the fold conditions as zeros of $\mathfrak H_0(a,W):=(\mathscr G_0(a,W),L(a,W))$. At a fold satisfying \eqref{5-5},
    \begin{equation}
        \begin{aligned}D\mathfrak H_0(a_\star,W_\star)&=\begin{pmatrix}0&M_\star\\[1mm](\partial_aL)_\star&(\partial_WL)_\star\end{pmatrix},\\ \det D\mathfrak H_0(a_\star,W_\star)&=-M_\star(\partial_aL)_\star\ne0.\end{aligned}
        \label{5-11}
    \end{equation}
    Thus the fold is an isolated nondegenerate zero of $\mathfrak H_0$, a formulation used in Section~\ref{sec7} to establish its persistence under regularization.

    \smallskip

    \subsection{A uniform choice of local parameter}
    The fixed-speed inverse involves $L^{-1}$ and loses uniform control as a fold is approached. Allowing $W$ to vary, we add a scalar normalization to eliminate the tangent direction to the solution curve.

    \begin{lemma}
        \label{lem5-3}
        Let $(Y,a,W)\in\mathscr U^s$, with $A$ invertible and $L^2+M^2>0$. The bordered operator $\mathscr D_{\rm b}:\mathcal X^s\times\mathbb R\to\mathcal Y^s\times\mathbb R$ defined by
        \begin{equation}
            \mathscr D_{\rm b}(h,\ell,\nu):=\begin{pmatrix}Ah+B_a\ell+B_W\nu\\[1mm]P_Y[h]+P_a\ell+\nu\\[1mm]M\ell-L\nu\end{pmatrix}
            \label{5-12}
        \end{equation}
        is an isomorphism. For data $(f,\rho,\beta)\in H_{\rm e}^s(\mathbb R)\times\mathbb R^2$, its inverse is given by
        \begin{equation}
            \ell=\frac{L\Lambda(f,\rho)+M\beta}{L^2+M^2},\qquad \nu=\frac{M\Lambda(f,\rho)-L\beta}{L^2+M^2},
            \label{5-13}
        \end{equation}
        where $\Lambda(f,\rho):=\rho-P_YA^{-1}f$ and $h=A^{-1}(f-B_a\ell-B_W\nu)$. Moreover, for every compact subarc $\mathscr K\subset\mathscr C_0$ on which $A$ is invertible and $L^2+M^2>0$,
        \begin{equation}
            \sup_{(Y,a,W)\in\mathscr K}\|\mathscr D_{\rm b}^{-1}\|_{\mathcal L(\mathcal Y^s\times\mathbb R,\mathcal X^s\times\mathbb R)}\le C_{\mathscr K}.
            \label{5-14}
        \end{equation}
    \end{lemma}

    \begin{proof}
        Eliminating $h$ from \eqref{5-12} gives
        \[
        \begin{pmatrix}L&M\\ M&-L\end{pmatrix}\begin{pmatrix}\ell\\ \nu\end{pmatrix}=\begin{pmatrix}\Lambda(f,\rho)\\ \beta\end{pmatrix}.
        \]
        The coefficient matrix squares to $(L^2+M^2)I$, yielding \eqref{5-13} and the bounded inverse. In particular,
        \[
        \sqrt{\ell^2+\nu^2}=\frac{\sqrt{|\Lambda(f,\rho)|^2+|\beta|^2}}{\sqrt{L^2+M^2}},
        \]
        while
        \[
        \begin{aligned}
            |\Lambda(f,\rho)|&\le|\rho|+\|P_Y\|\,\|A^{-1}\|\,\|f\|_{H^s},\\
            \|h\|_{H^{s+2}}&\le\|A^{-1}\|\bigl(\|f\|_{H^s}+\|B_a\|_{H^s}|\ell|+\|B_W\|_{H^s}|\nu|\bigr).
        \end{aligned}
        \]
        On $\mathscr K$, continuity and compactness give uniform bounds for $A^{-1}$ and the remaining blocks, together with $\inf_{\mathscr K}(L^2+M^2)>0$. These estimates prove \eqref{5-14}.
    \end{proof}

    At a fold, where $L=0$ and $M\ne0$, the homogeneous normalization $\beta=0$ fixes the conformal depth to first order while allowing a speed correction. When $M=0$ and $L\ne0$, it instead fixes the speed. The bordered system is therefore invertible across changes of local parameter.

    Fix a coupled solution $(Y_\circ,a_\circ,W_\circ)$ satisfying the hypotheses of Lemma~\ref{lem5-3}, and use the fixed coefficients $L_\circ,M_\circ$ to define
    \begin{equation}
        \mathscr F(Y,a,W)=0,\qquad M_\circ(a-a_\circ)-L_\circ(W-W_\circ)=\theta.
        \label{5-15}
    \end{equation}
    The derivative with respect to $(Y,a,W)$ at $\theta=0$ is $\mathscr D_{{\rm b},\circ}$. Hence the analytic implicit function theorem gives a unique local real-analytic solution curve $\gamma_\circ(\theta)$, containing all sufficiently nearby coupled solutions as $\theta$ varies. Differentiating \eqref{5-15} and using \eqref{5-13} yields
    \begin{equation}
        \gamma_\circ'(0)=\frac1{L_\circ^2+M_\circ^2}\bigl(-A_\circ^{-1}(M_\circ B_{a,\circ}-L_\circ B_{W,\circ}),M_\circ,-L_\circ\bigr).
        \label{5-16}
    \end{equation}
    This is the kernel direction in Lemma~\ref{lem2-8}, normalized by \eqref{5-15}. At a fold, its speed component vanishes but its depth component equals $M_\circ^{-1}\ne0$, so the curve remains regular.

    A compact subarc satisfying Lemma~\ref{lem5-3} is covered by finitely many such analytic charts, all subject to the uniform inverse bound \eqref{5-14}. They provide the point-vortex reference curves for the regularization in Section~\ref{sec7}, where the corresponding smooth-core estimates and convergence are established.

    \bigskip

    \section{Variational interpretation}
    \label{sec6}
    We relate the folds of Section~\ref{sec5} to the fixed-speed energy--momentum Hessian. Since a surface variation at fixed conformal depth generally displaces the physical vortex, we use its physical vertical coordinate to separate the surface and vortex variations and identify the scalar coefficient responsible for the change in negative index. Throughout, $W$ is fixed in each variation, and the index statements concern the static Hessian in the symmetric class, not dynamical stability. The point-vortex calculation also serves as a reference for the joint dual variational analysis of Part~III \cite{DuZouPartIII} and the dynamical results of Part~V \cite{DuZouPartV}.

    \subsection{The action in conformal coordinates}
    Fix $s>\frac12$ and $0<W<c_{\min}$, and write $\vartheta_Y:=1+|D|Y$, $J:=J(Y)$, and $\alpha:=\alpha(Y,a)=Z_Y'(-\mathrm{i}a)>0$. The gravitational and capillary potential energy is
    \begin{equation}
        \mathcal V(Y):=\frac g2\int_{\mathbb R}Y^2\vartheta_Y\,dx+\sigma\int_{\mathbb R}\bigl(\sqrt{J}-\vartheta_Y\bigr)\,dx.
        \label{6-1}
    \end{equation}
    The two terms in the capillary integrand are kept together because $\sqrt{J}-\vartheta_Y=(Y')^2/(\sqrt{J}+\vartheta_Y)$ is integrable under the graph condition, while its two terms need not be integrable separately. Set also $\mathcal I(Y):=\int_{\mathbb R}Y|D|Y\,dx$. We establish the identities first for smooth decaying data and then extend them to the stated Sobolev spaces by continuity.

    For the reconstructed flow of Proposition~\ref{prop2-1}, the laboratory-frame stream function satisfies
    \[
    \psi_{\rm lab}(Z_Y(x+\mathrm{i}y))=\psi_a(x,y)-We^{y|D|}Y(x).
    \]
    We renormalize its kinetic energy by excising a physical disk of radius $\delta$ about the vortex, subtracting $\Gamma^2(4\pi)^{-1}\log(1/\delta)$, and letting $\delta\to0$. Integration by parts removes the Dirichlet cross term; the harmonic contribution is $W^2\mathcal I(Y)/2$, and the vortex contribution is the Robin term computed in \cite[Section~4]{DuZouPartI}. Consequently,
    \begin{equation}
        E=\mathcal V(Y)+\frac{W^2}{2}\mathcal I(Y)+\frac{\Gamma^2}{4\pi}\log\bigl(2a\alpha\bigr).
        \label{6-2}
    \end{equation}

    With the momentum convention of \cite[Section~5.2]{VarholmWahlenWalsh2020},
    \[
    \mathcal P=\Gamma\operatorname{Im}p_a-\int_{\mathbb R}\eta_Y'(\xi)\varphi_{\rm tot}(\xi)\,d\xi,
    \]
    where $\varphi_{\rm tot}$ is the laboratory-frame velocity-potential trace. Up to an additive constant, it satisfies
    \[
    \begin{aligned}
        \frac{d}{dx}\bigl(\varphi_{\rm tot}(\Xi_Y(x))\bigr)&=q_{a,W}-W\vartheta_Y,\\
        \varphi_{\rm tot}(\Xi_Y(x))&=WHY(x)-\frac{\Gamma}{\pi}\arctan\frac xa.
    \end{aligned}
    \]
    Since $\operatorname{Im}p_a=-a+E_aY$, integration by parts gives
    \begin{equation}
        \mathcal P=\Gamma(-a+E_aY)-W\mathcal I(Y)-\Gamma E_aY=-W\mathcal I(Y)-\Gamma a.
        \label{6-3}
    \end{equation}
    The arctangent contribution is evaluated through the Poisson pairing $E_aY$; no $L^1$ assumption on $Y$ is needed.

    For waves traveling with velocity $-We_1$, the energy--momentum action is $E+W\mathcal P$. Its restriction to the reconstructed flow is
    \begin{equation}
        \mathcal J_W(Y,a):=\mathcal V(Y)-\frac{W^2}{2}\mathcal I(Y)+\frac{\Gamma^2}{4\pi}\log\bigl(2a\alpha(Y,a)\bigr)-\Gamma Wa.
        \label{6-4}
    \end{equation}
    The last two terms give the Kirchhoff--Routh function $K_{\Omega_Y,W}(a)$ from Part~I, retaining its dependence on $Y$. Define the surface variation operator
    \[
    \mathcal C_Yh:=\vartheta_Yh+Y'Hh.
    \]
    As $\mathcal C_Y=J^{3/2}\mathcal T_Y$, Lemma~\ref{lem2-6} makes it an isomorphism on $L^2(\mathbb R)$ preserving the even subspace.

    \begin{lemma}
        \label{lem6-1}
        The functional $\mathcal J_W:\mathcal O^s\times(0,\infty)\to\mathbb R$ is real analytic. For every $(Y,a)$ in this domain and $(h,\ell)\in\mathcal X^s$,
        \begin{equation}
            D\mathcal J_W(Y,a)[h,\ell]=\int_{\mathbb R}B(Y,a,W)\mathcal C_Yh\,dx+\Gamma P(Y,a,W)\left(\frac{E_ah}{\alpha(Y,a)}-\ell\right).
            \label{6-5}
        \end{equation}
        In particular, $\partial_a\mathcal J_W=-\Gamma P$, and $D\mathcal J_W(Y,a)=0$ if and only if $\mathscr F(Y,a,W)=0$.
    \end{lemma}

    \begin{proof}
        Real analyticity follows from Lemma~\ref{lem2-2}, using the integrable capillary expression in \eqref{6-1}, the Sobolev product and composition estimates, and the positivity of $a$ and $\alpha$ in the logarithmic term.

        Differentiating $\eta_Y(\Xi_Y(x))=Y(x)$ gives the physical graph variation
        \[
        D\eta_Y[h](\Xi_Y(x))=h+\frac{Y'}{\vartheta_Y}Hh.
        \]
        Since $d\xi=\vartheta_Y\,dx$, we obtain
        \[
        D\mathcal V(Y)[h]=\int_{\mathbb R}(gY+\sigma\kappa(Y))\mathcal C_Yh\,dx.
        \]
        Together with $D\mathcal I(Y)[h]=2\int_{\mathbb R}|D|Y\,h\,dx$ and $D_Y\alpha[h]=E_a(|D|h)$, this yields
        \begin{equation}
            D_Y\mathcal J_W[h]=\int_{\mathbb R}(gY+\sigma\kappa(Y))\mathcal C_Yh\,dx-W^2\int_{\mathbb R}|D|Y\,h\,dx+\frac{\Gamma^2}{4\pi\alpha}E_a(|D|h).
            \label{6-6}
        \end{equation}

        To identify the kinetic contribution to $B$, write $k_a(x):=a/[\pi(a^2+x^2)]$, so $q_{a,W}=W-\Gamma k_a$. The real $L^2$ adjoint $\mathcal C_Y^*f=\vartheta_Yf-H(Y'f)$ satisfies
        \begin{equation}
            \mathcal C_Y^*(J^{-1}-1)=-2|D|Y,\qquad \mathcal C_Y^*\left(\frac{k_a}{J}\right)=\frac{k_a}{\alpha}.
            \label{6-7}
        \end{equation}
        The first identity follows from the holomorphic boundary relation $H(Y'/J)=\vartheta_Y/J-1$ and $HY'=-|D|Y$. For the second, observe that
        \[
        \frac{\mathcal C_Yh}{J}=\operatorname{Re}\frac{h+\mathrm{i}Hh}{Z_Y'},\qquad E_a\left(\frac{\mathcal C_Yh}{J}\right)=\frac{E_ah}{\alpha}.
        \]
        The latter identity follows by evaluating the holomorphic quotient at $-\mathrm{i}a$, where $Z_Y'=\alpha$ is real. Poisson duality then gives the second identity in \eqref{6-7}.

        Using $k_a^2=(k_a/a-\partial_ak_a)/(2\pi)$ and $\partial_ak_a=-|D|k_a$, differentiation of \eqref{6-7} at fixed $Y$ gives
        \begin{equation}
            \mathcal C_Y^*\left(\frac{k_a^2}{J}\right)=\frac1{2\pi\alpha}\left[|D|k_a+\left(\frac1a+\frac{\alpha_a}{\alpha}\right)k_a\right],
            \label{6-8}
        \end{equation}
        where $\alpha_a=E_a(Y'')$. Expanding $q_{a,W}=W-\Gamma k_a$ and applying \eqref{6-7}--\eqref{6-8}, we obtain
        \begin{equation}
            \mathcal C_Y^*\left(\frac{q_{a,W}^2}{2J}-\frac{W^2}{2}\right)=-W^2|D|Y+\frac{\Gamma^2}{4\pi\alpha}|D|k_a-\frac{\Gamma P}{\alpha}k_a.
            \label{6-9}
        \end{equation}
        Here the coefficient of $k_a$ follows from $P=W-\Gamma(4\pi)^{-1}(a^{-1}+\alpha_a/\alpha)$. Pairing \eqref{6-9} with $h$ and using $\int_{\mathbb R}|D|k_a\,h\,dx=E_a(|D|h)$, comparison with \eqref{6-6} gives
        \[
        D_Y\mathcal J_W[h]=\int_{\mathbb R}B\mathcal C_Yh\,dx+\frac{\Gamma P}{\alpha}E_ah.
        \]
        On the other hand, direct differentiation of \eqref{6-4} at fixed $Y$ gives $\partial_a\mathcal J_W=-\Gamma P$, proving \eqref{6-5}.

        If $D\mathcal J_W=0$, variation in $a$ gives $P=0$, and hence $\int_{\mathbb R}B\mathcal C_Yh\,dx=0$ for every $h\in H_{\rm e}^{s+2}$. Density and invertibility of $\mathcal C_Y$ on $L_{\rm e}^2$ imply $B=0$. The converse follows immediately from \eqref{6-5}.
    \end{proof}

    In \eqref{6-5}, the term $\Gamma PE_ah/\alpha$ accounts for the displacement $D_Yp_a[h]=\mathrm{i}E_ah$ of the physical vortex under a surface variation at fixed conformal depth. Thus $B=0$ alone does not imply surface stationarity of $\mathcal J_W$ in these coordinates.

    On the Bernoulli chart of Section~\ref{sec5}, set
    \begin{equation}
        \Phi_{\rm B}(a,W):=\mathcal J_W(Y(a,W),a),\qquad \beta_{\rm g}(a,W):=1-\frac{E_aY_a}{\alpha(a,W)}.
        \label{6-10}
    \end{equation}
    Here $Y_a$ is taken at fixed $W$, and $\alpha(a,W)=1+E_a(|D|Y(a,W))$. Since $B=0$ on this chart, \eqref{6-5} applied to $(Y_a,1)$ yields
    \begin{equation}
        \partial_a\Phi_{\rm B}=-\Gamma\beta_{\rm g}\mathscr G_0.
        \label{6-11}
    \end{equation}
    At a coupled equilibrium, $\mathscr G_0=0$, and differentiation gives
    \begin{equation}
        \partial_a^2\Phi_{\rm B}=-\Gamma\beta_{\rm g}L.
        \label{6-12}
    \end{equation}

    The geometric meaning of $\beta_{\rm g}$ follows from the physical vortex height $p_2(a,W)=-a+E_aY(a,W)$:
    \begin{equation}
        \partial_ap_2=-1-E_a(|D|Y)+E_aY_a=-\alpha\beta_{\rm g}.
        \label{6-13}
    \end{equation}
    Thus $\beta_{\rm g}$ measures the change of physical vortex height relative to conformal depth after the surface response is included. On the small-speed branch, Lemma~\ref{lem3-1} gives $\|Y_a\|_{L^2}=O(W^{5/2})$. Since $a\asymp W^{-1}$, the Poisson bound yields $E_aY_a=O(W^3)$; together with $\alpha=1+O(W^3)$, we obtain $\beta_{\rm g}=1+O(W^3)>0$. We impose no global sign condition on $\beta_{\rm g}$.

    The function $\Phi_{\rm B}$ is restricted by $B=0$, not by surface stationarity at fixed $a$. Consequently, \eqref{6-12} is not yet the scalar Schur complement of the fixed-speed Hessian. We identify that coefficient next by eliminating the surface variation at fixed physical vortex height.

    \smallskip

    \subsection{The Hessian and the physical vortex coordinate}
    Fix $W$ and use the real $L^2$ pairing. At a coupled equilibrium, write $\mathcal H:=D^2\mathcal J_W(Y,a)$ and $\mathcal E h:=E_ah/\alpha$. The physical vortex height satisfies
    \[
    p_2(Y,a)=-a+E_aY,\qquad Dp_2(Y,a)[h,\ell]=E_ah-\alpha\ell.
    \]
    Since $\partial_ap_2=-\alpha<0$ at fixed $Y$, the change of variables $(Y,a)\mapsto(Y,p_2)$ is locally real analytic and invertible. With inverse depth map $a=\mathfrak a(Y,p_2)$, define
    \begin{equation}
        \widetilde{\mathcal J}_W(Y,p_2):=\mathcal J_W(Y,\mathfrak a(Y,p_2)).
        \label{6-14}
    \end{equation}
    A variation $(h,\nu)$ in these coordinates corresponds to $\ell=\mathcal E h-\nu/\alpha$. Here $\nu$ denotes vertical displacement of the physical vortex, not variation of the wave speed.

    \begin{lemma}
        \label{lem6-2}
        At a coupled equilibrium, the Hessian satisfies
        \begin{equation}
            \begin{aligned}
                \mathcal H[(h,\ell),(k,m)]={}&\langle Ah+B_a\ell,\mathcal C_Yk\rangle\\
                &+\Gamma\bigl(P_Y[h]+P_a\ell\bigr)(\mathcal E k-m).
            \end{aligned}
            \label{6-15}
        \end{equation}
        Its kernel, understood in the bilinear-form sense, equals $\ker\mathscr L$, where $\mathscr L=D_{(Y,a)}\mathscr F$.

        Suppose also that $A$ is invertible, and write $Y_a=-A^{-1}B_a$ and $\beta_{\rm g}=1-\mathcal E Y_a$ on the Bernoulli chart. The surface Hessian at fixed $p_2$ is represented by
        \begin{equation}
            \widetilde{\mathcal H}_{\rm s}=\mathcal C_Y^*\widetilde A,\qquad \widetilde A:=A+B_a\mathcal E=A(I-Y_a\mathcal E),
            \label{6-16}
        \end{equation}
        where $(Y_a\mathcal E)h:=Y_a\,\mathcal E h$. It is an isomorphism from $H_{\rm e}^{s+2}$ to $H_{\rm e}^s$ if and only if $\beta_{\rm g}\ne0$. At an equilibrium with $L=0$, this condition holds automatically.

        When $A$ is invertible and $\beta_{\rm g}\ne0$, surface stationarity at fixed $p_2$ determines a unique local real-analytic map $Y=\widetilde Y(p_2,W)$. The reduced action
        \begin{equation}
            \Phi(p_2,W):=\widetilde{\mathcal J}_W(\widetilde Y(p_2,W),p_2)
            \label{6-17}
        \end{equation}
        satisfies, at every coupled equilibrium in this chart,
        \begin{equation}
            \Phi_{p_2p_2}=-\frac{\Gamma L}{\alpha^2\beta_{\rm g}}.
            \label{6-18}
        \end{equation}
    \end{lemma}

    \begin{proof}
        Differentiating \eqref{6-5} at $B=P=0$ gives \eqref{6-15}, since derivatives of $\mathcal C_Y$ and $\mathcal E$ are multiplied by the vanishing residuals. If $(h,\ell)$ annihilates every test variation, testing with $(0,m)$ gives $P_Y[h]+P_a\ell=0$. Testing with $(k,0)$ then yields $\langle Ah+B_a\ell,\mathcal C_Yk\rangle=0$ for every $k\in H_{\rm e}^{s+2}$. Density and invertibility of $\mathcal C_Y$ on $L_{\rm e}^2$ imply $Ah+B_a\ell=0$. The converse follows from \eqref{6-15}, proving $\ker\mathcal H=\ker\mathscr L$.

        Substituting $\ell=\mathcal E h-\nu/\alpha$ into \eqref{6-5} gives, throughout the coordinate neighborhood,
        \begin{equation}
            D\widetilde{\mathcal J}_W(Y,p_2)[h,\nu]=\langle B,\mathcal C_Yh\rangle+\frac{\Gamma P}{\alpha}\nu,
            \label{6-19}
        \end{equation}
        with residuals evaluated at $(Y,\mathfrak a(Y,p_2),W)$. Thus surface stationarity at fixed $p_2$ is equivalent to $B=0$. Its linearization in $Y$ is $\widetilde A=A+B_a\mathcal E$, and differentiation of the surface gradient $\mathcal C_Y^*B$ at equilibrium gives \eqref{6-16}.

        Since $AY_a=-B_a$, we have $\widetilde A=A(I-Y_a\mathcal E)$. If $\beta_{\rm g}\ne0$, the rank-one factor has inverse
        \[
        (I-Y_a\mathcal E)^{-1}=I+\frac{Y_a\mathcal E}{\beta_{\rm g}}.
        \]
        If $\beta_{\rm g}=0$, then $(I-Y_a\mathcal E)Y_a=0$ with $Y_a\ne0$. Moreover, Lemma~\ref{lem2-6} and the Sobolev product estimates show that $\mathcal C_Y^*$ is an isomorphism on $H_{\rm e}^s$, with inverse $(\mathcal C_Y^*)^{-1}f=(\vartheta_Yf+H(Y'f))/J$. This proves the criterion for $\widetilde{\mathcal H}_{\rm s}$.

        For nonvanishing at $L=0$, Lemma~\ref{lem2-8} gives $\zeta:=(Y_a,1)\in\ker\mathscr L=\ker\mathcal H$. By symmetry and \eqref{6-15},
        \[
        0=\mathcal H[(h,0),\zeta]=\langle Ah,\mathcal C_YY_a\rangle-\Gamma\beta_{\rm g}P_Y[h].
        \]
        If $\beta_{\rm g}=0$, surjectivity of $A$ and density in $L_{\rm e}^2$ imply $\mathcal C_YY_a=0$. Hence $Y_a=0$, contradicting $\beta_{\rm g}=1-\mathcal E Y_a=0$. Thus $\beta_{\rm g}\ne0$.

        When $\beta_{\rm g}\ne0$, the analytic implicit function theorem applied to $B(Y,\mathfrak a(Y,p_2),W)=0$ gives $Y=\widetilde Y(p_2,W)$. Equation~\eqref{6-19} then yields
        \[
        \Phi_{p_2}=\frac{\Gamma P}{\alpha}.
        \]
        By local uniqueness, this stationary surface family is the Bernoulli chart expressed in $(p_2,W)$. From \eqref{6-13},
        \[
        \left.\frac{\partial a}{\partial p_2}\right|_W=-\frac1{\alpha\beta_{\rm g}}.
        \]
        At equilibrium $P=\mathscr G_0=0$, so differentiation gives $\Phi_{p_2p_2}=(\Gamma/\alpha)L(-1/(\alpha\beta_{\rm g}))$, proving \eqref{6-18}.
    \end{proof}

    At every coupled equilibrium, the physical-coordinate Hessian has the block representation
    \[
    D^2\widetilde{\mathcal J}_W=\begin{pmatrix}\widetilde{\mathcal H}_{\rm s}&\widetilde k\\ \widetilde k^*&\widetilde d\end{pmatrix},
    \]
    where
    \[
    \widetilde{\mathcal H}_{\rm s}=\mathcal C_Y^*(A+B_a\mathcal E),\qquad \widetilde k=-\alpha^{-1}\mathcal C_Y^*B_a,\qquad \widetilde d=-\Gamma\alpha^{-2}P_a,
    \]
    and $\widetilde k^*h=\langle\widetilde k,h\rangle$. These identities follow from \eqref{6-19} at $B=P=0$ and require no invertibility assumption.

    When $A$ is invertible and $\beta_{\rm g}\ne0$, differentiating the surface stationarity equation gives $\widetilde Y_{p_2}=-\widetilde{\mathcal H}_{\rm s}^{-1}\widetilde k$. Hence the scalar Schur complement is
    \begin{equation}
        \widetilde d-\langle\widetilde k,\widetilde{\mathcal H}_{\rm s}^{-1}\widetilde k\rangle=\Phi_{p_2p_2}=-\frac{\Gamma L}{\alpha^2\beta_{\rm g}}.
        \label{6-20}
    \end{equation}
    The elimination uses stationarity, not minimization; the surface Hessian may be indefinite. The resulting quadratic-form decomposition will be established on the common form domain in Proposition~\ref{prop6-4}. On the common local chart, $\Phi_{\rm B}(a,W)=\Phi(p_2(a,W),W)$; since $\Phi_{p_2}=0$ at equilibrium, $\partial_a^2\Phi_{\rm B}=\Phi_{p_2p_2}(\partial_ap_2)^2$, recovering \eqref{6-12}.

    At a simple fold from Proposition~\ref{prop5-2}, Lemma~\ref{lem6-2} gives $\beta_{{\rm g},\star}\ne0$, so the physical-coordinate surface Hessian remains invertible nearby. Along $W=w(a)$, using $L_\star=w'(a_\star)=0$, we obtain
    \begin{equation}
        \left.\frac{d}{da}\left(-\frac{\Gamma L(a,w(a))}{\alpha(a,w(a))^2\beta_{\rm g}(a,w(a))}\right)\right|_{a=a_\star}=-\frac{\Gamma(\partial_aL)_\star}{\alpha_\star^2\beta_{{\rm g},\star}}\ne0.
        \label{6-21}
    \end{equation}
    Thus the scalar Schur complement crosses zero nondegenerately at the fold. To determine the resulting change in negative index, we first place the Hessian forms on a common domain and establish their closedness and finite negative index.

    \smallskip

    \subsection{The common form domain}
    We extend the Hessian forms to fixed spaces of variations, without extending the nonlinear action itself to an open set of $H^1$ configurations. Set
    \[
    \begin{aligned}
        \mathscr H_{\rm s}&:=L_{\rm e}^2(\mathbb R),\qquad \mathscr V_{\rm s}:=H_{\rm e}^1(\mathbb R),\\
        \mathscr H&:=L_{\rm e}^2(\mathbb R)\times\mathbb R,\qquad \mathscr V:=H_{\rm e}^1(\mathbb R)\times\mathbb R.
    \end{aligned}
    \]
    At a coupled equilibrium, with $W$ fixed, define initially
    \[
    \mathfrak q_{\rm s}[h]:=D^2\widetilde{\mathcal J}_W[(h,0),(h,0)],\qquad \mathfrak q[(h,\nu)]:=D^2\widetilde{\mathcal J}_W[(h,\nu),(h,\nu)]
    \]
    for $h\in H_{\rm e}^{s+2}$ and physical vertical vortex displacement $\nu\in\mathbb R$. We use two arguments for the associated bilinear forms. The negative index is the maximal dimension of a subspace on which the form is strictly negative. Its kernel consists of variations pairing to zero with every test variation.

    \begin{lemma}
        \label{lem6-3}
        Let $(Y,a,W)\in\mathscr U^s$ be a coupled equilibrium. The forms $\mathfrak q_{\rm s}$ and $\mathfrak q$ extend continuously to $\mathscr V_{\rm s}$ and $\mathscr V$, respectively. They are closed and bounded below on $\mathscr H_{\rm s}$ and $\mathscr H$, and define self-adjoint operators with these form domains. Moreover,
        \begin{equation}
            n_-(\mathfrak q_{\rm s})<\infty,\qquad n_-(\mathfrak q)<\infty.
            \label{6-22}
        \end{equation}
        If $A$ is invertible and $\beta_{\rm g}\ne0$, the surface form induces an isomorphism $\widetilde{\mathcal H}_{\rm s}:H_{\rm e}^1\to H_{\rm e}^{-1}$. Its negative index is locally constant along the coupled solution set where these conditions hold.
    \end{lemma}

    \begin{proof}
        At fixed conformal depth, set $\mathfrak q_a[h]:=D_{YY}^2\mathcal J_W[h,h]$. Differentiating \eqref{6-4} gives
        \begin{equation}
            \begin{aligned}
                \mathfrak q_a[h]={}&\sigma\int_{\mathbb R}J^{-3/2}|\mathcal C_Y(h')|^2\,dx+g\int_{\mathbb R}\vartheta_Yh^2\,dx\\
                &+2g\int_{\mathbb R}Yh|D|h\,dx-W^2\langle h,|D|h\rangle-\frac{\Gamma^2}{4\pi\alpha^2}\bigl(E_a(|D|h)\bigr)^2.
            \end{aligned}
            \label{6-23}
        \end{equation}
        Indeed, the capillary second variation is $J^{-3/2}(\vartheta_Yh'-Y'|D|h)^2=J^{-3/2}|\mathcal C_Y(h')|^2$. At fixed physical vortex position, $\ell=\mathcal E h=E_ah/\alpha$. Since the action is stationary at the equilibrium, only the first derivative of this coordinate change enters the Hessian, yielding
        \begin{equation}
            \mathfrak q_{\rm s}[h]=\mathfrak q_a[h]-2\Gamma\,\mathcal E h\,P_Y[h]-\Gamma P_a(\mathcal E h)^2.
            \label{6-24}
        \end{equation}
        The Poisson terms in \eqref{6-23}--\eqref{6-24} are finite-rank forms bounded on $L_{\rm e}^2$.

        All terms extend continuously to $H_{\rm e}^1$. The graph bounds and invertibility of $\mathcal C_Y$ on $L^2$ imply
        \[
        \sigma\int_{\mathbb R}J^{-3/2}|\mathcal C_Y(h')|^2\,dx\ge c\|h'\|_{L^2}^2.
        \]
        The first-order terms are bounded by $C\|h\|_{L^2}\|h'\|_{L^2}$, while the remaining terms are bounded on $L^2$. Hence Young's inequality gives
        \begin{equation}
            \mathfrak q_{\rm s}[h]\ge c_1\|h'\|_{L^2}^2-C_1\|h\|_{L^2}^2,\qquad |\mathfrak q_{\rm s}[h]|\le C_2\|h\|_{H^1}^2.
            \label{6-25}
        \end{equation}
        Adding a sufficiently large multiple of $\|h\|_{L^2}^2$ gives a form norm equivalent to the $H^1$ norm. This proves closedness and the lower bound. By the block identities following Lemma~\ref{lem6-2},
        \[
        \mathfrak q[(h,\nu)]=\mathfrak q_{\rm s}[h]+2\nu\langle\widetilde k,h\rangle+\widetilde d\,\nu^2,
        \]
        where $\widetilde k=-\alpha^{-1}\mathcal C_Y^*B_a\in L_{\rm e}^2$ and $\widetilde d=-\Gamma\alpha^{-2}P_a\in\mathbb R$. These additional terms are bounded on $\mathscr H$, so the full form is likewise closed and bounded below. The representation theorem for closed lower-bounded forms gives the associated self-adjoint operators \cite{Kato1995}.

        For finiteness of the negative indices, compare the surface form with a positive reference form that includes the variable capillary coefficient. Let $M_f$ denote multiplication by $f$, set $[H,M_f]:=HM_f-M_fH$, and define $b_1:=J^{-3/2}\vartheta_YY'$ and $b_2:=J^{-3/2}(Y')^2$. Since $\mathcal C_Y=M_{\vartheta_Y}+M_{Y'}H$, $H^*=-H$, and $H^2=-I$,
        \begin{equation}
            \mathcal C_Y^*M_{J^{-3/2}}\mathcal C_Y=M_{J^{-1/2}}+\mathcal K_{\rm c},\qquad \mathcal K_{\rm c}:=-[H,M_{b_1}]-[H,M_{b_2}]H.
            \label{6-26}
        \end{equation}
        For $f\in H^1(\mathbb R)$, the Hilbert-transform commutator is Hilbert--Schmidt, with
        \[
        \|[H,M_f]\|_{\mathrm{HS}}^2=\frac1{\pi^2}\iint_{\mathbb R^2}\frac{|f(x)-f(y)|^2}{|x-y|^2}\,dx\,dy\le C\|f\|_{H^1}^2.
        \]
        Since $b_1,b_2\in H^{s+1}\subset H^1$, $\mathcal K_{\rm c}$ is compact and self-adjoint on $L^2$. Thus the capillary form equals
        \[
        \sigma\int_{\mathbb R}J^{-1/2}|h'|^2\,dx+\sigma\langle\mathcal K_{\rm c}h',h'\rangle,
        \]
        where the second term induces a compact map $H^1\to H^{-1}$.

        Set $m_Y:=J^{-1/4}$. Multiplication by $m_Y$ and $m_Y^{-1}$ is bounded on $H^1$, and we define
        \begin{equation}
            \mathfrak q_{\rm ref}[h]:=g\|m_Yh\|_{L^2}^2+\sigma\|(m_Yh)'\|_{L^2}^2-W^2\||D|^{1/2}(m_Yh)\|_{L^2}^2.
            \label{6-27}
        \end{equation}
        The subcritical symbol bound \eqref{2-25} yields $\mathfrak q_{\rm ref}[h]\ge c\|m_Yh\|_{H^1}^2\ge c'\|h\|_{H^1}^2$. Since $m_Y^2=J^{-1/2}$, the reference form has the same leading capillary coefficient as $\mathfrak q_{\rm s}$. We claim that
        \begin{equation}
            \mathfrak q_{\rm s}=\mathfrak q_{\rm ref}+\mathfrak r,
            \label{6-28}
        \end{equation}
        where $\mathfrak r$ is a compact symmetric form, in the sense that its associated operator $H_{\rm e}^1\to H_{\rm e}^{-1}$ is compact.

        Indeed, multiplication by a bounded coefficient vanishing at infinity is compact from $H^1$ to $L^2$, by cutoff approximation and local Rellich compactness. For the capillary difference,
        \[
        \int_{\mathbb R}J^{-1/2}|h'|^2\,dx-\int_{\mathbb R}|(m_Yh)'|^2\,dx=-2\int_{\mathbb R}m_Ym_Y'hh'\,dx-\int_{\mathbb R}|m_Y'|^2h^2\,dx.
        \]
        Since $m_Y'\in C_0(\mathbb R)\cap L^\infty$, these terms are compact, as is the commutator term in \eqref{6-26}. The gravitational differences involve multiplication by $\vartheta_Y-m_Y^2$ and $Y$, both vanishing at infinity, and are compact by the same argument. For the kinetic difference, writing $d_Y:=m_Y-1$ gives
        \[
        \begin{aligned}
            \langle m_Yh,|D|(m_Yh)\rangle-\langle h,|D|h\rangle={}&2\langle d_Yh,|D|h\rangle\\
            &+\langle d_Yh,|D|(d_Yh)\rangle.
        \end{aligned}
        \]
        Multiplication by $d_Y$ is compact from $H^1$ to $L^2$ and bounded on $H^1$. Since $|D|:H^1\to L^2$ is bounded, polarization of the kinetic differences shows that their associated operators factor through the compact map $M_{d_Y}:H^1\to L^2$ and hence are compact from $H^1$ to $H^{-1}$. The remaining Poisson terms are finite rank. This proves \eqref{6-28}.

        Using $\mathfrak q_{\rm ref}$ as an equivalent inner product on $H_{\rm e}^1$, we may write
        \[
        \mathfrak q_{\rm s}[h,k]=\mathfrak q_{\rm ref}[(I+\mathfrak K)h,k]
        \]
        with $\mathfrak K$ compact and self-adjoint. Its negative index is the total multiplicity of the eigenvalues of $\mathfrak K$ below $-1$ and is therefore finite. The same argument, using $\mathfrak q_{\rm ref}[h]+\nu^2$ as reference form and the finite-rank additional blocks, proves \eqref{6-22} for $\mathfrak q$.

        Now suppose $A$ is invertible and $\beta_{\rm g}\ne0$. We next exclude additional kernel vectors in $H_{\rm e}^1$. If $\mathfrak q_{\rm s}[h,k]=0$ for every $k\in H_{\rm e}^1$, then \eqref{6-16} gives $\mathcal C_Y^*\widetilde Ah=0$ in $H_{\rm e}^{-1}$. The explicit inverse formulas for $\mathcal C_Y^*$ and $\mathcal T_Y$, together with the $W^{1,\infty}$ bounds on their coefficients, extend to $H^{-1}$ by duality and density. Thus $\widetilde Ah=0$ in $H_{\rm e}^{-1}$, equivalently
        \[
        \mathcal T_Y(-\sigma h'')+\mathfrak c\,|D|h+\mathfrak d\,h'+gh+B_a\mathcal E h=0,
        \]
        where $\mathfrak c,\mathfrak d$ are the coefficients from Lemma~\ref{lem2-4}. For $h\in H^1$, all terms except $\mathcal T_Y(-\sigma h'')$ belong to $L^2$. Invertibility of $\mathcal T_Y$ on $L^2$ gives $h''\in L^2$, hence $h\in H^2$. The Sobolev multiplier estimates then yield, for $h\in H^r$ with $r\ge2$,
        \[
        h\in H^{2+\min\{s,r-1\}}.
        \]
        Iteration gives $h\in H_{\rm e}^{s+2}$. The invertibility of $\widetilde{\mathcal H}_{\rm s}$ on this space from Lemma~\ref{lem6-2} forces $h=0$. Since \eqref{6-28} makes the surface form Fredholm of index zero from $H_{\rm e}^1$ to $H_{\rm e}^{-1}$, its trivial kernel implies surjectivity and a bounded inverse.

        Finally, \eqref{6-23}--\eqref{6-24} depend continuously on the configuration in $\mathcal L(H_{\rm e}^1,H_{\rm e}^{-1})$. The associated bounded self-adjoint operators in a fixed equivalent inner product are therefore norm-continuous. Local invertibility keeps the spectrum away from zero, and the finite negative index is therefore locally constant.
    \end{proof}

    Near a simple fold with $A_\star$ invertible, Lemmas~\ref{lem6-2} and \ref{lem6-3} ensure that the surface Hessian remains invertible on the common form domain. This permits an exact Schur-complement reduction of the full Hessian without assuming that the surface form is positive.

    \begin{proposition}
        \label{prop6-4}
        Let $(Y,a,W)\in\mathscr U^s$ be a coupled equilibrium with $A$ invertible and $\beta_{\rm g}\ne0$. Set
        \[
        v_{\rm s}:=\widetilde{\mathcal H}_{\rm s}^{-1}\widetilde k,\qquad \mathfrak s:=\widetilde d-\langle\widetilde k,v_{\rm s}\rangle=-\frac{\Gamma L}{\alpha^2\beta_{\rm g}}.
        \]
        Then, on $\mathscr V=H_{\rm e}^1(\mathbb R)\times\mathbb R$,
        \begin{equation}
            \mathfrak q[(h,\nu)]=\mathfrak q_{\rm s}[h+v_{\rm s}\nu]+\mathfrak s\,\nu^2.
            \label{6-29}
        \end{equation}
        The conformal-coordinate Hessian extends to the same form domain and has the same negative index and nullity as $\mathfrak q$. In particular,
        \begin{equation}
            n_-(D^2\mathcal J_W)=n_-(\mathfrak q)=n_-(\mathfrak q_{\rm s})+\mathbf 1_{\{L\beta_{\rm g}>0\}}.
            \label{6-30}
        \end{equation}
        The kernel of $\mathfrak q$ is trivial for $L\ne0$ and is $\operatorname{span}\{(-v_{\rm s},1)\}$ for $L=0$.

        Along the simple-fold curve $\gamma_\star(a)$ of Proposition~\ref{prop5-2}, denote the full and surface forms by $\mathfrak q(a)$ and $\mathfrak q_{\rm s}(a)$, with $W=w(a)$ fixed in each Hessian. On a sufficiently small segment, $n_-(\mathfrak q_{\rm s}(a))$ is constant and, for every sufficiently small $\delta>0$,
        \begin{equation}
            n_-(\mathfrak q(a_\star+\delta))-n_-(\mathfrak q(a_\star-\delta))=\operatorname{sgn}\bigl(\beta_{{\rm g},\star}(\partial_aL)_\star\bigr).
            \label{6-31}
        \end{equation}
        At the fold, $n_-(\mathfrak q(a_\star))=n_-(\mathfrak q_{\rm s}(a_\star))$ and $\dim\ker\mathfrak q(a_\star)=1$.
    \end{proposition}

    \begin{proof}
        By Lemma~\ref{lem6-3}, $v_{\rm s}\in H_{\rm e}^1$ satisfies $\mathfrak q_{\rm s}[v_{\rm s},h]=\langle\widetilde k,h\rangle$ for every $h\in H_{\rm e}^1$. Its value agrees with the Sobolev inverse of Lemma~\ref{lem6-2} by uniqueness. Hence
        \[
        \mathfrak q_{\rm s}[h+v_{\rm s}\nu]=\mathfrak q_{\rm s}[h]+2\nu\langle\widetilde k,h\rangle+\nu^2\langle\widetilde k,v_{\rm s}\rangle,
        \]
        and comparison with the full block form and \eqref{6-20} proves \eqref{6-29}.

        The transformation $(h,\nu)\mapsto(h+v_{\rm s}\nu,\nu)$ is a bounded isomorphism of $\mathscr V$. Thus \eqref{6-29} is a congruence with the direct sum of $\mathfrak q_{\rm s}$ and $\mathfrak s\nu^2$, preserving negative indices and nullities. Since $\mathfrak q_{\rm s}$ is nondegenerate,
        \[
        n_-(\mathfrak q)=n_-(\mathfrak q_{\rm s})+\mathbf 1_{\{\mathfrak s<0\}},\qquad \mathfrak s<0\ \Longleftrightarrow\ L\beta_{\rm g}>0,
        \]
        and the kernel is one-dimensional exactly when $\mathfrak s=0$, with generator $(-v_{\rm s},1)$.

        For conformal coordinates, define $\mathscr T(h,\ell):=(h,E_ah-\alpha\ell)$. The Poisson evaluation is bounded on $L^2$, so $\mathscr T$ is a bounded isomorphism on both $\mathscr H$ and $\mathscr V$, with inverse $\ell=(E_ah-\nu)/\alpha$. At equilibrium, criticality eliminates the second derivative of the coordinate change, giving
        \[
        D^2\mathcal J_W[(h,\ell),(h,\ell)]=\mathfrak q[(h,E_ah-\alpha\ell)].
        \]
        By density, this identity extends to $\mathscr V$, and the congruence preserves negative indices and nullities. The negative index on the original higher-regularity space $\mathcal X^s$ agrees with that on $\mathscr V$: any finite-dimensional negative subspace can be approximated within $\mathcal X^s$, preserving strict negativity of its Gram matrix.

        To identify the nullities, we verify that every form-kernel vector has the original regularity. If $(h,\nu)\in\ker\mathfrak q$, its surface equation gives $\widetilde{\mathcal H}_{\rm s}h+\widetilde k\nu=0$ in $H_{\rm e}^{-1}$. Using \eqref{6-16} and $\widetilde k=-\alpha^{-1}\mathcal C_Y^*B_a$, we obtain
        \[
        (A+B_a\mathcal E)h=\frac{\nu}{\alpha}B_a.
        \]
        Since $B_a\in H_{\rm e}^{s+1}$, the elliptic regularity argument of Lemma~\ref{lem6-3} applies and yields $h\in H_{\rm e}^{s+2}$. Thus no additional kernel vectors arise on the larger form domain, proving the coordinate assertion and \eqref{6-30}.

        Along $\gamma_\star$, restrict the neighborhood so that $A$ remains invertible and $\beta_{\rm g}\ne0$. The surface negative index is constant by Lemma~\ref{lem6-3}. Formula~\eqref{6-21} gives a simple zero of $\mathfrak s$ at $a_\star$, with
        \[
        \operatorname{sgn}\mathfrak s'(a_\star)=-\operatorname{sgn}\bigl(\beta_{{\rm g},\star}(\partial_aL)_\star\bigr).
        \]
        As $a$ increases through the fold, the scalar negative index changes by $-\operatorname{sgn}\mathfrak s'(a_\star)$, proving \eqref{6-31}. At the fold, the scalar form contributes one kernel direction and no negative direction.
    \end{proof}

    The jump is oriented by increasing conformal depth rather than speed; the two nearby waves at the same speed therefore have negative indices differing by one. The congruence preserves indices and nullities, though not individual eigenvalues.

    The same nondegenerate crossing has a direct expression in the original conformal kernel direction. Along the fold curve, set
    \[
    \mathcal H(a):=D^2\mathcal J_{w(a)}(Y(a,w(a)),a),\qquad \zeta(a):=(Y_a(a,w(a)),1),
    \]
    where $Y_a$ is taken at fixed $W$. From \eqref{6-15} and $AY_a+B_a=0$, we have $\mathcal H(a)[\zeta(a),\zeta(a)]=-\Gamma\beta_{\rm g}(a,w(a))L(a,w(a))$. At the fold, $\zeta_\star\in\ker\mathcal H(a_\star)$, so differentiation, using $L_\star=w'(a_\star)=0$, yields
    \begin{equation}
        \left.\frac{d}{da}\mathcal H(a)[\zeta_\star,\zeta_\star]\right|_{a=a_\star}=-\Gamma\beta_{{\rm g},\star}(\partial_aL)_\star\ne0.
        \label{6-32}
    \end{equation}
    The conclusions concern the static energy--momentum Hessian in the symmetric class, not dynamical stability; they do not include the internal degrees of freedom of a smooth vortex core. Part~III \cite{DuZouPartIII} develops the joint dual variational construction and effective Hessian analysis for finite-core waves, while Part~V \cite{DuZouPartV} studies fixed-impulse coercivity, energy exchange, and conditional logarithmic-time orbital tracking. We turn next to smooth-core regularization of point-vortex solution arcs and their simple folds.

    \begin{proof}[Proof of Theorem~\ref{thm1-2}]
        Proposition~\ref{prop5-2} proves part~{\rm(i)}, including the quadratic turning, fixed-speed solution count, and linearization properties. For part~{\rm(ii)}, Lemma~\ref{lem6-2} gives $\beta_{{\rm g},\star}\ne0$. On a sufficiently small segment, $A$ and the physical-coordinate surface Hessian remain invertible. Lemma~\ref{lem6-3} provides the common closed form domains, finite negative indices, and constancy of the surface index. Proposition~\ref{prop6-4} then proves the jump \eqref{1-14}, the kernel assertions, and the equivalence of the conformal- and physical-coordinate Hessian indices and nullities.
    \end{proof}

    \bigskip

    \section{Regularization along solution curves}
    \label{sec7}
    We regularize compact regular arcs of the point-vortex solution curve, using conformal depth $a$ and wave speed $W$ as independent parameters. At a simple fold, $L=\partial_a\mathscr G_0=0$ makes the fixed-speed coupled linearization singular, but the Bernoulli surface operator and the mass-constrained inner operator on the translation complement remain invertible. Using the local construction of Part~I \cite{DuZouPartI}, we solve the complementary equations jointly and leave the vertical translation condition as a scalar drift equation. Two-parameter $C^2$ convergence of this drift then gives persistence of the regularized arcs and their specified simple folds.

    Fix $s>\frac12$, and let $\mathscr K\subset\mathscr C_0$ be a compact regular subarc on which $A$ is invertible and $L^2+M^2>0$. Compactness gives
    \[
        \sup_{\mathscr K}\|A^{-1}\|_{\mathcal L(H_{\rm e}^s,H_{\rm e}^{s+2})}<\infty,\qquad \inf_{\mathscr K}(L^2+M^2)>0.
    \]
    The Bernoulli equation can therefore be solved uniformly, including at folds where $L=0$ and $M\ne0$. The admissible core size may depend on $\mathscr K$; we make no assertion about a global smooth-vorticity continuum at fixed $\varepsilon>0$.

    Assume $f\in C_{\rm loc}^{2,\gamma}(\mathbb R)$, $\gamma\in(0,1)$, and use the radial profile $U$ from \cite[Assumption~5.1]{DuZouPartI}, satisfying
    \[
    -\Delta U=f(U),\qquad \int_{\mathbb R^2}f(U)\,dX=\Gamma,\qquad U(X)=-\frac{\Gamma}{2\pi}\log|X|\quad (|X|\ge1).
    \]
    In particular, $U$ is nondegenerate modulo translations, $\operatorname{supp}f(U),\operatorname{supp}f'(U)\subset\overline{B_1}$, and $\int_{\mathbb R^2}f'(U)\,dX\ne0$. Fix $\nu\in(0,1)$ and $0<\gamma'<\gamma$. Inner parameter derivatives are always taken at fixed $X$, and no differentiability in $\varepsilon$ at zero is required.

    \smallskip

    \subsection{A uniform complementary inverse}
    Lemma~\ref{lem2-8} supplies a finite covering of $\mathscr K$ by Bernoulli charts. In each chart, let $Y=Y_0(a,W)$ on an open parameter set $\mathcal U_{\rm par}\subset(0,\infty)\times(0,c_{\min})$, with $B(Y_0(a,W),a,W)=0$ and $A_0(a,W):=D_YB(Y_0(a,W),a,W)$ invertible. Choose $K\Subset\mathcal U_{\rm par}$ whose interior contains the corresponding arc segment, and work on a slightly larger parameter neighborhood. Shrinking the configuration neighborhoods ensures uniform surface bounds, positivity of $1+|D|Y$ and $\alpha(Y,a):=Z_Y'(-\mathrm ia)$, and separation of $p_{Y,a}:=Z_Y(-\mathrm ia)$ from the free surface. The ensuing estimates hold uniformly on these neighborhoods.

    Use the weighted H\"older space $\mathcal X_*:=C_{\nu,\mathrm e}^{2,\gamma'}(\mathbb R^2)$ of Part~I, with evenness in $X_1$ excluding horizontal translations. Set $Z_0:=\partial_2U$, and choose $\eta_0\in C_c^\infty(\mathbb R^2)$, even in $X_1$ and odd in $X_2$, such that $\int_{\mathbb R^2}\eta_0Z_0\,dX=1$. Then $\int_{\mathbb R^2}\eta_0\,dX=0$. Fix $R_*>2$ with $\operatorname{supp}\eta_0\Subset B_{R_*}$ and define
    \[
        \begin{aligned}
            \mathcal X_{R_*}&:=\{\phi\in\mathcal X_*:\operatorname{supp}(-\Delta\phi)\subset\overline{B_{R_*}}\},\\
            \mathcal Y_{R_*}^{0}&:=\left\{r\in C_{\mathrm e}^{0,\gamma'}(\mathbb R^2):\operatorname{supp}r\subset\overline{B_{R_*}},\ \int_{\mathbb R^2}r\,dX=0\right\}.
        \end{aligned}
    \]
    Both are closed Banach spaces with the inherited norms. For $\phi\in\mathcal X_{R_*}$, weighted decay gives $\int_{\mathbb R^2}(-\Delta\phi)\,dX=0$, so $-\Delta:\mathcal X_{R_*}\to\mathcal Y_{R_*}^{0}$ is well defined.

    Define the projections
    \[
        \Pi_0\phi:=\left(\int_{\mathbb R^2}\eta_0\phi\,dX\right)Z_0,\qquad \ell_0(r):=\int_{\mathbb R^2}rZ_0\,dX,\qquad Q_0r:=r-\ell_0(r)\eta_0.
    \]
    Since $-\Delta Z_0=f'(U)Z_0$ is supported in $\overline{B_1}$ and $\eta_0$ has fixed support and zero integral, these projections preserve the corresponding fixed-source spaces. Put $\mathcal X_{R_*}^{\perp}:=\ker\Pi_0\cap\mathcal X_{R_*}$ and $\mathcal Y_{R_*}^{\perp}:=\ker\ell_0\cap\mathcal Y_{R_*}^{0}$, where $\perp$ denotes the chosen complements, and set
    \[
        \mathbb X_{R_*}^s:=H_{\rm e}^{s+2}(\mathbb R)\times\mathcal X_{R_*}^{\perp},\qquad \mathbb Y_{R_*}^s:=H_{\rm e}^{s}(\mathbb R)\times\mathcal Y_{R_*}^{\perp}.
    \]
    The condition $\Pi_0\phi=0$ normalizes the whole-plane correction rather than its Dirichlet reconstruction.

    For $\phi\in\mathcal X_{R_*}$, write $g_\phi:=-\Delta\phi$. The zero-mass Newton representation is
    \begin{equation}
        \phi(X)=\frac1{2\pi}\int_{B_{R_*}}\log\frac1{|X-S|}\,g_\phi(S)\,dS.
        \label{7-5}
    \end{equation}
    The difference is an entire harmonic function that vanishes at infinity. Cancellation of the constant logarithmic term and exterior harmonic estimates give, for $|X|\ge2R_*$ and $k\ge0$,
    \[
    |D_X^k\phi(X)|\le C_{k,R_*}|X|^{-1-k}\|g_\phi\|_{C^{0,\gamma'}}\le C_{k,R_*}|X|^{-1-k}\|\phi\|_{\mathcal X_{R_*}}.
    \]

    Set $z_a:=-\mathrm ia$, $\delta:=\varepsilon/\alpha(Y,a)$, and $H:=H_{\varepsilon,Y,a}:=a/\delta$. The inner coordinates $z=z_a+\delta X$ transform $\Pi_-$ into $\mathcal H_H:=\{X_2<H\}$. Uniformly, $H\asymp\varepsilon^{-1}$, so $H>2R_*$ for small $\varepsilon$. Define
    \[
        \mathfrak P_{\varepsilon,Y,a}\phi(X):=\mathfrak P_H\phi(X):=\phi(X)-\phi(X_1,2H-X_2),\qquad X\in\mathcal H_H.
    \]
    The reflected term is harmonic in $\mathcal H_H$, and hence $-\Delta_X\mathfrak P_H\phi=g_\phi$ with zero Dirichlet data on $X_2=H$. Equivalently, with $G_{\Pi_-}(z,\zeta)=(2\pi)^{-1}\log(|z-\overline\zeta|/|z-\zeta|)$, define $\mathcal C_{\varepsilon,Y,a}[\phi](z):=\int_{B_{R_*}}G_{\Pi_-}(z,z_a+\delta S)g_\phi(S)\,dS$. Writing $S^H:=(S_1,2H-S_2)$, we have the exact identity
    \[
        \mathcal C_{\varepsilon,Y,a}[\phi](z_a+\delta X)=\frac1{2\pi}\int_{B_{R_*}}\log\frac{|X-S^H|}{|X-S|}\,g_\phi(S)\,dS=\mathfrak P_H\phi(X).
    \]
    The correction therefore satisfies the Dirichlet condition exactly and decays at infinity.

    For $\mathfrak b_H:=\mathfrak P_H-I$, the exterior estimates following \eqref{7-5} yield, on every fixed inner disk,
    \begin{equation}
        \|\partial_H^j\mathfrak b_H[\phi]\|_{C^{2,\gamma'}(\overline{B_R})}\le C_{R,R_*}H^{-1-j}\|\phi\|_{\mathcal X_{R_*}},\qquad j=0,1,2.
        \label{7-8}
    \end{equation}
    If $\theta=(a,W)$ and $Y=Y(\theta)$ is a $C^2$ family with uniformly bounded $C^2(K;H^{s+2})$ norm, then $\partial_\theta^\beta H=O(H)$ for $1\le|\beta|\le2$. The chain rule therefore gives
    \[
    \sum_{|\beta|\le2}\sup_{\theta\in K}\|\partial_\theta^\beta\mathfrak b_{H_{\varepsilon,Y(\theta),a}}[\phi]\|_{C^{2,\gamma'}(\overline{B_R})}\le C_{K,R}\varepsilon\|\phi\|_{\mathcal X_{R_*}}.
    \]
    These estimates keep $\phi$ fixed and concern fixed inner disks, not the whole-plane weighted norm.

    Use the approximate radial core constructed in Part~I. Set $P_{\varepsilon,Y,a}(z):=\int_{B_1}G_{\Pi_-}(z,z_a+\delta S)f(U(S))\,dS$. Radiality and the mean-value property give
    \begin{equation}
        P_{\varepsilon,Y,a}(z)=\Gamma G_{\Pi_-}(z,z_a),\qquad z\in\Pi_-,\quad |z-z_a|>\delta.
        \label{7-9}
    \end{equation}
    Outside the core, this field agrees exactly with the point-vortex Green field. On the inner scale,
    \[
        P_{\varepsilon,Y,a}(z_a+\delta X)=U(X)+\frac{\Gamma}{2\pi}\log|X-2He_2|,\qquad X\in\mathcal H_H,
    \]
    where $e_2=(0,1)$. Choose $\mu_{\varepsilon,Y,a,W}^{\rm app}:=\Gamma(2\pi)^{-1}\log(2H)-Wa$ and define
    \begin{equation}
        r_\varepsilon(Y,a,W;X):=\frac{\Gamma}{2\pi}\log\frac{|X-2He_2|}{2H}+W\delta X_2.
        \label{7-11}
    \end{equation}
    Then $P_{\varepsilon,Y,a}(z_a+\delta X)+W(-a+\delta X_2)-\mu_{\varepsilon,Y,a,W}^{\rm app}=U(X)+r_\varepsilon(Y,a,W;X)$. Since
    \[
    \log\frac{|X-2He_2|}{2H}=\frac12\log\left(1-\frac{\delta X_2}{a}+\frac{\delta^2|X|^2}{4a^2}\right),
    \]
    Taylor expansion on fixed disks gives, for the parameter families above,
    \begin{equation}
        \left\|r_\varepsilon-\frac{\varepsilon}{\alpha}\left(W-\frac{\Gamma}{4\pi a}\right)X_2\right\|_{C^2(K;C^{2,\gamma'}(\overline{B_R}))}\le C_{K,R}\varepsilon^2.
        \label{7-12}
    \end{equation}
    The logarithmic argument stays uniformly positive, and two parameter derivatives preserve the stated remainder bound. The same bounds hold for independent surface and depth variations. In particular, $r_\varepsilon=O(\varepsilon)$ on fixed inner disks.

    For small $\phi\in\mathcal X_{R_*}$ and $m\in\mathbb R$, define
    \[
        V_{\varepsilon,Y,a,W,\phi,m}:=U+r_\varepsilon(Y,a,W;\,\cdot)+\mathfrak P_H\phi-m.
    \]
    Writing $\mu:=\mu_{\varepsilon,Y,a,W}^{\rm app}+m$ and $z=z_a+\delta X$, this is $V_\varepsilon(X)=P_{\varepsilon,Y,a}(z)+\mathcal C_{\varepsilon,Y,a}[\phi](z)+W\operatorname{Im}z-\mu$.

    We claim that there exist $\rho_*,\varepsilon_*,c_*>0$, uniform in the chosen neighborhoods, such that
    \begin{equation}
        V_{\varepsilon,Y,a,W,\phi,m}(X)\le-c_*\qquad (X\in\mathcal H_H\setminus B_{R_*})
        \label{7-14}
    \end{equation}
    whenever $0<\varepsilon<\varepsilon_*$ and $\|\phi\|_{\mathcal X_{R_*}}+|m|\le\rho_*$. Indeed, $V_\varepsilon$ is harmonic outside $\overline{B_{R_*}}$, since the nonlinear sources are supported there and the reflected singularities lie outside $\mathcal H_H$. Set $\kappa_*:=\Gamma(2\pi)^{-1}\log R_*>0$. On $\partial B_{R_*}$, we have $U=-\kappa_*$, while \eqref{7-8}, \eqref{7-12}, and the smallness of $\|\phi\|_{\mathcal X_{R_*}}+|m|$ give $V_\varepsilon\le-\kappa_*/2$. On $X_2=H$, both Dirichlet fields vanish and $V_\varepsilon=-\mu$, where $\mu=\Gamma(2\pi)^{-1}\log(2H)-Wa+m\to+\infty$ uniformly. At spatial infinity, the Dirichlet fields decay and $W\operatorname{Im}z\le0$, so $\limsup_{|X|\to\infty}V_\varepsilon(X)\le-\mu$. The maximum principle on truncated exterior domains therefore proves \eqref{7-14} with a uniform $c_*>0$. Moreover, since $R_*>2$ and $U$ is strictly negative near $\partial B_{R_*}$, the same local perturbation estimates give a fixed annulus around this circle on which $V_\varepsilon<0$. As $f$ vanishes on $(-\infty,0]$, the nonlinear source is supported strictly inside $B_{R_*}$ and can be extended by zero without a cutoff.

    We now choose $m$ to impose the circulation constraint. Define $j_{\varepsilon,Y,a}(X):=|Z_Y'(z_a+\delta X)|^2/\alpha(Y,a)^2$, so the physical area element is $\varepsilon^2j_{\varepsilon,Y,a}\,dX$. The exact mass equation is
    \begin{equation}
        \mathcal M_\varepsilon(Y,a,W,\phi,m):=\int_{B_{R_*}}j_{\varepsilon,Y,a}(X)f(V_\varepsilon(X))\,dX-\Gamma=0.
        \label{7-15}
    \end{equation}
    Its limiting counterpart is
    \[
        \mathcal M_0(\phi,m):=\int_{B_{R_*}}f(U+\phi-m)\,dX-\Gamma.
    \]
    Set $I_0:=\int_{\mathbb R^2}f'(U)\,dX\ne0$. Since $\mathcal M_0(0,0)=0$ and $\partial_m\mathcal M_0(0,0)=-I_0$, the scalar implicit function theorem gives a local $C^2$ map $m=\mathfrak m(\phi)$ with $\mathfrak m(0)=0$ and
    \[
    D\mathfrak m(0)[h]=m_0[h]:=I_0^{-1}\int_{\mathbb R^2}f'(U)h\,dX.
    \]

    For small $\varepsilon$, the same argument yields a unique $C^2$ level $m=m_\varepsilon(Y,a,W,\phi)$ on a fixed trial neighborhood. The reconstruction estimates and $f\in C_{\rm loc}^{2,\gamma}$, with $\gamma'<\gamma$, give the required $C^2$ regularity in the spatial $C^{0,\gamma'}$ target. Uniform convergence of the mass functional and its level derivative permits $|\partial_m\mathcal M_\varepsilon|\ge|I_0|/2$, yielding uniform first and second derivative bounds. Since $j_{\varepsilon,Y,a}-1$, $r_\varepsilon$, and $\mathfrak P_H-I$ are $O(\varepsilon)$ on the source disk, comparison with $\mathcal M_0$ gives
    \begin{equation}
        |m_\varepsilon-\mathfrak m(\phi)|\le C\varepsilon,\qquad |m_\varepsilon|\le C\bigl(\varepsilon+\|\phi\|_{\mathcal X_{R_*}}\bigr).
        \label{7-17}
    \end{equation}
    Choose the trial neighborhood so that \eqref{7-14} continues to hold. The estimate $m_\varepsilon=O(\varepsilon)$ will apply along the complementary solutions, not to arbitrary trial corrections.

    Henceforth set $V_\varepsilon:=U+r_\varepsilon+\mathfrak P_H\phi-m_\varepsilon(Y,a,W,\phi)$ and $j_\varepsilon:=j_{\varepsilon,Y,a}$. By \eqref{7-14}, $j_\varepsilon f(V_\varepsilon)$ vanishes in a fixed annulus around $\partial B_{R_*}$ and extends by zero to $C_{\mathrm e}^{0,\gamma'}(\mathbb R^2)$. Define the mass-constrained inner residual
    \begin{equation}
        \mathcal E_\varepsilon(Y,a,\phi;W):=-\Delta_X\phi+f(U)-j_\varepsilon f(V_\varepsilon).
        \label{7-18}
    \end{equation}
    With $\Phi_\varepsilon(z):=P_{\varepsilon,Y,a}(z)+\mathcal C_{\varepsilon,Y,a}[\phi](z)+W\operatorname{Im}z$, the reconstruction gives
    \[
    \mathcal E_\varepsilon(X)=\delta^2\left[-\Delta_z\Phi_\varepsilon-\frac{|Z_Y'|^2}{\varepsilon^2}f(\Phi_\varepsilon-\mu_\varepsilon)\right](z_a+\delta X),\qquad X\in\mathcal H_H.
    \]
    This holds on the source disk by $\delta^2|Z_Y'|^2/\varepsilon^2=j_\varepsilon$, and both sides vanish outside it. Moreover, \eqref{7-15} and the zero mass of $-\Delta\phi$ give $\int_{\mathbb R^2}\mathcal E_\varepsilon\,dX=0$. Thus $\mathcal E_\varepsilon\in\mathcal Y_{R_*}^{0}$ and $Q_0\mathcal E_\varepsilon\in\mathcal Y_{R_*}^{\perp}$.

    For the surface response, write $\mathcal N_{\varepsilon,Y,a}\phi:=\partial_y\mathcal C_{\varepsilon,Y,a}[\phi](\cdot,0)$ and $k_d(x):=d/[\pi(x^2+d^2)]$. Differentiating the half-plane Green function gives
    \[
        \mathcal N_{\varepsilon,Y,a}\phi(x)=-\int_{B_{R_*}}k_{a-\delta S_2}(x-\delta S_1)g_\phi(S)\,dS.
    \]
    Since $\int g_\phi=0$ and $\widehat{k_d}(\xi)=e^{-d|\xi|}$,
    \[
    \widehat{\mathcal N_{\varepsilon,Y,a}\phi}(\xi)=-e^{-a|\xi|}\int_{B_{R_*}}\left[e^{\delta(S_2|\xi|-\mathrm iS_1\xi)}-1\right]g_\phi(S)\,dS.
    \]
    If $a_->0$ is the uniform depth lower bound and $\delta R_*\le a_-/2$, the exponential difference is bounded by $C\varepsilon|\xi|e^{\delta R_*|\xi|}$. Consequently,
    \[
    |\widehat{\mathcal N_{\varepsilon,Y,a}\phi}(\xi)|\le C\varepsilon|\xi|e^{-a_-|\xi|/2}\|g_\phi\|_{L^1},
    \]
    and, for every fixed $r\ge0$,
    \begin{equation}
        \|\mathcal N_{\varepsilon,Y,a}\phi\|_{H^r(\mathbb R)}\le C_{r,K}\varepsilon\|\phi\|_{\mathcal X_{R_*}}.
        \label{7-20}
    \end{equation}
    The zero-mass cancellation produces the factor $\varepsilon$; separation from the surface gives any fixed boundary Sobolev regularity. The Fourier formula also controls two parameter derivatives, as used below.

    Both Dirichlet fields have vanishing tangential derivative on $y=0$. By \eqref{7-9}, the normal derivative of the reconstructed stream function is $q_{a,W}+\mathcal N_{\varepsilon,Y,a}\phi$. Writing $\mathcal N_\varepsilon:=\mathcal N_{\varepsilon,Y,a}$, the Bernoulli residual becomes
    \begin{equation}
        \mathcal B_\varepsilon(Y,a,\phi;W):=B(Y,a,W)+\frac{q_{a,W}}{J(Y)}\mathcal N_\varepsilon\phi+\frac{(\mathcal N_\varepsilon\phi)^2}{2J(Y)}.
        \label{7-21}
    \end{equation}
    In particular, $\mathcal B_\varepsilon(Y,a,0;W)=B(Y,a,W)$ exactly. Combine the surface equation and the projected inner residual as
    \[
        \mathscr R_\varepsilon(Y,\phi;a,W):=\begin{pmatrix}\mathcal B_\varepsilon(Y,a,\phi;W)\\[1mm]Q_0\mathcal E_\varepsilon(Y,a,\phi;W)\end{pmatrix},\qquad \phi\in\mathcal X_{R_*}^{\perp}.
    \]
    By \eqref{7-20}, the Sobolev product estimates, and the fixed-support and circulation properties, $\mathscr R_\varepsilon$ takes values in $\mathbb Y_{R_*}^s$. For $\varepsilon>0$, it is $C^2$ in $(Y,a,W,\phi)$: the geometric operators are smooth, and the nonlinear derivatives involve only $f'$ and $f''$. The condition $\gamma'<\gamma$ ensures continuity of the second derivatives in the spatial $C^{0,\gamma'}$ target by H\"older interpolation, without requiring $f'''$.

    At $\varepsilon=0$, the limiting residual is
    \[
        \begin{aligned}
            \mathcal E_0(\phi)&:=-\Delta\phi+f(U)-f\bigl(U+\phi-\mathfrak m(\phi)\bigr),\\
            \mathscr R_0(Y,\phi;a,W)&:=\bigl(B(Y,a,W),Q_0\mathcal E_0(\phi)\bigr).
        \end{aligned}
    \]
    Since $D\mathfrak m(0)=m_0$, its inner linearization is $D_\phi\mathcal E_0(0)=\mathscr L_0$, where $\mathscr L_0\varphi=-\Delta\varphi-f'(U)(\varphi-m_0[\varphi])$ is the mass-constrained operator of Part~I.

    The projected inverse from Part~I \cite[Eq.~(5.18)]{DuZouPartI} extends to the present exponent $\gamma'<\gamma$: the weighted Newton inverse, Fredholm argument, and translation-kernel characterization remain valid because $f'(U)\in C_c^{0,\gamma}$. We use this inverse on the fixed-source spaces. Set $\mathscr L_\perp:=Q_0\mathscr L_0|_{\mathcal X_{R_*}^{\perp}}$. For $r\in\mathcal Y_{R_*}^{\perp}$, the Part~I inverse gives a unique $\phi\in\ker\Pi_0\cap\mathcal X_*$ with $Q_0\mathscr L_0\phi=r$. Since $-\Delta Z_0=f'(U)Z_0$ and $\int f'(U)Z_0\,dX=0$, integration by parts gives $\ell_0(\mathscr L_0\phi)=0$. Consequently, the projected equation implies $\mathscr L_0\phi=r$, or equivalently,
    \[
    -\Delta\phi=r+f'(U)\bigl(\phi-m_0[\phi]\bigr).
    \]
    The right-hand side is supported in $\overline{B_{R_*}}$, so $\phi\in\mathcal X_{R_*}^{\perp}$. Equivalence of the weighted and unweighted target norms on fixed support yields $\|\phi\|_{\mathcal X_{R_*}}\le C_{R_*}\|r\|_{C^{0,\gamma'}}$. Thus $\mathscr L_\perp:\mathcal X_{R_*}^{\perp}\to\mathcal Y_{R_*}^{\perp}$ is an isomorphism.

    Linearize the complementary map at the reference Bernoulli surface:
    \[
        \mathbb L_\varepsilon(a,W):=D_{(Y,\phi)}\mathscr R_\varepsilon(Y_0(a,W),0;a,W):\mathbb X_{R_*}^s\longrightarrow\mathbb Y_{R_*}^s.
    \]
    At $\varepsilon=0$, the surface and inner equations decouple into the invertible blocks $A_0(a,W)$ and $\mathscr L_\perp$. The next lemma shows that their invertibility persists uniformly under regularization, even where $L=0$.

    \begin{lemma}
        \label{lem7-1}
        For every compact parameter set $K$ chosen above, there exist $\varepsilon_K,C_K>0$ such that, with $\mathbb L_0(a,W):=\operatorname{diag}(A_0(a,W),\mathscr L_\perp)$,
        \begin{equation}
            \sup_{(a,W)\in K}\|\mathbb L_\varepsilon(a,W)-\mathbb L_0(a,W)\|_{\mathcal L(\mathbb X_{R_*}^s,\mathbb Y_{R_*}^s)}\le C_K\varepsilon
            \label{7-25}
        \end{equation}
        for $0<\varepsilon<\varepsilon_K$. After decreasing $\varepsilon_K$ if necessary, each $\mathbb L_\varepsilon(a,W)$ is an isomorphism and
        \begin{equation}
            \sup_{(a,W)\in K}\|\mathbb L_\varepsilon(a,W)^{-1}\|_{\mathcal L(\mathbb Y_{R_*}^s,\mathbb X_{R_*}^s)}\le C_K.
            \label{7-26}
        \end{equation}
    \end{lemma}

    \begin{proof}
        We estimate the four blocks of $\mathbb L_\varepsilon-\mathbb L_0$, differentiating in $(Y,\phi)$ at fixed $(a,W)$. The constants are uniform on $K$.

        Since $\mathcal N_{\varepsilon,Y,a}\phi=0$ at $\phi=0$, \eqref{7-21} gives
        \[
            D_Y\mathcal B_\varepsilon=A_0(a,W),\qquad D_\phi\mathcal B_\varepsilon[\varphi]=\frac{q_{a,W}}{J(Y_0)}\mathcal N_{\varepsilon,Y_0,a}\varphi.
        \]
        Thus the upper-left block agrees exactly with $A_0$, while \eqref{7-20} and the uniform multiplier bounds make the upper-right block $O(\varepsilon)$ from $\mathcal X_{R_*}^{\perp}$ to $H_{\rm e}^s$.

        For the inner blocks, put $\overline m_\varepsilon(Y,a,W):=m_\varepsilon(Y,a,W,0)$ and define $V_\varepsilon^0:=U+r_\varepsilon-\overline m_\varepsilon$, $c_\varepsilon:=j_\varepsilon f'(V_\varepsilon^0)$, and $I_\varepsilon:=\int_{B_{R_*}}c_\varepsilon\,dX$, with $c_\varepsilon$ extended by zero. Equations~\eqref{7-12} and \eqref{7-17} give $V_\varepsilon^0-U=O(\varepsilon)$ and $j_\varepsilon-1=O(\varepsilon)$ in $C^{0,\gamma'}$ on the source disk. The identity
        \[
        f'(V_\varepsilon^0)-f'(U)=(V_\varepsilon^0-U)\int_0^1f''\bigl(U+t(V_\varepsilon^0-U)\bigr)\,dt
        \]
        and the uniform H\"older bounds for the integral coefficient imply
        \[
        \|c_\varepsilon-f'(U)\|_{C^{0,\gamma'}(\mathbb R^2)}+|I_\varepsilon-I_0|\le C_K\varepsilon.
        \]
        In particular, $|I_\varepsilon|\ge|I_0|/2$.

        Differentiating \eqref{7-15} with respect to $\phi$ gives
        \[
            \lambda_\varepsilon[\varphi]:=D_\phi m_\varepsilon[\varphi]=\frac{\displaystyle\int_{B_{R_*}}c_\varepsilon\mathfrak P_H\varphi\,dX}{I_\varepsilon}.
        \]
        By the coefficient estimate and \eqref{7-8}, $\|\lambda_\varepsilon-m_0\|_{\mathcal L(\mathcal X_{R_*},\mathbb R)}\le C_K\varepsilon$. Consequently,
        \[
            D_\phi\mathcal E_\varepsilon[\varphi]=-\Delta\varphi-c_\varepsilon\bigl(\mathfrak P_H\varphi-\lambda_\varepsilon[\varphi]\bigr).
        \]
        Subtracting $\mathscr L_0\varphi$ leaves
        \[
        -(c_\varepsilon-f'(U))\bigl(\varphi-m_0[\varphi]\bigr)-c_\varepsilon(\mathfrak P_H-I)\varphi+c_\varepsilon(\lambda_\varepsilon-m_0)[\varphi].
        \]
        Each term is $O(\varepsilon)\|\varphi\|_{\mathcal X_{R_*}}$ in $C^{0,\gamma'}$. The differentiated circulation constraint gives zero integral, and the support remains in $\overline{B_{R_*}}$. Applying $Q_0$ gives the required $O(\varepsilon)$ bound for the lower-right block.

        For the lower-left block, we first establish
        \begin{equation}
            \|D_Yj_\varepsilon[h]\|_{C^{0,\gamma'}(\overline{B_{R_*}})}+\|D_Yr_\varepsilon[h]\|_{C^{0,\gamma'}(\overline{B_{R_*}})}\le C_K\varepsilon\|h\|_{H^{s+2}}.
            \label{7-30}
        \end{equation}
        In \eqref{7-11}, every surface derivative acts through $\delta=\varepsilon/\alpha(Y,a)$ or $H^{-1}=\delta/a$, retaining a factor $\varepsilon$. For $j_\varepsilon$, use the exact identity
        \[
        j_\varepsilon(X)-1=\frac2{\alpha^2}\operatorname{Re}\int_0^\delta\overline{Z_Y'(z_a+tX)}Z_Y''(z_a+tX)X\,dt.
        \]
        The integration interval has length $O(\varepsilon)$ and $D_Y\delta[h]=O(\varepsilon)\|h\|_{H^{s+2}}$. Uniform interior conformal estimates, valid because the evaluation points stay separated from the boundary, yield \eqref{7-30}.

        At $\phi=0$, differentiation of the circulation constraint in $Y$ gives
        \begin{equation}
            D_Y\overline m_\varepsilon[h]=\frac{\displaystyle\int_{B_{R_*}}\bigl(D_Yj_\varepsilon[h]f(V_\varepsilon^0)+c_\varepsilon D_Yr_\varepsilon[h]\bigr)\,dX}{I_\varepsilon}.
            \label{7-31}
        \end{equation}
        Thus $|D_Y\overline m_\varepsilon[h]|\le C_K\varepsilon\|h\|_{H^{s+2}}$. Since $D_Y(\mathfrak P_H\phi)=0$ at $\phi=0$,
        \[
        D_Y\mathcal E_\varepsilon[h]=-D_Yj_\varepsilon[h]f(V_\varepsilon^0)-c_\varepsilon\bigl(D_Yr_\varepsilon[h]-D_Y\overline m_\varepsilon[h]\bigr).
        \]
        This residual has fixed support and zero integral by the differentiated mass constraint. Equations~\eqref{7-30}--\eqref{7-31} therefore give $\|D_Y\mathcal E_\varepsilon[h]\|_{\mathcal Y_{R_*}^{0}}\le C_K\varepsilon\|h\|_{H^{s+2}}$. After applying $Q_0$, this gives the lower-left block estimate and completes \eqref{7-25}.

        Finally, $A_0(a,W)^{-1}$ is uniformly bounded on $K$ and $\mathscr L_\perp^{-1}$ is fixed, so $\sup_K\|\mathbb L_0^{-1}\|\le C_0$. Decreasing $\varepsilon_K$ until
        \[
        \sup_K\|\mathbb L_0^{-1}(\mathbb L_\varepsilon-\mathbb L_0)\|\le\frac12,
        \]
        the Neumann series gives $\sup_K\|\mathbb L_\varepsilon^{-1}\|\le2C_0$, proving \eqref{7-26}.
    \end{proof}

    The complementary inverse remains uniformly bounded at a simple fold because it uses $A_0$ and $\mathscr L_\perp$, not an inverse involving $L$. The vertical translation condition remains to be imposed. At $(Y,\phi)=(Y_0(a,W),0)$, the Bernoulli residual vanishes exactly, while \eqref{7-12}, \eqref{7-17}, and the metric estimates give $\|Q_0\mathcal E_\varepsilon(Y_0(a,W),a,0;W)\|_{\mathcal Y_{R_*}^{\perp}}\le C_K\varepsilon$. The next subsection uses this residual and the uniform inverse to construct a two-parameter $C^2$ complementary family, leaving the translation equation to the reduced drift.

    \smallskip

    \subsection{Two-parameter joint reduction and the regular drift limit}
    We now solve the coupled surface--core complement with independent parameters $a$ and $W$. Write $\theta=(a,W)$, $Y=Y_0(\theta)+\eta$, and $\zeta=(\eta,\phi)\in\mathbb X_{R_*}^s$, and define
    \[
        \mathscr S_\varepsilon(\theta,\zeta):=\mathscr R_\varepsilon(Y_0(\theta)+\eta,\phi;a,W).
    \]
    We work on a fixed neighborhood of $K$ with compact closure in $\mathcal U_{\rm par}$ and a correction ball independent of $\varepsilon$. Estimates are established on a slightly larger compact parameter set before restriction to $K$. In explicit parameter derivatives, $\zeta$ and the inner coordinate $X$ are held fixed.

    \begin{proposition}
        \label{prop7-2}
        There exist $\rho_0,\varepsilon_K,C_K>0$ such that, for every $0<\varepsilon<\varepsilon_K$ and $(a,W)$ in a neighborhood of $K$, the complementary equations $\mathscr R_\varepsilon(Y,\phi;a,W)=0$ have a unique solution satisfying $\|Y-Y_0(a,W)\|_{H^{s+2}}+\|\phi\|_{\mathcal X_{R_*}}<\rho_0$ and $\phi\in\mathcal X_{R_*}^{\perp}$. The resulting maps $(Y_\varepsilon(a,W),\phi_\varepsilon(a,W))$ are $C^2$ and satisfy
        \begin{equation}
            \|Y_\varepsilon-Y_0\|_{C^2(K;H_{\rm e}^{s+2})}+\|\phi_\varepsilon\|_{C^2(K;\mathcal X_{R_*}^{\perp})}\le C_K\varepsilon.
            \label{7-33}
        \end{equation}
        Moreover, the selected level $m_\varepsilon^{\rm sol}(a,W):=m_\varepsilon(Y_\varepsilon(a,W),a,W,\phi_\varepsilon(a,W))$ satisfies
        \begin{equation}
            \|m_\varepsilon^{\rm sol}\|_{C^2(K)}\le C_K\varepsilon.
            \label{7-34}
        \end{equation}
    \end{proposition}

    \begin{proof}
        Choose a fixed correction ball $\|\zeta\|_{\mathbb X_{R_*}^s}\le\rho$ sufficiently small that all configurations and selected levels remain in the neighborhoods of Section~7.1. All estimates on this ball are uniform in $\varepsilon$ and $\zeta$. Spatial H\"older norms below are taken on $\overline{B_{R_*}}$ unless indicated otherwise.

        Evaluate the geometric quantities at $Y=Y_0(\theta)+\eta$. The reconstruction estimates \eqref{7-8}, \eqref{7-12}, and \eqref{7-20}, together with the interior conformal bounds, give
        \begin{equation}
            \begin{aligned}
                \|j_\varepsilon-1\|_{C^2(K;C^{0,\gamma'})}+\|r_\varepsilon\|_{C^2(K;C^{2,\gamma'})}&\le C_K\varepsilon,\\
                \|\mathfrak P_H-I\|_{C^2(K;\mathcal L(\mathcal X_{R_*},C^{2,\gamma'}))}&\le C_K\varepsilon,\\
                \|\mathcal N_\varepsilon\|_{C^2(K;\mathcal L(\mathcal X_{R_*},H^{s+1}))}&\le C_K\varepsilon,
            \end{aligned}
            \label{7-35}
        \end{equation}
        where $\mathcal N_\varepsilon:=\mathcal N_{\varepsilon,Y_0(\theta)+\eta,a}$. Indeed, $Z_Y'(z_a+\delta X)=\alpha+\delta\int_0^1Z_Y''(z_a+t\delta X)X\,dt$, so $j_\varepsilon-1$ retains a factor $\delta=O(\varepsilon)$ after two parameter derivatives. The same follows from the explicit formulas for $r_\varepsilon$ and $\mathfrak P_H-I$. Differentiating the Fourier representation preceding \eqref{7-20}, with $\phi$ fixed, preserves the zero-mass cancellation and bounds each differentiated term by $C_K\varepsilon|\xi|(1+|\xi|)^2e^{-a_-|\xi|/2}\|g_\phi\|_{L^1}$ for $|\beta|\le2$. Integration against the Sobolev weight proves the boundary estimate in \eqref{7-35}.

        We first estimate the circulation level. Put $\widehat m_\varepsilon(\theta,\zeta):=m_\varepsilon(Y_0(\theta)+\eta,a,W,\phi)$ and $V_\varepsilon:=U+r_\varepsilon+\mathfrak P_H\phi-\widehat m_\varepsilon$. For $i,j\in\{a,W\}$, set $I_\varepsilon:=\int_{B_{R_*}}j_\varepsilon f'(V_\varepsilon)\,dX$, $T_i:=\partial_i r_\varepsilon+(\partial_i\mathfrak P_H)\phi$, and $T_{ij}:=\partial_{ij}r_\varepsilon+(\partial_{ij}\mathfrak P_H)\phi$. By circulation transversality, $|I_\varepsilon|\ge c_m>0$, while \eqref{7-35} gives $\|T_i\|_{C^{2,\gamma'}}+\|T_{ij}\|_{C^{2,\gamma'}}\le C_K\varepsilon$.

        Differentiating the exact constraint $\int_{B_{R_*}}j_\varepsilon f(V_\varepsilon)\,dX=\Gamma$ at fixed $\zeta$ yields
        \[
            \partial_i\widehat m_\varepsilon=\frac1{I_\varepsilon}\int_{B_{R_*}}\bigl[(\partial_i j_\varepsilon)f(V_\varepsilon)+j_\varepsilon f'(V_\varepsilon)T_i\bigr]\,dX.
        \]
        Hence $\partial_i\widehat m_\varepsilon=O(\varepsilon)$ and $\partial_iV_\varepsilon=T_i-\partial_i\widehat m_\varepsilon=O_{C^{2,\gamma'}}(\varepsilon)$. Differentiating again gives
        \[
            \begin{aligned}
                I_\varepsilon\partial_{ij}\widehat m_\varepsilon
                ={}&\int_{B_{R_*}}\bigl[(\partial_{ij}j_\varepsilon)f(V_\varepsilon)+j_\varepsilon f'(V_\varepsilon)T_{ij}\bigr]\,dX\\
                &+\int_{B_{R_*}}f'(V_\varepsilon)\bigl[(\partial_i j_\varepsilon)\partial_jV_\varepsilon+(\partial_j j_\varepsilon)\partial_iV_\varepsilon\bigr]\,dX\\
                &+\int_{B_{R_*}}j_\varepsilon f''(V_\varepsilon)\partial_iV_\varepsilon\partial_jV_\varepsilon\,dX.
            \end{aligned}
        \]
        The first integral is $O(\varepsilon)$ and the remaining ones are $O(\varepsilon^2)$. Thus $\partial_{ij}\widehat m_\varepsilon=O(\varepsilon)$ and $\partial_{ij}V_\varepsilon=O_{C^{2,\gamma'}}(\varepsilon)$. Since $\mathfrak m(\phi)$ has no explicit dependence on $\theta$, \eqref{7-17} now gives
        \begin{equation}
            \sup_{\|\zeta\|_{\mathbb X_{R_*}^s}\le\rho}\|\widehat m_\varepsilon(\cdot,\zeta)-\mathfrak m(\phi)\|_{C^2(K)}\le C_K\varepsilon.
            \label{7-38}
        \end{equation}
        These estimates apply throughout the trial ball and do not assume $\phi=O(\varepsilon)$.

        We next compare the nonlinear complementary maps. At fixed $\zeta$, the terms $-\Delta\phi$ and $f(U)$ are independent of $\theta$, so
        \[
        \partial_i\mathcal E_\varepsilon=-(\partial_i j_\varepsilon)f(V_\varepsilon)-j_\varepsilon f'(V_\varepsilon)\partial_iV_\varepsilon.
        \]
        Equations~\eqref{7-35}--\eqref{7-38} and a second differentiation give $\|\partial_i\mathcal E_\varepsilon\|_{\mathcal Y_{R_*}^{0}}+\|\partial_{ij}\mathcal E_\varepsilon\|_{\mathcal Y_{R_*}^{0}}\le C_K\varepsilon$. For continuity in the spatial H\"older target, the trial functions $V_\varepsilon$ have uniformly bounded $C^1$ norms and take values in a fixed compact interval. Since $f''\in C_{\rm loc}^{0,\gamma}$, the compositions $f''(V_\varepsilon)$ are uniformly bounded in $C^{0,\gamma}$. For such uniformly bounded families, H\"older interpolation gives $\|f''(V_1)-f''(V_2)\|_{C^{0,\gamma'}}\le C\|V_1-V_2\|_{C^0}^{\gamma-\gamma'}$. This ensures continuity of the second derivatives of the nonlinear map in the $C^{0,\gamma'}$ target, without requiring $f'''$. The differentiated residuals retain fixed support and zero integral by \eqref{7-14} and the circulation constraint.

        At zero order, \eqref{7-35} and \eqref{7-38} imply $V_\varepsilon-[U+\phi-\mathfrak m(\phi)]=O_{C^{2,\gamma'}}(\varepsilon)$, hence $\mathcal E_\varepsilon-\mathcal E_0(\phi)=O_{\mathcal Y_{R_*}^{0}}(\varepsilon)$. For the surface component, \eqref{7-21} gives $\|\mathcal B_\varepsilon-B(Y_0(\theta)+\eta,a,W)\|_{C^2(K;H_{\rm e}^s)}\le C_K(\varepsilon\|\phi\|_{\mathcal X_{R_*}}+\varepsilon^2\|\phi\|_{\mathcal X_{R_*}}^2)$. Defining $\mathscr S_0(\theta,\zeta):=(B(Y_0(\theta)+\eta,a,W),Q_0\mathcal E_0(\phi))$, we obtain
        \begin{equation}
            \sup_{\|\zeta\|_{\mathbb X_{R_*}^s}\le\rho}\|(\mathscr S_\varepsilon-\mathscr S_0)(\cdot,\zeta)\|_{C^2(K;\mathbb Y_{R_*}^s)}\le C_K\varepsilon.
            \label{7-39}
        \end{equation}

        Since $B(Y_0(\theta),a,W)=0$ and $\mathcal E_0(0)=0$, we have $\mathscr S_0(\theta,0)=0$. Moreover,
        \[
        B(Y_0(\theta)+\eta,a,W)-B(Y_0(\theta),a,W)=\int_0^1D_YB(Y_0(\theta)+t\eta,a,W)[\eta]\,dt.
        \]
        The analyticity estimates of Lemma~\ref{lem2-2} show that its first two parameter derivatives retain the factor $\|\eta\|_{H^{s+2}}$. Combining this with \eqref{7-39} gives
        \begin{equation}
            \begin{aligned}
                \|\mathscr S_\varepsilon(\cdot,0)\|_{C^2(K;\mathbb Y_{R_*}^s)}&\le C_K\varepsilon,\\
                \|\partial_i\mathscr S_\varepsilon(\theta,\zeta)\|_{\mathbb Y_{R_*}^s}+\|\partial_{ij}\mathscr S_\varepsilon(\theta,\zeta)\|_{\mathbb Y_{R_*}^s}&\le C_K(\varepsilon+\|\zeta\|_{\mathbb X_{R_*}^s}).
            \end{aligned}
            \label{7-40}
        \end{equation}
        The same reconstruction bounds and differentiated circulation constraint control derivatives involving the correction variables. As $|I_\varepsilon|\ge c_m$, implicit differentiation of $\widehat m_\varepsilon$ gives uniformly bounded derivatives of total order at most two, involving only $f'$ and $f''$. Consequently,
        \begin{equation}
            \|D_\zeta^2\mathscr S_\varepsilon\|_{\mathcal L^2(\mathbb X_{R_*}^s,\mathbb Y_{R_*}^s)}+\sum_{i\in\{a,W\}}\|D_\zeta\partial_i\mathscr S_\varepsilon\|_{\mathcal L(\mathbb X_{R_*}^s,\mathbb Y_{R_*}^s)}\le C_K.
            \label{7-41}
        \end{equation}
        All bounds hold on the fixed trial ball. The argument requires only uniform boundedness, not $O(\varepsilon)$ convergence, of these mixed derivatives.

        By Lemma~\ref{lem7-1}, $\mathbb L_\varepsilon(\theta)=D_\zeta\mathscr S_\varepsilon(\theta,0)$ has a uniformly bounded inverse, with norm at most $C_0$. Define
        \[
            \mathscr T_{\varepsilon,\theta}(\zeta):=\zeta-\mathbb L_\varepsilon(\theta)^{-1}\mathscr S_\varepsilon(\theta,\zeta).
        \]
        Equations~\eqref{7-40}--\eqref{7-41} give $\|\mathscr T_{\varepsilon,\theta}(0)\|\le C_K\varepsilon$ and $\|D_\zeta\mathscr T_{\varepsilon,\theta}(\zeta)\|\le C\|\zeta\|$. Choose $\rho_0\in(0,\rho)$ with $C\rho_0\le1/2$, and then $\varepsilon_K$ small enough that $2C_K\varepsilon_K<\rho_0$. The map $\mathscr T_{\varepsilon,\theta}$ is then a contraction of the fixed closed ball of radius $\rho_0$ into itself. Its unique fixed point $\zeta_\varepsilon(\theta)$ satisfies $\|\zeta_\varepsilon(\theta)\|\le2C_K\varepsilon$, proving existence, uniqueness, and the zero-order estimate.

        Set $\mathbb A_\varepsilon(\theta):=D_\zeta\mathscr S_\varepsilon(\theta,\zeta_\varepsilon(\theta))$. Since $\mathbb A_\varepsilon-\mathbb L_\varepsilon=O(\varepsilon)$, its inverse is uniformly bounded. The $C^2$ implicit function theorem gives local $C^2$ dependence on $\theta$, and uniqueness identifies the solutions on overlaps. Differentiating $\mathscr S_\varepsilon(\theta,\zeta_\varepsilon(\theta))=0$ yields $\mathbb A_\varepsilon\partial_i\zeta_\varepsilon=-\partial_i\mathscr S_\varepsilon$, hence $\partial_i\zeta_\varepsilon=O(\varepsilon)$ by \eqref{7-40}. A second differentiation gives
        \[
            \begin{aligned}
                \mathbb A_\varepsilon\partial_{ij}\zeta_\varepsilon
                ={}&-\partial_{ij}\mathscr S_\varepsilon-D_\zeta\partial_i\mathscr S_\varepsilon[\partial_j\zeta_\varepsilon]-D_\zeta\partial_j\mathscr S_\varepsilon[\partial_i\zeta_\varepsilon]\\
                &-D_\zeta^2\mathscr S_\varepsilon[\partial_i\zeta_\varepsilon,\partial_j\zeta_\varepsilon],
            \end{aligned}
        \]
        with derivatives evaluated at $(\theta,\zeta_\varepsilon(\theta))$. The first three terms are $O(\varepsilon)$ and the last is $O(\varepsilon^2)$ by \eqref{7-40}--\eqref{7-41}. Uniform invertibility of $\mathbb A_\varepsilon$ gives $\partial_{ij}\zeta_\varepsilon=O(\varepsilon)$, proving \eqref{7-33}.

        Finally, $m_\varepsilon^{\rm sol}(\theta)=\widehat m_\varepsilon(\theta,\zeta_\varepsilon(\theta))$. Its value is $O(\varepsilon)$ by \eqref{7-17}, while \eqref{7-38} and the chain rule give the same bound for its first derivatives. At second order, the explicit derivatives $\partial_{ij}\widehat m_\varepsilon$ are $O(\varepsilon)$, and every remaining term contains a first or second derivative of $\zeta_\varepsilon$ multiplied by uniformly bounded derivatives of $\widehat m_\varepsilon$. This proves \eqref{7-34}.
    \end{proof}

    On the complementary family, write $\alpha_\varepsilon(a,W):=Z_{Y_\varepsilon(a,W)}'(-\mathrm ia)$ and $H_\varepsilon:=a\alpha_\varepsilon/\varepsilon$. Set $v_\varepsilon:=r_\varepsilon(Y_\varepsilon,a,W;\,\cdot)+\mathfrak P_{H_\varepsilon}\phi_\varepsilon-m_\varepsilon^{\rm sol}$. Equations~\eqref{7-33}--\eqref{7-35} give
    \begin{equation}
        \|v_\varepsilon\|_{C^2(K;C^{2,\gamma'}(\overline{B_{R_*}}))}\le C_K\varepsilon.
        \label{7-44}
    \end{equation}
    Thus the shifted inner stream function is $U+v_\varepsilon$ on the source disk. Define the normalized scalar drift
    \begin{equation}
        \mathscr G_\varepsilon(a,W):=\frac{\alpha_\varepsilon(a,W)}{\Gamma\varepsilon}\ell_0\bigl(\mathcal E_\varepsilon(Y_\varepsilon(a,W),a,\phi_\varepsilon(a,W);W)\bigr).
        \label{7-45}
    \end{equation}
    Unless otherwise stated, parameter derivatives below include the complementary solution and selected level, while the inner coordinate $X$ is fixed.

    \begin{lemma}
        \label{lem7-3}
        For sufficiently small $\varepsilon_K$, the drift $\mathscr G_\varepsilon$ is $C^2$ on a neighborhood of $K$ and satisfies
        \begin{equation}
            \left\|\ell_0(\mathcal E_\varepsilon)-\frac{\Gamma\varepsilon}{\alpha_\varepsilon}P(Y_\varepsilon,a,W)\right\|_{C^2(K)}\le C_K\varepsilon^2,
            \label{7-46}
        \end{equation}
        where the residual is evaluated along the complementary family. Consequently,
        \begin{equation}
            \|\mathscr G_\varepsilon-\mathscr G_0\|_{C^2(K)}\le C_K\varepsilon,\qquad \mathscr G_0(a,W):=P(Y_0(a,W),a,W).
            \label{7-47}
        \end{equation}
        Within this family, the full mass-constrained regularized equations hold if and only if $\mathscr G_\varepsilon(a,W)=0$.
    \end{lemma}

    \begin{proof}
        Set $\delta_\varepsilon:=\varepsilon/\alpha_\varepsilon$, $b_\varepsilon:=E_a(Y_\varepsilon'')$, and $d_0:=W-\Gamma/(4\pi a)$. All expansions below hold on $\overline{B_{R_*}}$, including their first two parameter derivatives.

        First, let $h_\varepsilon:=(\mathfrak P_{H_\varepsilon}-I)\phi_\varepsilon$. By \eqref{7-8}, the reflection operator and its first two parameter derivatives are $O(\varepsilon)$ on the source disk. Since $\phi_\varepsilon=O_{C^2(K;\mathcal X_{R_*}^{\perp})}(\varepsilon)$, the product rule gives
        \begin{equation}
            \|h_\varepsilon\|_{C^2(K;C^{2,\gamma'}(\overline{B_{R_*}}))}\le C_K\varepsilon^2.
            \label{7-48}
        \end{equation}

        By symmetry, $Z_{Y_\varepsilon}''(z_a)=\mathrm ib_\varepsilon$. Taylor expansion at $z_a=-\mathrm ia$ gives
        \[
        Z_{Y_\varepsilon}'(z_a+\delta_\varepsilon X)=\alpha_\varepsilon+\mathrm i\delta_\varepsilon b_\varepsilon X+\delta_\varepsilon^2X^2\int_0^1(1-t)Z_{Y_\varepsilon}'''(z_a+t\delta_\varepsilon X)\,dt.
        \]
        Taking the squared modulus and using $\operatorname{Re}(\mathrm iX)=-X_2$, together with \eqref{7-12}, yields
        \[
            j_\varepsilon=1+\varepsilon j_{1,\varepsilon}+R_{j,\varepsilon},\qquad r_\varepsilon=\varepsilon r_{1,\varepsilon}+R_{r,\varepsilon},
        \]
        where $j_{1,\varepsilon}(X):=-2b_\varepsilon X_2/\alpha_\varepsilon^2$ and $r_{1,\varepsilon}(X):=d_0X_2/\alpha_\varepsilon$. The remainders satisfy
        \begin{equation}
            \|R_{j,\varepsilon}\|_{C^2(K;C^{0,\gamma'}(\overline{B_{R_*}}))}+\|R_{r,\varepsilon}\|_{C^2(K;C^{2,\gamma'}(\overline{B_{R_*}}))}\le C_K\varepsilon^2.
            \label{7-50}
        \end{equation}
        Indeed, the evaluation points remain uniformly separated from the free surface, and $\delta_\varepsilon$ and its first two parameter derivatives are $O(\varepsilon)$. The interior conformal estimates therefore preserve the quadratic remainder after two parameter differentiations.

        Define $T_f(v):=f(U+v)-f(U)-f'(U)v$ and $c_\varepsilon:=m_\varepsilon^{\rm sol}-m_0[\phi_\varepsilon]$. Substituting $v_\varepsilon=\phi_\varepsilon-m_\varepsilon^{\rm sol}+\varepsilon r_{1,\varepsilon}+R_{r,\varepsilon}+h_\varepsilon$ into \eqref{7-18} gives
        \begin{equation}
            \mathcal E_\varepsilon=\mathscr L_0\phi_\varepsilon+c_\varepsilon f'(U)+\varepsilon F_{1,\varepsilon}+\mathfrak r_\varepsilon,
            \label{7-51}
        \end{equation}
        where
        \[
        F_{1,\varepsilon}:=-j_{1,\varepsilon}f(U)-r_{1,\varepsilon}f'(U)=\frac{2b_\varepsilon}{\alpha_\varepsilon^2}X_2f(U)-\frac{d_0}{\alpha_\varepsilon}X_2f'(U)
        \]
        and
        \begin{equation}
            \begin{aligned}
                \mathfrak r_\varepsilon={}&-R_{j,\varepsilon}f(U)-f'(U)(R_{r,\varepsilon}+h_\varepsilon)\\
                &-(j_\varepsilon-1)f'(U)v_\varepsilon-j_\varepsilon T_f(v_\varepsilon).
            \end{aligned}
            \label{7-52}
        \end{equation}
        The level difference $c_\varepsilon$ is kept explicit because its vertical translation projection vanishes exactly.

        The first three terms in \eqref{7-52} are $O_{C^2(K;C^{0,\gamma'})}(\varepsilon^2)$ by \eqref{7-44}, \eqref{7-48}, and \eqref{7-50}. For the nonlinear term, Taylor's formula gives $T_f(v_\varepsilon)=v_\varepsilon^2\int_0^1(1-t)f''(U+tv_\varepsilon)\,dt$. To avoid assuming $f'''$, differentiate the defining difference directly:
        \[
        \begin{aligned}
            \partial_iT_f(v_\varepsilon)&=[f'(U+v_\varepsilon)-f'(U)]\partial_iv_\varepsilon,\\
            \partial_{ij}T_f(v_\varepsilon)&=f''(U+v_\varepsilon)\partial_iv_\varepsilon\partial_jv_\varepsilon+[f'(U+v_\varepsilon)-f'(U)]\partial_{ij}v_\varepsilon.
        \end{aligned}
        \]
        Since $v_\varepsilon$ and its first two parameter derivatives are $O(\varepsilon)$ in $C^{2,\gamma'}$, while $f''(U+tv_\varepsilon)$ is uniformly bounded in $C^{0,\gamma'}$ for $\gamma'<\gamma$, every term is $O(\varepsilon^2)$. The nonlinear sources vanish in a fixed annulus around $\partial B_{R_*}$, so zero extension preserves these bounds. Consequently,
        \begin{equation}
            \|\mathfrak r_\varepsilon\|_{C^2(K;C^{0,\gamma'}(\mathbb R^2))}\le C_K\varepsilon^2,\qquad \operatorname{supp}\mathfrak r_\varepsilon\subset\overline{B_{R_*}}.
            \label{7-53}
        \end{equation}
        The remainder need not have zero integral: the circulation constraint applies to the full residual $\mathcal E_\varepsilon$.

        Project \eqref{7-51} onto the translation mode $Z_0=\partial_2U$. Since $-\Delta Z_0=f'(U)Z_0$, Green's identity and the decay of $\phi_\varepsilon$ and $Z_0$ give $\ell_0(\mathscr L_0\phi_\varepsilon)=0$. Also,
        \[
        \ell_0(f'(U))=\int_{\mathbb R^2}\partial_2f(U)\,dX=0.
        \]
        As $U$, $\mathscr L_0$, and $\ell_0$ are independent of $(a,W)$, these cancellations hold after parameter differentiation. Thus only $\varepsilon F_{1,\varepsilon}$ and $\mathfrak r_\varepsilon$ contribute.

        The radial identities from \cite[Proof of Proposition~5.4]{DuZouPartI} give $\int_{\mathbb R^2}X_2f'(U)\partial_2U\,dX=-\Gamma$. Writing $F_f(\tau):=\int_0^\tau f(r)\,dr$, the radial equation and $U(1)=0$, $U'(1)=-\Gamma/(2\pi)$ yield
        \[
            \int_{\mathbb R^2}F_f(U)\,dX=\frac{\Gamma^2}{8\pi},\qquad \int_{\mathbb R^2}X_2f(U)\partial_2U\,dX=-\frac{\Gamma^2}{8\pi}.
        \]
        Therefore,
        \begin{equation}
            \ell_0(F_{1,\varepsilon})=\frac{\Gamma d_0}{\alpha_\varepsilon}-\frac{\Gamma^2b_\varepsilon}{4\pi\alpha_\varepsilon^2}=\frac{\Gamma}{\alpha_\varepsilon}P(Y_\varepsilon,a,W).
            \label{7-55}
        \end{equation}
        The first term represents the background and image-vortex drift, while the second is the conformal metric correction generated by the free surface. Together they yield the full regular drift of the coupled point-vortex problem.

        Applying $\ell_0$ to \eqref{7-51} and using \eqref{7-53} and \eqref{7-55}, we obtain
        \[
        \ell_0(\mathcal E_\varepsilon)=\frac{\Gamma\varepsilon}{\alpha_\varepsilon}P(Y_\varepsilon,a,W)+O_{C^2(K)}(\varepsilon^2),
        \]
        which proves \eqref{7-46}. By \eqref{7-45}, $\mathscr G_\varepsilon=P(Y_\varepsilon,a,W)+O_{C^2(K)}(\varepsilon)$. Analyticity of $P$ and \eqref{7-33} give $\|P(Y_\varepsilon,a,W)-P(Y_0,a,W)\|_{C^2(K)}\le C_K\varepsilon$, proving \eqref{7-47}.

        Finally, $Q_0\mathcal E_\varepsilon=0$ implies $\mathcal E_\varepsilon=\ell_0(\mathcal E_\varepsilon)\eta_0$. Since $\alpha_\varepsilon/(\Gamma\varepsilon)>0$, this translation component vanishes precisely when $\mathscr G_\varepsilon=0$. The Bernoulli equation, Dirichlet condition, and exact circulation constraint have already been imposed, so the scalar drift equation is equivalent to the full regularized system within the complementary family.
    \end{proof}

    Set $L_\varepsilon:=\partial_a\mathscr G_\varepsilon$ and $M_\varepsilon:=\partial_W\mathscr G_\varepsilon$. Lemma~\ref{lem7-3} gives
    \[
        \|L_\varepsilon-L\|_{C^1(K)}+\|M_\varepsilon-M\|_{C^1(K)}\le C_K\varepsilon.
    \]
    Since $L^2+M^2$ is uniformly positive on the reference arc, the reduced gradient remains nonzero for sufficiently small $\varepsilon$, including near simple folds where $L=0$ and $M\ne0$. This nondegeneracy allows us to construct the regularized solution arcs and prove persistence of their specified simple folds.

    \smallskip

    \subsection{Persistence of compact arcs and folds}
    Let $I=[t_-,t_+]$ be a compact interval and $\bm{\gamma}:I\to\mathscr C_0$ a smooth regular embedding with image $\mathscr K$, where $A$ is invertible and $L^2+M^2>0$. Write $\bm{\gamma}(t)=(Y(t),a(t),W(t))$. The regularity hypotheses allow an extension to a slightly larger interval $I'\supset I$, used to obtain uniform estimates up to the endpoints. Use finitely many Bernoulli charts; the projection to the $(a,W)$ plane need not be globally injective. The parameter $t$ labels steady configurations and is not physical time.

    Use the complementary solutions $(Y_\varepsilon(a,W),\phi_\varepsilon(a,W))$ from Proposition~\ref{prop7-2}, with the same core profile, circulation constraint, fixed-source spaces, and translation normalization in every chart. The remaining scalar drift equation selects the neighboring regularized arc.

    \begin{proposition}
        \label{prop7-4}
        There exist $\varepsilon_{\mathscr K},C_{\mathscr K}>0$ such that, for every $0<\varepsilon<\varepsilon_{\mathscr K}$, the regularized problem admits a regular $C^2$ embedded solution arc $\mathscr C_{\varepsilon,\mathscr K}$ parametrized by $\bm{\gamma}_\varepsilon(t):=(\widehat Y_\varepsilon(t),a_\varepsilon(t),W_\varepsilon(t))$, $t\in I$, with
        \begin{equation}
            \|\bm{\gamma}_\varepsilon-\bm{\gamma}\|_{C^2(I;H_{\rm e}^{s+2}(\mathbb R)\times\mathbb R^2)}+\|\widehat\phi_\varepsilon\|_{C^2(I;\mathcal X_{R_*}^{\perp})}\le C_{\mathscr K}\varepsilon.
            \label{7-57}
        \end{equation}
        The correspondence is locally unique under the prescribed mass and translation normalizations and the transverse condition \eqref{7-62}.

        The associated smooth-vorticity waves $(\Omega_\varepsilon(t),\Psi_\varepsilon(t),\mu_\varepsilon(t))$ travel with velocity $-W_\varepsilon(t)e_1$ and satisfy
        \begin{equation}
            -\Delta_p\Psi_\varepsilon(t,p)=\omega_{\varepsilon,t}(p),\qquad \omega_{\varepsilon,t}(p)=\varepsilon^{-2}f\bigl(\Psi_\varepsilon(t,p)-\mu_\varepsilon(t)\bigr)\quad\text{in }\Omega_\varepsilon(t),
            \label{7-58}
        \end{equation}
        together with the kinematic and Bernoulli boundary conditions and the laboratory-frame normalization at infinity. Their nonnegative vorticity satisfies
        \begin{equation}
            \int_{\Omega_\varepsilon(t)}\omega_{\varepsilon,t}(p)\,dp=\Gamma,\qquad \operatorname{supp}\omega_{\varepsilon,t}\subset B_{C_{\mathscr K}\varepsilon}(p_\varepsilon(t)),
            \label{7-59}
        \end{equation}
        where $p_\varepsilon(t):=Z_{\widehat Y_\varepsilon(t)}(-\mathrm ia_\varepsilon(t))$. With $p(t):=Z_{Y(t)}(-\mathrm ia(t))$, we have $\sup_I|p_\varepsilon-p|\le C_{\mathscr K}\varepsilon$ and $\omega_{\varepsilon,t}\rightharpoonup\Gamma\delta_{p(t)}$ uniformly in $t$ against bounded Lipschitz test functions.

        If $\mathscr K$ contains a specified interior simple fold $(Y_\star,a_\star,W_\star)$, then, for sufficiently small $\varepsilon>0$, the regularized arc contains a unique nearby simple fold in the prescribed normalized configuration neighborhood. Its parameters satisfy
        \begin{equation}
            \mathscr G_\varepsilon(a_{\varepsilon,\star},W_{\varepsilon,\star})=0,\qquad L_\varepsilon(a_{\varepsilon,\star},W_{\varepsilon,\star})=0.
            \label{7-60}
        \end{equation}
        Writing $\theta_{\varepsilon,\star}:=(a_{\varepsilon,\star},W_{\varepsilon,\star})$ and $\chi_{\varepsilon,\star}:=-\partial_aL_\varepsilon(\theta_{\varepsilon,\star})/M_\varepsilon(\theta_{\varepsilon,\star})$, we have
        \begin{equation}
            |a_{\varepsilon,\star}-a_\star|+|W_{\varepsilon,\star}-W_\star|+|\chi_{\varepsilon,\star}-\chi_\star|\le C_\star\varepsilon,\qquad \chi_\star:=-\frac{(\partial_aL)_\star}{M_\star}\ne0.
            \label{7-61}
        \end{equation}
        In particular, the quadratic turning direction and local fixed-speed two-wave structure are preserved.
    \end{proposition}

    \begin{proof}
        To construct the regularized arc, set $\theta(t):=(a(t),W(t))$, with $L(t),M(t)$ evaluated along $\bm{\gamma}$, and define
        \[
        d(t):=L(t)^2+M(t)^2,\qquad n(t):=\frac{(L(t),M(t))}{d(t)}.
        \]
        On the extended arc, $d$ is uniformly positive. In each Bernoulli chart, $Y(t)=Y_0(\theta(t))$, so regularity of $\bm{\gamma}$ implies $\theta'(t)\ne0$. Differentiating $\mathscr G_0(\theta(t))=0$ gives $L(t)a'(t)+M(t)W'(t)=0$. Hence $n(t)$ is transverse to the reference tangent, with $\nabla\mathscr G_0(\theta(t))\cdot n(t)=1$.

        Seek the regularized parameters in the form $\theta_\varepsilon(t)=\theta(t)+\tau_\varepsilon(t)n(t)$, equivalently,
        \begin{equation}
            M(t)\bigl(a_\varepsilon(t)-a(t)\bigr)-L(t)\bigl(W_\varepsilon(t)-W(t)\bigr)=0.
            \label{7-62}
        \end{equation}
        This is a correspondence condition between curves, not an additional physical equation. At a fold, $L=0$ and $M\ne0$, so paired waves have the same conformal depth but need not have the same speed.

        In each chart, put $g_\varepsilon(t,\tau):=\mathscr G_\varepsilon(\theta(t)+\tau n(t))$. Since $g_0(t,0)=0$ and $\partial_\tau g_0(t,0)=1$, compactness and a finite covering give a uniform $\tau_0>0$ such that $\partial_\tau g_0\ge3/4$ for $|\tau|\le\tau_0$. By \eqref{7-47}, $\|g_\varepsilon-g_0\|_{C^2}\le C_{\mathscr K}\varepsilon$, so $\partial_\tau g_\varepsilon\ge1/2$ for small $\varepsilon$. Moreover, $|g_\varepsilon(t,0)|\le C_{\mathscr K}\varepsilon$. Thus $g_\varepsilon(t,\cdot)$ has opposite signs at $\pm\tau_0$ and a unique zero $\tau_\varepsilon(t)$ satisfying $|\tau_\varepsilon(t)|\le C_{\mathscr K}\varepsilon$.

        The scalar implicit function theorem gives $\tau_\varepsilon\in C^2(I)$. Differentiating $g_\varepsilon(t,\tau_\varepsilon(t))=0$ yields
        \[
        (\partial_\tau g_\varepsilon)\tau_\varepsilon'=-\partial_tg_\varepsilon,\qquad (\partial_\tau g_\varepsilon)\tau_\varepsilon''=-\partial_{tt}g_\varepsilon-2(\partial_{t\tau}g_\varepsilon)\tau_\varepsilon'-(\partial_{\tau\tau}g_\varepsilon)(\tau_\varepsilon')^2.
        \]
        Since $g_0(t,0)\equiv0$, its first two $t$-derivatives vanish there. The $C^2$ comparison and $\tau_\varepsilon=O(\varepsilon)$ therefore give $\partial_tg_\varepsilon(t,\tau_\varepsilon(t))=O(\varepsilon)$ and $\partial_{tt}g_\varepsilon(t,\tau_\varepsilon(t))=O(\varepsilon)$. The preceding identities imply $\tau_\varepsilon',\tau_\varepsilon''=O(\varepsilon)$, hence
        \begin{equation}
            \|\tau_\varepsilon\|_{C^2(I)}+\|\theta_\varepsilon-\theta\|_{C^2(I;\mathbb R^2)}\le C_{\mathscr K}\varepsilon.
            \label{7-63}
        \end{equation}

        Define the surface and core corrections along the arc by
\[
\widehat Y_\varepsilon(t):=Y_\varepsilon(\theta_\varepsilon(t)),\qquad \widehat\phi_\varepsilon(t):=\phi_\varepsilon(\theta_\varepsilon(t)).
\]
Proposition~\ref{prop7-2} and \eqref{7-63} give $\|\widehat\phi_\varepsilon\|_{C^2(I;\mathcal X_{R_*}^{\perp})}\le C_{\mathscr K}\varepsilon$. For the surface, decompose
        \[
        \widehat Y_\varepsilon-Y=(Y_\varepsilon-Y_0)\circ\theta_\varepsilon+Y_0\circ\theta_\varepsilon-Y_0\circ\theta.
        \]
        The first term is $O(\varepsilon)$ in $C^2$ by \eqref{7-33}, and the second by \eqref{7-63} and smoothness of $Y_0$. This proves \eqref{7-57}; any higher derivatives used in the composition comparison concern only the smooth reference maps.

        On overlaps of the reference configuration charts, Bernoulli uniqueness identifies the surface maps $Y_0$, and the local uniqueness in Proposition~\ref{prop7-2} identifies the corresponding normalized complementary solutions after the neighborhoods are suitably restricted. The reduced drift and the transverse condition \eqref{7-62} are therefore consistent on these overlaps, so the locally constructed families patch into a $C^2$ solution curve. Configurations with identical $(a,W)$ but distinct surfaces remain in separate configuration neighborhoods and are not identified.

        The resulting curve is regular by the $C^1$ estimate in \eqref{7-57} and the nonvanishing tangent of the reference embedding. To verify injectivity, first use regularity and $C^1$ closeness to obtain uniform local injectivity on sufficiently short parameter intervals. For parameter pairs separated by a fixed positive distance, compactness of $I$ and injectivity of $\bm{\gamma}$ give a positive lower bound for their distance in the full space $H_{\rm e}^{s+2}(\mathbb R)\times\mathbb R^2$. The $C^0$ estimate in \eqref{7-57} preserves this separation for small $\varepsilon$. Hence $\bm{\gamma}_\varepsilon$ is a regular embedded arc. Local uniqueness holds within the prescribed normalized configuration neighborhoods and under the transverse correspondence \eqref{7-62}.

        We next prove fold persistence. Work in a Bernoulli chart containing $\theta_\star:=(a_\star,W_\star)$, and set $\mathfrak H_\varepsilon:=(\mathscr G_\varepsilon,L_\varepsilon)$. By \eqref{7-47},
        \[
        \|\mathfrak H_\varepsilon-\mathfrak H_0\|_{C^1(K_\star)}\le C_\star\varepsilon
        \]
        on a compact neighborhood $K_\star$ of $\theta_\star$. At the reference fold, $\mathfrak H_0(\theta_\star)=0$ and, by \eqref{5-11},
        \[
        J_\star:=D\mathfrak H_0(\theta_\star),\qquad \det J_\star=-M_\star(\partial_aL)_\star\ne0.
        \]
        Choose a closed ball $B_\rho(\theta_\star)$ contained in the chart and in a transverse-coordinate neighborhood over the interior of $I$, such that $\|I-J_\star^{-1}D\mathfrak H_0(\theta)\|\le1/4$ throughout the ball. For small $\varepsilon$, the map $\mathscr T_\varepsilon(\theta):=\theta-J_\star^{-1}\mathfrak H_\varepsilon(\theta)$ has Lipschitz constant at most $1/2$ there, while $|\mathscr T_\varepsilon(\theta_\star)-\theta_\star|\le C_\star\varepsilon$. Decreasing $\varepsilon$ makes $\mathscr T_\varepsilon$ a contraction of the ball into itself, with a unique fixed point $\theta_{\varepsilon,\star}$. Consequently,
        \[
        \mathfrak H_\varepsilon(\theta_{\varepsilon,\star})=0,\qquad |\theta_{\varepsilon,\star}-\theta_\star|\le2|J_\star^{-1}\mathfrak H_\varepsilon(\theta_\star)|\le C_\star\varepsilon.
        \]
        Thus \eqref{7-60} holds, and the fold lies on the regularized arc constructed above. No differentiability with respect to $\varepsilon$ at zero is needed.

        The $C^2$ drift comparison, the location estimate, and smoothness of $\mathscr G_0$ give
        \[
        M_\varepsilon(\theta_{\varepsilon,\star})=M_\star+O(\varepsilon),\qquad \partial_aL_\varepsilon(\theta_{\varepsilon,\star})=(\partial_aL)_\star+O(\varepsilon).
        \]
        Both remain nonzero, and therefore
        \[
        \chi_{\varepsilon,\star}:=-\frac{\partial_aL_\varepsilon(\theta_{\varepsilon,\star})}{M_\varepsilon(\theta_{\varepsilon,\star})}=\chi_\star+O(\varepsilon),
        \]
        proving \eqref{7-61}. The estimate uses smoothness of the limiting drift, not a uniform third-derivative bound for $\mathscr G_\varepsilon$.

        Since $M_\varepsilon(\theta_{\varepsilon,\star})\ne0$, the reduced equation determines a local $C^2$ function $W=w_\varepsilon(a)$. Differentiating $\mathscr G_\varepsilon(a,w_\varepsilon(a))=0$ gives $w_\varepsilon'(a_{\varepsilon,\star})=0$ and $w_\varepsilon''(a_{\varepsilon,\star})=\chi_{\varepsilon,\star}\ne0$. Hence
        \begin{equation}
            w_\varepsilon(a)-W_{\varepsilon,\star}=\frac{\chi_{\varepsilon,\star}}2(a-a_{\varepsilon,\star})^2+o(|a-a_{\varepsilon,\star}|^2).
            \label{7-64}
        \end{equation}
        Continuity of $w_\varepsilon''$ makes the speed strictly monotone on either side of the fold in a sufficiently small interval. Thus nearby fixed speeds have exactly two normalized solutions if $\chi_{\varepsilon,\star}(W-W_{\varepsilon,\star})>0$, none if the product is negative, and only the fold solution when $W=W_{\varepsilon,\star}$.

        The turning point also corresponds to a fold of the full normalized steady system. Adjoin to $\mathscr R_\varepsilon$ the scalar residual $\alpha(Y,a)(\Gamma\varepsilon)^{-1}\ell_0(\mathcal E_\varepsilon)$. At a solution, the derivative of the complementary equations with respect to $(Y,\phi)$ is invertible by Proposition~\ref{prop7-2}. Eliminating these variables leaves the scalar linearization $L_\varepsilon\ell+M_\varepsilon\nu$ in the depth and speed variations $(\ell,\nu)$. Since $M_\varepsilon\ne0$ near the fold, the full derivative is surjective when both parameters vary. At fixed speed it is invertible for $L_\varepsilon\ne0$, while at the fold its kernel is one-dimensional, generated in $(Y,\phi,a)$ coordinates by
        \[
        \bigl(\partial_aY_\varepsilon(\theta_{\varepsilon,\star}),\partial_a\phi_\varepsilon(\theta_{\varepsilon,\star}),1\bigr).
        \]
        Here the derivatives are taken at fixed $W$. Thus the scalar fold corresponds precisely to the loss of fixed-speed invertibility in the full normalized problem.

        It remains to reconstruct the physical waves. Fix $t\in I$ and abbreviate $Y:=\widehat Y_\varepsilon(t)$, $a:=a_\varepsilon(t)$, $W:=W_\varepsilon(t)$, and $\phi:=\widehat\phi_\varepsilon(t)$. Define
        \[
        \begin{aligned}
            \Phi_\varepsilon(t,z)&:=P_{\varepsilon,Y,a}(z)+\mathcal C_{\varepsilon,Y,a}[\phi](z)+W\operatorname{Im}z,\\
            \mu_\varepsilon(t)&:=\mu_{\varepsilon,Y,a,W}^{\rm app}+m_\varepsilon(Y,a,W,\phi),\\
            \Omega_\varepsilon(t)&:=Z_Y(\Pi_-),\qquad \Psi_\varepsilon(t,p):=\Phi_\varepsilon(t,Z_Y^{-1}(p)).
        \end{aligned}
        \]
        By Lemma~\ref{lem7-3}, $\mathscr G_\varepsilon(\theta_\varepsilon(t))=0$ completes the inner equation, giving \eqref{7-58}. The Dirichlet reconstruction yields $\Psi_\varepsilon=0$ on the free surface, and $\mathcal B_\varepsilon=0$ gives the Bernoulli condition. The relative velocity is $-\nabla_p^\perp\Psi_\varepsilon$, while the laboratory-frame velocity $-\nabla_p^\perp\Psi_\varepsilon-We_1$ tends to zero at infinity.

        For completeness, let $F_f(\tau):=\int_0^\tau f(r)\,dr$ and define
        \[
        \mathsf p_\varepsilon(t,p):=\frac{W_\varepsilon(t)^2}{2}-\frac12|\nabla_p\Psi_\varepsilon(t,p)|^2-gp_2-\varepsilon^{-2}F_f(\Psi_\varepsilon(t,p)-\mu_\varepsilon(t)).
        \]
        The identity $(v_{\rm rel}\cdot\nabla_p)v_{\rm rel}=\nabla_p(|v_{\rm rel}|^2/2)+\omega_{\varepsilon,t}\nabla_p\Psi_\varepsilon$, together with \eqref{7-58}, gives the stationary Euler momentum equation. Since $\mu_\varepsilon(t)>0$ for small $\varepsilon$ by its logarithmic leading term, $F_f(-\mu_\varepsilon(t))=0$ on the free surface; the Bernoulli equation then yields $\mathsf p_\varepsilon=\sigma\kappa_\varepsilon$. The reconstruction therefore solves the full traveling-wave problem, as in Part~I \cite{DuZouPartI}.

        The circulation is exactly $\Gamma$ by \eqref{7-15}, and $\omega_{\varepsilon,t}\ge0$ since $f\ge0$. By \eqref{7-14}, the active inner region lies in $B_{R_*}$. With $\delta_\varepsilon(t)=\varepsilon/\alpha(Y,a)$, uniform interior conformal estimates give
        \[
        |Z_Y(-\mathrm ia+\delta_\varepsilon(t)X)-p_\varepsilon(t)|\le C_{\mathscr K}\varepsilon,\qquad |X|\le R_*,
        \]
        proving \eqref{7-59}. Smooth dependence of the conformal map and \eqref{7-57} also give $\sup_I|p_\varepsilon-p|\le C_{\mathscr K}\varepsilon$. Since the vorticity has total mass $\Gamma$, for every bounded Lipschitz function $\varphi$,
        \begin{equation}
            \sup_{t\in I}\left|\int_{\mathbb R^2}\varphi(p)\omega_{\varepsilon,t}(p)\,dp-\Gamma\varphi(p(t))\right|\le C_{\mathscr K}\Gamma\operatorname{Lip}(\varphi)\varepsilon.
            \label{7-65}
        \end{equation}
        To obtain this bound, subtract $\varphi(p(t))$ inside the integral and use the support and center estimates. This proves uniform concentration along the arc.

        The core geometry is controlled by \eqref{7-44}. On the fixed inner source disk, $U+v_\varepsilon\to U$ in $C^{2,\gamma'}$, uniformly in $t$. Since $U_r(1)=-\Gamma/(2\pi)\ne0$, its zero level near $\partial B_1$ is a unique $C^{2,\gamma'}$ radial perturbation of the unit circle. The strict signs of $U$ away from this circle and \eqref{7-14} exclude additional active components. Because $f$ is positive precisely on $(0,\infty)$, this zero level bounds the active vorticity region.

        Moreover,
        \[
        \frac{Z_Y(-\mathrm ia+\delta_\varepsilon(t)X)-p_\varepsilon(t)}{\varepsilon}=X+O(\varepsilon)
        \]
        in local spatial norms, uniformly in $t$, since $Z_Y'(-\mathrm ia)\delta_\varepsilon(t)=\varepsilon$. The inverse coordinate maps also converge to the identity. Hence the rescaled physical core boundaries converge to $\partial B_1$ in $C^{1,\gamma''}$ for every $0<\gamma''<\gamma'$, and, on every fixed inner disk,
        \[
        \sup_{t\in I}\sup_{|X|\le R}\left|\varepsilon^2\omega_{\varepsilon,t}(p_\varepsilon(t)+\varepsilon X)-f(U(X))\right|\longrightarrow0,\qquad R>0.
        \]
        Interior Schauder estimates for \eqref{7-58} give $\Psi_\varepsilon(t,\cdot)\in C_{\rm loc}^{4,\gamma'}(\Omega_\varepsilon(t))$ and $\omega_{\varepsilon,t}\in C_c^{2,\gamma'}(\Omega_\varepsilon(t))$ for fixed $\varepsilon,t$. The active-core boundary is $C^{2,\gamma'}$ by the nonvanishing normal derivative. No $C^\infty$ regularity is claimed under the present assumptions on $f$.

        Outside the moving cores, \eqref{7-9} identifies the reference field exactly with the point-vortex Green field. The fixed-source representation, $\phi_\varepsilon=O(\varepsilon)$, and convergence of the free surfaces and vortex centers give local convergence of the outer fields away from the point-vortex positions, after identification by the conformal maps.
    \end{proof}

    Where $L\ne0$, the reference and regularized waves may be compared at fixed speed, whereas the transverse correspondence \eqref{7-62} permits speed variation through a fold. Any prescribed finite collection of interior simple folds can be preserved with a common sufficiently small core-size threshold. The argument does not address higher-order degeneracies or yield a uniform positive core size along a noncompact global curve. Uniqueness is restricted to the prescribed normalized configuration neighborhoods, and the steady fold geometry alone implies neither preservation of the energy--momentum Hessian index nor dynamical stability.

    \begin{proof}[Proof of Theorem~\ref{thm1-3}]
        Lemmas~\ref{lem7-1} and \ref{lem7-3}, together with Proposition~\ref{prop7-2}, provide the normalized complementary solutions and their $C^2$ reduced-drift convergence. Proposition~\ref{prop7-4} constructs the regular embedded arc satisfying \eqref{7-57} and the transverse correspondence \eqref{7-62}; its physical reconstruction proves part~{\rm(i)}. The exact circulation, support estimate \eqref{7-59}, and concentration estimate \eqref{7-65} prove part~{\rm(ii)}. Finally, the nondegeneracy \eqref{5-11} and the $C^2$ comparison \eqref{7-47} yield the unique nearby simple fold, while \eqref{7-61}--\eqref{7-64} establish its quadratic turning and fixed-speed solution count, proving part~{\rm(iii)}.
    \end{proof}

    Theorem~\ref{thm1-3} extends the regularization of individual point-vortex waves in Part~I to compact regular solution arcs, including their specified simple folds. The resulting smooth-core waves have the prescribed circulation and admit uniform control of the coupled free boundaries and vortex centers.

    \bigskip
\noindent{\bf Declaration of generative AI and AI-assisted technologies in the manuscript preparation process:}
	The authors developed the mathematical framework and proof strategies. ChatGPT (OpenAI) was used to assist with the preparation of schematic figures and the revision of the mathematical exposition and language of the manuscript. The authors reviewed and checked all AI-assisted material and take full responsibility for the content of this work.
	
    \bigskip
	
	\noindent{\bf Conflict of interest statement:} On behalf of all authors, the corresponding author states that there is no conflict of interest.
	
	\bigskip
	
	\noindent{\bf Data availability statement:} Our manuscript has no associated data.
	
	\bigskip
	
	\bibliographystyle{plain}
	\bibliography{reference}

\end{document}